\documentclass[a4paper,11pt, DIV=12]{scrartcl}
\usepackage{amsmath,amssymb,amsfonts,amsthm,exscale,calc}
\usepackage{latexsym,textcomp}
\usepackage{a4wide}
\usepackage{dsfont}
\usepackage[german, english]{babel}
\usepackage{hyperref}
\setkomafont{sectioning}{\bfseries}
\usepackage{enumerate}
\usepackage[T1]{fontenc}
\usepackage[utf8]{inputenc}

\newtheorem{lemma}{Lemma}[section]
\newtheorem{theorem}[lemma]{Theorem}

\newtheorem{proposition}[lemma]{Proposition}
\newtheorem{corollary}[lemma]{Corollary}

\newtheorem{problem}{Problem}

\theoremstyle{definition}
\newtheorem{definition}[lemma]{Definition}
\newtheorem{example}[lemma]{Example}
\newtheorem*{remark}{Remark}

\newtheorem*{convention}{Convention}

\numberwithin{equation}{section}
\newcommand{\comment}[1]{}

\newcommand{\cA}{{\mathcal A}}

\newcommand{\cC}{{\mathcal C}}

\newcommand{\cS}{{\mathcal S}}
\newcommand{\cF}{{\mathcal F}}
\newcommand{\cH}{{\mathcal H}}

\newcommand{\cM}{{\mathcal M}}
\newcommand{\cQ}{{\mathcal Q}}
\newcommand{\cD}{{\mathcal D}}
\newcommand{\Dm}{\mathcal{D}_E}
\newcommand{\Qm}{\mathcal{Q}_E}
\newcommand{\Qmc}{\mathcal{Q}^c_E}
\newcommand{\Wm}{W_{\mathrm{min}}}

\newcommand{\Z}{{\mathbb Z}}
\newcommand{\R}{{\mathbb R}}
\newcommand{\IC}{{\mathbb C}}

\newcommand{\N}{{\mathbb N}}

\newcommand{\Hmv}{{\mathcal H}_{\mu,V}}
\newcommand{\Mm}{{\mathcal M}_{\mu,\Phi,W}}

\newcommand{\Deg}{{\mathrm {Deg}}}

\newcommand{\as}[1]{\langle #1\rangle}
\newcommand{\av}[1]{\left\vert #1\right\vert}

\newcommand{\Hm}[1]{\leavevmode{\marginpar{\tiny%
$\hbox to 0mm{\hspace*{-0.5mm}$\leftarrow$\hss}%
\vcenter{\vrule depth 0.1mm height 0.1mm width \the\marginparwidth}%
\hbox to 0mm{\hss$\rightarrow$\hspace*{-0.5mm}}$\\\relax\raggedright
#1}}}

\begin{document}

\title{On the existence and uniqueness of self-adjoint realizations of discrete (magnetic) Schrödinger operators}

\author{Marcel Schmidt}

\maketitle

\begin{abstract}
In this expository paper we  answer  two fundamental questions concerning discrete magnetic Schrödinger operator associated with weighted graphs. We discuss when formal expressions of such operators give rise to self-adjoint operators, i.e., when they have self-adjoint restrictions. If such self-adjoint restrictions exist, we explore when they are unique. 
\end{abstract}




\section{Introduction}

Discrete Laplacians and discrete magnetic Schrödinger operators feature in many different areas of mathematics. They are used in combinatorics and computer science, appear as discretizations of (pseudo-)differential operators on Riemannian manifolds, serve as   toy models for Hamiltonians in mathematical physics and play an important role in the study of random walks - just to name a few.  Even though discrete operators are used for very different means, their basic  structure is always the same. Given a graph $G = (X,E)$ consisting of vertices $X$ and edges $E$ one defines the discrete Laplacian $\Delta$ acting on functions on $X$ (up to a sign) by 
$$\Delta f(x) = \sum_{(x,y) \in E} (f(x) - f(y)).$$
Since this operator itself is often not flexible enough, one then introduces various weights to decorate the discrete Laplacian to obtain a discrete  magnetic Schrödinger operator $\cM$ of the form
$$\cM f(x) = \frac{1}{\mu(x)} \sum_{(x,y) \in E} b(x,y) (f(x) - e^{i \theta(x,y)}f(y)) + V(x)f(x).$$
Here,  $\mu$ is a weight on the vertices, $b$ is an edge weight, $e^{i \theta(x,y)}$ is a magnetic field and $V$ is a potential. The operator with  vanishing $V$ and $\theta$ is called weighted discrete Laplacian. For other appropriate choices of weights these operators include the adjacency operator and the transition operator of Markov chains.

From the viewpoint of applications a very important case is when the formal expression for $\cM$ gives rise to a self-adjoint operator on $\ell^2(X,\mu)$. Of course, this puts some symmetry restrictions on the weights; the edge weight $b$ has to be symmetric, $\theta$ has to be anti-symmetric and $V$ needs to be real-valued. In fact, if $X$ is finite, any self-adjoint matrix is a discrete magnetic Schrödinger operator of this form and if $X$ is infinite, the same holds true for self-adjoint finite range operators on $\ell^2(X,\mu)$.

This text discusses the two most basic questions about discrete magnetic Schrödinger operators.
\begin{enumerate} 
 \item Given $\mu,b,\theta,V$ does there exist a self-adjoint restriction of $\cM$ on $\ell^2(X,\mu)$?
 \item If there exists a self-adjoint restriction of $\cM$ on $\ell^2(X,\mu)$, is it unique?
\end{enumerate}
It is our goal to provide  expository answers to both questions, which are as comprehensive as possible. We present the best known answer to the first question with two restrictions on the considered operators. We  only treat the existence of lower semi-bounded restrictions of $\cM$ whose domain of the associated quadratic form  contains the finitely supported functions; we call such operators  {\em realizations} of $\cM$. While the first assumption is technical, the second is natural since  on discrete spaces functions with finite support can be seen as test functions.  For the second question there is vast amount of results  and we had to make some choices. Here we focus on current results that  involve the  geometry of the weighted graph and omit some more abstract criteria.  

For the sake of being expository we repeat several known arguments and summarize some basic definitions and properties of quadratic forms in Appendix~\ref{appendix:quadratic forms}. However, we would like to stress that everything that we call a theorem in this text goes in one form or another beyond what can be found in the literature.

The weighted graphs that we treat are not assumed to be locally finite.  This has several reasons which are based on applications; here we name three. First of all, the weighted discrete Laplacian of any given graph can be approximated (in the strong resolvent sense) by  bounded weighted discrete Laplacians. If the given graph is connected, the graphs of the approximating operators are not locally finite. Secondly, approximations of long-range non-local pseudo-differential  operators by weighted discrete Laplacians often make use of graphs that are  not locally finite. For concrete examples  we refer to \cite{CTW,CKK}, where convergence to the generator of stable-like processes is studied. Thirdly, fractional powers of discrete Laplacians, the discrete analogue of generators of stable-like processes, tend to be weighted discrete Laplacians of not locally finite graphs.  

 Since there has recently been some interest in discrete magnetic Schrödinger operators on   functions taking values in vector bundles, see \cite{MT,GMT} and references therein, everything is formulated in this slightly more general setting, see Section~\ref{section:discrete magnetic schroedinger operators}.

 It is folklore that the existence (and uniqueness) of bounded realizations of the weighted discrete Laplacian is equivalent to the weighted vertex degree being bounded, see e.g. \cite{KL}. In Section~\ref{section:bounded realizations} we extend this result to  magnetic Schrödinger operators. Here, the difficulty lays in treating potentials without a fixed sign. If the weighted  vertex degree is unbounded, the weighted graph Laplacian is an unbounded operator and so answers to our questions naturally become more difficult.

 In the case when the underlying graph is locally finite, it is possible to obtain the existence of realizations of a magnetic Schrödinger operator by studying self-adjoint extensions of the corresponding minimal  magnetic Schrödinger operator acting on finitely supported functions, see Subsection~\ref{subsection:admissible endomorphisms and fc}. Since finitely supported functions are the test functions on discrete spaces, in principle this theory works exactly as for Schrödinger operators on open subsets of Euclidean space.
 
 The real challenge for answering the first question lies with graphs that are not locally finite. In this case discrete magnetic Schrödinger operators need not map finitely supported functions to $\ell^2(X,\mu)$. Therefore, the strategy of the locally finite case, namely  seeking for self-adjoint extensions of the minimal operator, does not work. A first systematic study of discrete Schrödinger operators with nonnegative potentials on not necessarily locally finite graphs via Dirichlet forms  is contained in \cite{KL}. Discrete magnetic Schrödinger operators on such graphs with potentials without a fixed sign but with small negative part were then treated in \cite{GKS,GMT} by means of perturbation theory of quadratic forms. The results on the existence of realizations of \cite{KL,GKS,GMT} are presented in Subsection~\ref{subsection:Nonnegative endomorphisms and small perturbations}. For discrete (magnetic) Schrödinger operators whose negative part of the potential cannot be treated through perturbation theory the existence of realizations was still open. We settle this issue and prove existence for a very general class of so-called admissible potentials in Subsection~\ref{subsection:admissible endomorphisms and domination}. In the scalar case without magnetic field our results are optimal, see Subsection~\ref{subsection:summary and examples}. As a byproduct we obtain an entirely analytic proof for the closability of  associated quadratic forms on finitely supported functions, see Theorem~\ref{theorem:existence of realizations for magnetic operators} and Theorem~\ref{theorem:existence of realizations for scalar operators}, which was only previously obtained through probabilistic arguments, see \cite{GKS}.

Our second question on uniqueness of self-adjoint restrictions of discrete Laplacians and discrete magnetic Schrödinger operators has seen quite some attention in recent years, see \cite{TTV,TTVa,TTVIII, Gol,Gol2,GS,GKS,HKLW,HKMW,Jor,JP,KL2,KL,KS,Mas,Mil,Mil2,MT2,MT,Tor,Web,Woj2,Woj} among others.

There are different classes of operators in which uniqueness  can be studied. If the graph is locally finite,  we study essential self-adjointness of the minimal operator - the magnetic Schrödinger operator on finitely supported functions. This corresponds to uniqueness in the class of all self-adjoint operators. As discussed above, for general not locally finite graphs the minimal operator need not exist and so the concept of essential self-adjointness of the minimal operator makes no sense. In this case one can still ask for uniqueness of realizations. By definition (see above) this corresponds to uniqueness in the class of semi-bounded self-adjoint operators.  For both cases we present two criteria (and corollaries) involving the geometry of the weighted graph and properties of the underlying measure. 

In Subsection~\ref{subsection:a measure space criterion} we show that if the measure is large enough along infinite paths, the discussed uniqueness properties hold, see Theorem~\ref{theorem:measure space uniqueness} and Corollary~\ref{corollary:measure space criterion}. Our presentation is based on   results for Schrödinger operators with nonnegative potentials from \cite{KL} and for scalar magnetic Schrödinger operators from \cite{Gol2}. Subsection~\ref{subsection:a metric space criterion} is devoted to uniqueness criteria that involve the metric geometry with respect to an intrinsic metric. We prove uniqueness if there is no metric boundary or if the potential is growing fast enough towards the boundary, see Theorem~\ref{theorem:metric space uniqueness} and Corollary~\ref{corollary:completeness}. The presented proofs are based on \cite{MT}, which treats path metrics only. We allow arbitrary intrinsic metrics. As one possible application for this more general situation we discuss metrics coming from embeddings of the graph into Euclidean space or into complete Riemannian manifolds.

In the scalar case with vanishing magnetic field and nonnegative potential an important class of realizations are nonnegative ones whose associated quadratic form is a Dirichlet form.   They correspond to Markov processes on the underlying space through the Feynman-Kac formula and are therefore called Markovian realizations. Section~\ref{section:Markovian realizations} focuses on on them. In Subsection~\ref{subsection:structure of Markovian realizations} we show that the quadratic form of a  Markovian realization lies between a minimal Dirichlet form, the one with ``Dirichlet boundary conditions at infinity'', and a maximal Dirichlet form, the one with ``Neumann boundary conditions at infinity''. In this generality our result is new,  for locally finite graphs it is contained in \cite{HKLW}. Our discussion is based on methods developed in \cite{Schmi3}.

In Subsection~\ref{subsection:markov uniqueness} we study Markov uniqueness, i.e., uniqueness of Markovian realizations. By the previous discussion this is equivalent to the maximal and the minimal Dirichlet form being equal. Of course, the results of Section~\ref{section:uniqueness} can always be applied, but in general their assumptions are stronger than what is actually needed for Markov uniqueness. We show that smallness of ``the boundary", in the sense that its capacity vanishes, is equivalent to Markov uniqueness, see Theorem~\ref{theorem:capacity criterion boundary}. Here,  ``the boundary" can have two forms; it can be either a boundary with respect to a compactification or, in the locally finite case,  a metric boundary with respect to a path metric. For the abstract boundary coming from a compactification the characterization of Markov uniqueness is taken from \cite{Schmi3}, while the path metric case is treated in \cite{HKMW}. In Theorem~\ref{theorem:markov uniqueness 2} we characterize Markov uniqueness in the case when the underlying measure is finite. In particular, we relate Markov uniqueness for all finite measures to uniqueness of arbitrary realizations and the existence of intrinsic metrics with finite distance balls with respect to particular finite measures. This part is taken from \cite{Puc}.

Even though we try to give comprehensive answers to our two questions, there are several problems that remain unresolved. We collect the - in our view - most important ones in Section~\ref{section:open problems} and comment on their relevance and expected outcome. 

For the convenience of the reader (and to fix notation) we recall some basic properties of quadratic forms in Appendix~\ref{appendix:quadratic forms}.

{\em Acknowledgements. }  The author is indebted to Matthias Keller for asking him several of the questions answered in this article and for encouraging him to write it. Moreover, he expresses his gratitude to Simon Puchert for sharing new ideas on Markov uniqueness and for  allowing him to use some of the results of the Master's thesis \cite{Puc}.  Furthermore, several discussions with   Melchior Wirth on domination of quadratic forms and with  Daniel Lenz and Radoslaw Wojciechowski on  related topics resulted in various insights that improved this article.

\section{Discrete (magnetic) Schrödinger operators  - the setup and basic properties} \label{section:discrete magnetic schroedinger operators}
 
In this section we introduce weighted graphs, discrete magnetic Schrödinger operators and Schrödinger forms.  We discuss their connections via Green's formulae and Kato's inequality. The contents of this section are all known in one form or another. 

\subsection{Graphs and discrete Schrödinger operators} \label{subsection:graph laplacians} In this subsection we introduce graphs and the associated (formal) discrete Schrödinger operators.

A {\em (symmetric) weighted graph} $(X,b)$ consists of a countable set of {\em vertices} $X \neq \emptyset$ and an {\em edge weight function} $b:X \times X \to [0,\infty)$ with the following properties
\begin{itemize}
 \item[(b0)] $b(x,x) = 0$ for all $x \in X$,
 \item[(b1)] $b(x,y) = b(y,x)$ for all $x,y \in X$,
 \item[(b2)] $\sum_{z \in X} b(x,z) < \infty$ for all $x \in X$.
\end{itemize}
 The {\em vertex degree} of a vertex $x \in X$ is
$$\deg(x) := \sum_{z \in X} b(x,z).$$
We say that two vertices $x,y \in X$ are {\em connected by an edge} if $b(x,y) >0$. In this case, we write $x \sim y$. A {\em path} is a finite or infinite sequence of vertices $(x_1,x_2,\ldots)$  such that $x_j \sim x_{j + 1}$ for all $j  = 1,2,\ldots$

For $U \subseteq X$ we define the {\em combinatorial neighborhood of} $U$ by
$$N(U) := U \cup \{x \in X \mid \text{there ex. } y \in U \text{ with } y \sim x\}.$$
A weighted graph $(X,b)$ is called {\em locally finite}, if for any $x \in X$ the set $N(\{x\})$ is finite, i.e., if all vertices have only finitely many neighbors. A vertex $x \in X$ is {\em isolated} if $N(\{x\}) = \emptyset.$
%

We equip the vertex set $X$ with the discrete topology and write $C(X)$ for the space of all complex-valued functions on $X$ (the space of continuous complex-valued functions with respect to the discrete topology). Its subspace of functions with finite support (the functions of compact support for the discrete topology) is denoted by $C_c(X)$.  Moreover, we write $\ell^\infty(X)$ for the space of bounded complex-valued functions on $X$. 

Given two real-valued functions $f,g \in C(X)$ we denote by $f \wedge g:= \mathrm{min} \{f,g\}$ their minimum and by $f \vee g := \mathrm{max} \{f,g\}$ their maximum. Moreover, we let $f_+ = f \vee 0$ and $f_- = (-f) \vee 0$ the {\em positive part} and the {\em negative part} of $f$, respectively.

For a set $K \subseteq X$ we let $1_K:X \to \{0,1\}$ be the {\em indicator function of $K$}, i.e., $1_K(x) = 1$, if $x \in K$, and $1_K(x) = 0$, else. The indicator function of the singleton set $\{x\}$ is denoted by $\delta_x$.

A strictly positive function $\mu:X \to (0,\infty)$ is called {\em weight}. Any weight $\mu$ induces a Radon measure of full support on all subsets of $X$ by letting
$$\mu(A) := \sum_{x \in A} \mu(x), \quad A \subseteq X.$$
In what follows we shall not distinguish between the weight and the measure it induces. 

If we equip $C(X)$ with the topology of pointwise convergence,   its continuous dual space $C(X)'$ is isomorphic to $C_c(X)$. Given a weight $\mu$, the dual pairing $(\cdot,\cdot):C_c(X) \times C(X) \to \IC$,
$$(\varphi,f) := \sum_{x \in X} \overline{\varphi(x)} f(x) \mu(x)$$
induces an anti-linear isomorphism via $C_c(X) \to C(X)', \varphi \mapsto (\varphi,\cdot)$. Note that $(\cdot,\cdot)$ depends on the choice of $\mu$.

The space of square integrable functions with respect to a weight $\mu$ is denoted by 
$$\ell^2(X,\mu) := \{f \in C(X) \mid \sum_{x \in X} |f(x)|^2 \mu(x) < \infty\}.$$
It is a Hilbert space when equipped with the inner product 
$$\as{f,g}_2 := \sum_{x \in X} \overline{f(x)}g(x) \mu(x).$$
The norm induced by $\as{\cdot,\cdot}_2$ is denoted by $\|\cdot\|_2$. Note that for $\varphi \in C_c(X)$ and $f \in \ell^2(X,\mu)$ we have $\as{\varphi,f}_2 = (\varphi,f)$. As for the inner product $\as{\cdot,\cdot}_2$, throughout the text all inner products are assumed to be linear in the second argument and anti-linear in the first. 

To a graph $(X,b)$, a weight $\mu$ and a real-valued $V \in C(X)$ we associate the {\em formal discrete Schrödinger operator} $\cH = \cH_{\mu,V}:\cF  \to C(X)$, where
$$\cF := \{f \in C(X) \mid \sum_{y \in X} b(x,y) |f(y)| < \infty \text{ for all }x \in X\},$$
on which $\cH$ acts by
$$\cH f(x) := \frac{1}{\mu(x)} \sum_{y \in X} b(x,y) (f(x) - f(y)) + V(x) f(x), \quad x \in X.$$
The definition of $\cF$ ensures that this sum  converges absolutely. Bounded functions are always contained in $\cF$ and we have $\cF = C(X)$ if and only if the graph $(X,b)$ is locally finite. Sometimes, a real-valued function $V$ as  above is simply called {\em potential}.

%

\subsection{Discrete magnetic Schrödinger operators} \label{subsection:magnetic laplacians} In this subsection we introduce  discrete magnetic Schrödinger operators. They are  vector-valued versions of discrete Schrödinger operators with an additional magnetic interaction term.

A {\em Hermitian vector bundle} over a countable discrete base space $X$ is a collection $E = 
(E_x)_{x \in X}$ of isomorphic complex Hilbert spaces of dimension greater or equal to $1$. We denote the  inner product on $E_x$ by $\as{\cdot,\cdot}_x$ and the induced norm by $|\cdot|_x$. Often it is clear from the context in which of the $E_x$ we consider the inner product and the norm. If this is the case, we drop the subscript $x$.  We write
$$\Gamma(X; E) := \{f:X \to \bigsqcup_{x \in X} E_x \mid f(x) \in E_x\}$$
for the space of all {\em sections} of the bundle $E$. The subspace of all {\em sections with finite support} is denoted by $\Gamma_c(X; E)$.  For a given weight $\mu$ on $X$ we define $(\cdot,\cdot)_E:\Gamma_c(X; E) \times  \Gamma(X; E)  \to\IC$  by
$$(\varphi,f)_E = \sum_{x \in X} \as{\varphi(x),f(x)}\mu(x).$$
%
%
 Scalar valued functions naturally act on the space of all sections by pointwise multiplication so that $\Gamma(X; E)$ is a module over $C(X)$. More precisely, for $\varphi \in C(X)$ and $f \in \Gamma(X; E)$ we define $\varphi  f \in \Gamma(X; E)$ by  $(\varphi f)(x)  := \varphi(x)f(x)$.  
 
 For $f \in \Gamma(X; E)$ we let $|f| \in C(X)$ be given by $|f|(x) = |f(x)|$ and $\mathrm{sgn}f \in \Gamma(X; E)$ by
$$\mathrm{sgn} f (x) =  \begin{cases}
                     \frac{1}{|f(x)|} f(x) &\text{if } f(x) \neq 0\\
                     0 &\text{if } f(x) = 0
                    \end{cases}.
$$
With this notation any $f \in \Gamma(X; E)$ can be written as $f = |f| \mathrm{sgn} f$.

Given a weight $\mu$ on $X$ the corresponding space of {\em $\ell^2$-sections} is
$$\ell^2(X,\mu; E) := \{f \in \Gamma(X; E) \mid \sum_{x \in X} |f(x)|^2 \mu(x) < \infty\}.$$
It is a Hilbert space when equipped with the inner product
$$\as{f,g}_{2;\,E} := \sum_{x \in X} \as{f(x),g(x)} \mu(x).$$
The norm induced by $\as{\cdot,\cdot}_{2;\, E}$ is denoted by $\|\cdot\|_{2;\, E}$. 
\begin{remark}
 The identity $\as{\varphi,f}_{2;\, E} = (\varphi,f)_E$ holds for $\varphi \in \Gamma_c(X; E)$ and $f \in \ell^2(X,\mu; E)$ and a similar statement is valid for $\as{\cdot, \cdot}_2$ and $(\cdot,\cdot)$ in the scalar case, cf. Subsection~\ref{subsection:graph laplacians}. Because of this it is tempting to abuse notation and denote both pairings by the same symbol. In this text we introduced different notation for a reason. Often we can easily perform pointwise computations that lead to an identity or inequality on the level of $(\varphi,f)_E$ for pairs $\varphi \in \Gamma_c(X; E), f\in S$, where $S$ is some subspace of $\Gamma(X; E)$. However, in applications, we are truly interested in these identities on the level of $\as{g,f}_{2;\, E}$ for pairs $g \in G$ and $f \in F$, where $F$ and $G$ are subspaces of $\ell^2(X,\mu; E)$ with  $\Gamma_c(X; E) \subseteq F,G$. In order to lift results from $(\cdot,\cdot)_E$ to $\as{\cdot,\cdot}_{2;\, E}$ usually additional arguments are needed. In order to make such arguments more transparent, we distinguish the pairings.  

\end{remark}

A {\em unitary connection} on a Hermitian bundle $E = (E_x)_{x \in X}$ over $X$ is a family $\Phi = (\Phi_{x,y})_{x,y \in X}$ of unitary maps $\Phi_{x,y} :E_y \to E_x$ such that $\Phi_{x,y} = \Phi_{y,x}^{-1}$. A {\em bundle endomorphism} on  a Hermitian vector bundle $E = (E_x)_{x \in X}$ over $X$ is a collection $W = (W_x)_{x \in X}$ of bounded linear maps $W_x: E_x \to E_x$. It is called {\em self-adjoint}, if for all $x \in X$ the operator $W_x$ is self-adjoint. 

Let $(X,b)$ a graph and $E$ a Hermitian vector bundle over $X$. To a weight $\mu$ on $X$, a self-adjoint bundle endomorphism $W$ and a unitary connection $\Phi$ on $E$ we associate the {\em formal magnetic Schrödinger operator} $\cM = \Mm:\cF_E \to \Gamma(X; E)$, where 
$$\cF_E  := \{f \in  \Gamma(X; E) \mid \sum_{y \in X} b(x,y) |f(y)| < \infty \text{ for all } x \in X\}, $$
on which $\cM$  acts by
$$\cM f(x) :=  \frac{1}{\mu(x)} \sum_{y \in X} b(x,y) (f(x) - \Phi_{x,y}f(y)) + W_x f(x), \quad x \in X.$$
Note that since $|\Phi_{x,y}f(y)|_x = |f(y)|_y$,  for each $f \in \cF_E$ the sum in the definition of $\cM f(x)$ converges absolutely in the Hilbert space $E_x$. 

\begin{remark}[Scalar magnetic Schrödinger operators]
 If $E_x = \IC$ for all $x \in X$, then $\Gamma(X; E)$ can be identified with $C(X)$. In this case, a self-adjoint bundle endomorphism $W$ of $E$ acts on $\Gamma(X; E)$ as multiplication by a potential, i.e., there exists a real-valued $V \in C(X)$ such that $W_x f(x) = V(x)f(x)$ for $x \in X$.  Moreover, any connection $\Phi$ of $E$ is parametrized by a function 
 $$\theta:X \times X \to \R / 2\pi \Z,$$
 with $\theta(x,y) = - \theta(x,y)$ for all $x,y \in X$, via the identity
 $$\Phi_{x,y} z = e^{i \theta(x,y)} z, \quad z \in \IC.$$
 The corresponding magnetic Schrödinger operator is called  {\em scalar magnetic Schrödinger operator} and is denoted by $\cM_{\mu,\theta,V}$ . If $\theta = 0$, then $\Phi_{x,y} = \mathrm{Id}$ and $\cM_{\mu,0,V} =    \Hmv$. Therefore, the discrete Schrödinger operators discussed in Subsection~\ref{subsection:graph laplacians} are a special case of discrete magnetic Schrödinger operators.
\end{remark}

\subsection{Standing assumptions and notation}

Unless otherwise specified we always assume the following.

\begin{itemize}
 \item $(X,b)$ is a weighted graph.
 \item  $\mu$ is a weight on $X$ and $V \in C(X)$ is real-valued. 
 \item $\cH = \cH_{\mu, V}$ is the associated formal discrete Schrödinger operator.
 \item $E = (E_x)_{x \in X}$ is a Hermitian vector bundle over $X$, $\Phi = (\Phi_{x,y})_{x,y \in X}$ is a unitary connection  on $E$  and $W$  is a self-adjoint bundle endomorphism on $E$. 
 \item $\cM = \cM_{\mu, \Phi,W }$ is the associated formal discrete Schrödinger operator.
\end{itemize}

Whenever no confusion can arise we suppress the dependence of objects from these data in our notation. Often we deal with complex-valued functions and sections ($E$-valued functions) at the same time. Therefore, a subscript $E$ indicates spaces of $E$-valued functions or functionals on $E$-valued functions.

\subsection{Green's formula and  domination} In this subsection we discuss  Green's formula for the magnetic Schrödinger operator $\cM$, which is in principle an integration-by-parts formula. Moreover, we give a version of Kato's inequality on the level of the pairing $(\cdot,\cdot)_E$.

The following integration-by-parts formula is an extension of \cite[Lemma~4.7]{HK} to magnetic Schrödinger operators.

\begin{lemma}[Green's formula] \label{lemma:Green's formula general}
For all  $f\in  \cF_E$ and $\varphi \in \Gamma_c(X; E)$ the following sums converge absolutely and satisfy the stated identities.
\begin{align*}
 (\varphi, \cM f)_E &= \sum_{x\in X} \as{\varphi(x),\cM f (x)} \mu(x) = \sum_{x\in X} \as{ \cM \varphi(x), f(x)} \mu(x) \\
&=\frac{1}{2} \sum_{x,y \in X} b(x,y) \as{\varphi(x) - \Phi_{x,y} \varphi(y),f(x) - \Phi_{x,y} f(y)} + \sum_{x \in X} \as{W_x \varphi(x),f(x)} \mu(x). 
\end{align*}
\end{lemma}
\begin{proof}
The first identity is just writing out the definition of the pairing $(\cdot,\cdot)_E$. Since the support of $\varphi$ is finite,  we have
$$\sum_{x,y \in X} b(x,y)\av{\as{\varphi(x),f(x)}} =  \sum_{x \in X} \av{\as{\varphi(x),f(x)}} \deg (x) < \infty.$$
Moreover, the finiteness of the support of $\varphi$, $f \in \cF_E$ and the Cauchy-Schwarz inequality yield
$$\sum_{x,y \in X} b(x,y)\lvert\as{\Phi_{x,y} \varphi(y),f(x)}\rvert \leq \sum_{y \in X} |\varphi(y)| \sum_{x \in X} b(x,y) |f(x)| < \infty.$$
With these integrability properties at hand, the convergence of the sums and the other identities follow from rearranging the sum
$$\frac{1}{2} \sum_{x,y \in X} b(x,y) \as{\varphi(x) - \Phi_{x,y} \varphi(y),f(x) - \Phi_{x,y} f(x)} + \sum_{x \in X} \as{W_x \varphi(x), f(x)} \mu(x)$$
with the help of Fubini's theorem.
\end{proof}
 For a bundle endomorphism $W$ and a potential $V$ we write $W \geq V$ if for all sections $f \in \Gamma(X; E)$ and $x \in X$ the inequality $\as{W_xf(x),f(x)} \geq V(x) |f(x)|^2$ holds.  The following lemma provides some form of Kato's inequality for the magnetic operator $\cM = \Mm$ and the discrete Schrödinger operator $\cH = \Hmv$ when $W \geq V$.  The presented proof is taken from \cite{LSW1}.
\begin{lemma}[Kato's inequality]\label{lemma:Katos inequality}
 Let $W \geq V$. If  $f \in \cF_E$ and $\varphi \in \Gamma_c(X; E)$ with $\as{f(x),\varphi(x)} = |f(x)| |\varphi(x)|$ for all $x \in X$, then
 $$\mathrm{Re}\, (\varphi,\cM f)_E \geq  (|\varphi|,\cH |f|).$$
\end{lemma}
\begin{proof}
It follows from the definitions that $f \in \cF_E$ implies $|f| \in \cF$. We have $\as{f(x),\varphi(x)} = |f(x)| |\varphi(x)|$ by assumption. Moreover, the Cauchy-Schwarz inequality and $|\Phi_{x,y}f(y)|_x = |f(y)|_y$ yields 
$$-\mathrm{Re} \as{\Phi_{x,y} f(y),\varphi(x)} \geq - |f(y)| |\varphi(x)|.$$
Summing up these inequalities, multiplying them by $b(x,y)$ and summing over $x,y \in X$ shows the desired inequality for the operators $\cM_{\mu,\Phi,0}$ and $\cH_{\mu,0}$.

It remains to prove the inequality $\mathrm{Re}\, \as{W_xf(x), \varphi(x)} \geq  V(x) |f(x)||\varphi(x)|.$ Without loss of generality we can assume $f(x) \neq 0 \neq \varphi(x)$. Since 
$$|f(x)| |\varphi(x)| = \as{f(x),\varphi(x)} = |f(x)| |\varphi(x)| \as{\mathrm{sgn} f(x),\mathrm{sgn} \varphi(x)},$$
we have $\mathrm{sgn} f(x) = \mathrm{sgn} \varphi(x)$. Using $W \geq V$ and $|\mathrm{sgn} f(x)| = 1$, we obtain 
$$\as{W_xf(x), \varphi(x)} = |f(x)| |\varphi(x)| \as{W_x \mathrm{sgn}f(x),\mathrm{sgn} f(x)} \geq |f(x)| |\varphi(x)| V(x).$$   
This finishes the proof.
\end{proof}

As a corollary we obtain that solutions to the eigenvalue equation with respect to $\cM$ yield nonnegative subsolutions to the eigenvalue   equation with respect to $\cH$. 

\begin{corollary}
 Let $f \in \cF_E$ and $\lambda \in \R$  with $\cM f = \lambda f$. Then $\cH |f| \leq \lambda |f|$.
\end{corollary}
\begin{proof}
For $z \in X$ let $\xi \in E_z$ with $|\xi| = 1$ and consider the section $\varphi \in \Gamma_c(X; E)$ that is given by
$$\varphi(x) = \begin{cases}
                \frac{1}{\mu(z) |f(z)|}f(z) &\text{if } x = z \text{ and } f(z) \neq 0\\
                \frac{1}{\mu(z)} \xi &\text{if } x = z \text{ and } f(z) = 0\\
                0 &\text{else}
               \end{cases}.
 $$
 It satisfies $|\varphi(z)| = \mu(z)^{-1}$ and $\varphi(x) = 0$ for $x \neq z$. Therefore, $\as{f(x),\varphi(x)} = 0 = |\varphi(x)||f(x)|$ if $x \neq z$ and $\as{\varphi(z),f(z)} = \mu(z)^{-1} |f(z)| = |\varphi(z)| |f(z)|$.  From Kato's inequality we infer 
\begin{align*}
 \lambda |f(z)| = (\varphi,\lambda f)_E = (\varphi,\cM f)_E \geq (|\varphi|,\cH |f|) =\cH |f|(z).
\end{align*}
This finishes the proof.
\end{proof}

\subsection{Schrödinger forms and a ground state transform} \label{subsection:Schrödinger forms}

In this subsection we introduce Schrödinger forms, which are related to magnetic Schrödinger operators by Green's formula. One possible domain for such forms are the functions of finite support another choice are functions of finite energy.

If $A$ is bounded and self-adjoint, the operator $|A|$ is defined by means of the functional calculus. Moreover, we let $A_+ = \frac{1}{2} (|A| + A)$ and $A_- =  \frac{1}{2}(|A| - A)$. Then $|A|,A_+,A_-$ are nonnegative self-adjoint operators and $A = A_+ - A_-$. 

The space of {\em sections of  finite magnetic energy} $\Dm = \cD_{\Phi,W;\, E}$ is defined by
$$\Dm := \{f \in \Gamma(X; E) \mid \sum_{x,y \in X}b(x,y)|f(x) - \Phi_{x,y} f(y)|^2 + \sum_{x \in X} \as{|W_x|f(x),f(x)}\mu(x) < \infty\}.$$
On it we introduce the sesquilinear {\em magnetic Schrödinger form} $\Qm = \mathcal{Q}_{\Phi,W;\, E}:\cD_E \times \cD_E \to \IC$, which acts by
$$\Qm(f,g) := \frac{1}{2} \sum_{x,y \in X} b(x,y) \as{f(x) - \Phi_{x,y} f(y),g(x) - \Phi_{x,y} g(y)} + \sum_{x \in X} \as{W_x f(x),g(x)} \mu(x).$$
It follows from the summability condition (b2) of the graph $(X,b)$ that $\Gamma_c(X; E) \subseteq \Dm.$ We denote the restriction of $\Qm$ to $\Gamma_c(X; E)$ by $\Qmc = \mathcal{Q}^c_{\Phi,W;\, E}$. The next lemma is a consequence of Green's formula. It relates $\Qm$ and $\Qmc$ to the operator $\cM$.

\begin{lemma}\label{lemma:greens formula forms} The inclusion $\Dm \subseteq \cF_E$ holds and for all $\varphi \in \Gamma_c(X; E)$ and $f \in \Dm$ we have $\Qm (\varphi,f) = (\varphi,\cM f)_E$. In particular, if $f \in \Gamma_c(X; E)$, then $\Qm^c (\varphi,f) = (\varphi,\cM f)_E$.
\end{lemma}
\begin{proof}
 Because of Lemma~\ref{lemma:Green's formula general} it suffices to prove the inclusion $\Dm \subseteq \cF_E$. The identity $|\Phi_{x,y} f(y)| = |f(y)|$ and the Cauchy-Schwarz inequality yield
 \begin{align*}
  \sum_{y \in X} b(x,y) |f(y)| &\leq \sum_{y \in X} b(x,y) |f(x) - \Phi_{x,y} f(y)| + |f(x)| \deg (x)\\
  &\leq \deg(x)^{1/2}  \left(\sum_{y \in X} b(x,y) |f(x) - \Phi_{x,y} f(y)|^2\right)^{1/2} + |f(x)| \deg (x).
 \end{align*}
This finishes the proof.
\end{proof}

  As remarked in Subsection~\ref{subsection:magnetic laplacians}, if $E = (\IC)_{x \in X}$, then $\Gamma(X; E) = C(X)$ and $\Gamma_c(X; E) = C_c(X)$. Moreover, the unitary connection $\Phi = (\Phi_{x,y})_{x,y}$ is parametrized by an antisymmetric function $\theta:X \times X \to \R / 2\pi \Z$ via $\Phi_{x,y} = e^{i \theta(x,y)}$ and the bundle endomorphism $W$ is  induced by a real valued function $V$. In this case, we drop the subscript $E$ and write $\cQ_{\theta, V}$ respectively $\cQ^c_{\theta, V}$ for the associated {\em scalar magnetic Schrödinger forms}.

  If, moreover,  $\Phi_{x,y} = \mathrm{Id}$, which is the same as $\theta = 0$,  we also drop the subscript $\theta$ and write $\cQ = \cQ_V := \cQ_{0,V} $ respectively $\cQ^c = \cQ^c_V := \cQ^c_{0, V}$ for the associated {\em Schrödinger forms}.  The corresponding space of {\em functions of finite energy} $\mathcal{D} = \mathcal{D}_{V}$ is 
  $$\cD = \{f \in C(X) \mid \sum_{x,y \in X}b(x,y)|f(x) - f(y)|^2 + \sum_{x \in X} |V(x)| |f(x)|^2 \mu(x) < \infty\},$$
  and $\cQ$ acts upon $\cD$ by
  $$\cQ(f,g) = \frac{1}{2}\sum_{x,y \in X}b(x,y)\overline{(f(x) - f(y))}(g(x)-g(y)) + \sum_{x \in X} V(x) \overline{f(x)} g(x) \mu(x).$$
  In this situation  the  formula discussed in Lemma~\ref{lemma:greens formula forms} reads
$$\cQ^c(f,g) = (f,\cH g), \quad f,g \in C_c(X).$$

We finish this section with the discussion of a ground state transform for the operator $\cH$ and the form $\cQ^c$. For a nonnegative function $f$ we set
$$b^{(f)}:X \times X \to [0,\infty),\, b^{(f)}(x,y) = f(x)f(y) b(x,y).$$
Clearly, $b^{(f)}(x,x) = 0$ and $b^{(f)}(x,y) = b^{(f)} (y,x)$ for all $x,y \in X$. Therefore,  $b^{(f)}$ satisfies assumptions (b0) and (b1). It satisfies (b2)  if and only if  $f \in \cF$. In this case, $(X,b^{(f)})$ is a weighted graph in the sense of Subsection~\ref{subsection:graph laplacians}. The Schrödinger form on $C_c(X)$ that is associated with the graph $(X,b^{(f)})$ and the potential $V = 0$ is denoted by $\cQ^{c, f}$. If $f$ is a subsolution to an eigenvalue equation, then $\cQ^c$ and $\cQ^{c, f}$ are related through the following lemma.   The presented proof is taken from \cite{HK}, which contains a version for nonnegative potentials.

\begin{lemma}[Ground state transform]\label{lemma:ground state transform}
 Let $f \in \cF$ nonnegative and $\lambda \in \R$ such that $\cH f \leq \lambda f$. For all $\varphi \in C_c(X)$ the inequality
 $$\cQ^c(f \varphi) \leq \cQ^{c,f}(\varphi) + \lambda \|f \varphi\|_2^2.$$
 holds.
\end{lemma}
 \begin{proof}
  Since $f$ is nonnegative and $f|\varphi|^2$ has compact support, the  inequality $\cH f \leq \lambda f$ yields
  $$ \lambda \|f \varphi \|^2_2 =  (f|\varphi|^2,\lambda f) \geq (f|\varphi|^2, \cH f).$$
  From Green's formula (Lemma~\ref{lemma:Green's formula general}) we infer
  \begin{align*}
   (f|\varphi|^2, \cH f) = \frac{1}{2} \sum_{x,y \in X}  b(x,y) (f(x) - f(y)) (f(x)|\varphi(x)|^2 - f(y) |\varphi(y)|^2) + \sum_{x \in X} V(x) |f\varphi| (x)^2.
  \end{align*}
  With this at hand, the desired inequality follows from the identity
  $$(f(x) - f(y)) (f(x)|\varphi(x)|^2 - f(y) |\varphi(y)|^2) = |f(x) \varphi(x) - f(y)\varphi(y)|^2 - f(x)f(y) |\varphi(x) - \varphi(y)|^2$$
  and the fact that all occurring sums converge due to $f \in \cF$ and the finite support of $f \varphi$.
 \end{proof}

\section{Existence of realizations of $\cH$ and $\cM$}\label{section:existence of realizations}

In this section we discuss the existence of self-adjoint restrictions of the formal operators $\cH$ and $\cM$  with two additional properties. The operators that we seek for are lower semi-bounded and they  reflect the discrete nature of the underlying space, i.e., we assume that test functions are cores in the domains of the associated quadratic forms. This is made precise in  Definition~\ref{definition:discrete magnetic Schrödinger operators}. In Subsection~\ref{subsection:Nonnegative endomorphisms and small perturbations} we prove that such realizations exist for nonnegative bundle endomorphisms and potentials, see Proposition~\ref{proposition:nonnegative endomorphisms}, and we use perturbation theory of quadratic forms to lift this result to  bundle endomorphisms and potentials with small negative part, see Proposition~\ref{proposition:small negative part}. In Subsection~\ref{subsection:admissible endomorphisms and domination} we prove the existence of realizations for scalar Schrödinger operators with admissible potentials, see  Theorem~\ref{theorem:existence of realizations for scalar operators}, and lift this result to magnetic Schrödinger operators with the help of domination of the associated quadratic forms, see Theorem~\ref{theorem:existence of realizations for magnetic operators}. Domination can be avoided when the underlying graph is locally finite or satisfies the weaker finiteness condition (FC); this is the content of Subsection~\ref{subsection:admissible endomorphisms and fc}.  Instead of discussing  literature and examples in the main text,  for the most part we postpone them to Subsection~\ref{subsection:summary and examples}.

\begin{definition}[Discrete magnetic Schrödinger operator] \label{definition:discrete magnetic Schrödinger operators} A self-adjoint operator $M$ on $\ell^2(X,\mu; E)$ is called a {\em realization of $\cM$} if the following assertions are satisfied.
\begin{enumerate}
 \item[(R1)] $M$ is semi-bounded from below.
 \item[(R2)] The domain of the quadratic form of $M$ contains $\Gamma_c(X; E)$.
 \item[(R3)] The inclusion $D(M) \subseteq \cF_E$ holds  and $Mf = \cM f$ for all $f \in D(M)$.
\end{enumerate}
 A self-adjoint operator is called {\em discrete magnetic Schrödinger operator} if it is a realization of a formal magnetic Schrödinger operator. 
\end{definition}

\begin{remark}
\begin{enumerate}
 \item Condition (R1) is technical and restricts the class of  operators to the ones that can be treated with quadratic form methods. In most applications the considered discrete magnetic Schrödinger operators satisfy this assumption. However, there is one instance where this is not the case. The adjacency operator that acts formally on functions in $\cF$ by
 $$A f (x) = \frac{1}{\mu(x)}\sum_{y \in X} b(x,y) f(y) $$
 may have self-adjoint restrictions without being lower semi-bounded, see \cite{Gol}. Note that $A$ equals the scalar magnetic Schrödinger operator $ \cM_{\mu,-\pi,-\mathrm{Deg}}$, where $\mathrm{Deg} = \mu^{-1} \deg$.
 
 \item Condition (R2) pays tribute to the fact that we deal with  discrete operators.  The space   $\Gamma_c(X; E)$ can be seen as the space of test functions (sections).  It is  tempting to replace (R2) by the stronger $\Gamma_c(X; E) \subseteq D(M)$. However, we shall  see below that for graphs which are not locally finite there  exist realizations $M$ of $\cM$ that satisfy (R2) but not  $\Gamma_c(X; E) \subseteq D(M).$ For the existence part see Theorem~\ref{theorem:existence of realizations for magnetic operators} and for the statement about the domain see Corollary~\ref{corollary:operator domain (FC)}.
 
 \item Since the operator $\cH$ is a special instance of the operator $\cM$, see Subsection~\ref{subsection:magnetic laplacians}, this also defines  realizations of $\cH$. They are called {\em discrete Schrödinger operators}. 
\end{enumerate}
\end{remark}

If there exist realizations of $\cM$, then they are necessarily restrictions of the following operator. 
\begin{definition}[Maximal restriction of $\cM$]
The domain of the {\em maximal restriction $M^{\mathrm{max}}$ of $\cM$}  is  $D(M^{\mathrm{max}}) = \{f \in  \cF_E \cap \ell^2(X,\mu; E) \mid \cM f \in \ell^2(X,\mu; E)\}$, on which it acts by $ M^{\mathrm{max}} f = \cM f$.  In the scalar case the maximal restriction of $\cH$ is denoted by $H^{\mathrm{max}}$.  
\end{definition}

\begin{remark}
 In general $\Gamma_c(X; E)$ does not belong to $D(M^{\mathrm{max}})$, see Subsection~\ref{subsection:admissible endomorphisms and fc}. Nevertheless, it can be checked that it is always densely defined. It is unclear whether $M^{\mathrm{max}}$ is a closed operator on $\ell^2(X,\mu; E)$.
\end{remark}

 Typically realizations of $\cM$ and $\cH$ come from closed lower semi-bounded quadratic forms on $\ell^2(X,\mu; E)$ that extend $\cQ^c_E$, respectively forms  on $\ell^2(X,\mu)$ that extend $\cQ^c$.   If this is the case, $\cQ^c_E$ and $\cQ^c$ need to be semi-bounded and closable. Their closures and associated operators deserve a name.

\begin{definition}[Closure of $\cQ^c_E$ and $\cQ^c$]
 If $\cQ^c_E = \cQ^c_{\Phi,W; E}$ is lower semi-bounded and closable on $\ell^2(X,\mu; E)$, we denote its closure by $Q^0_E= Q^0_{\mu,\Phi,W; E}$  and write $M^0 = M^0_{\mu,\Phi,W}$ for the associated self-adjoint operator. Likewise, if $\cQ^c = \cQ^c_V$ is lower semi-bounded and closable on $\ell^2(X,\mu)$, we denote its closure by $Q^0= Q^0_{\mu,V}$  and write $H^0 = H^0_{\mu,V}$ for the associated self-adjoint operator.
\end{definition}
\begin{remark}
In all the cases where we can prove the  closability of $\cQ^0_E$ it will turn out that $M^0$ is indeed a realization of $\cM$. We can think of $M^0$ as the realization of $\cM$ with Dirichlet boundary conditions at infinity.
\end{remark}
  The remaining course of this section is as follows. The simplest situation for proving the existence of realizations of $\cM$ is when the bundle endomorphism $W$ is nonnegative. In this case, the closability of $\cQ^c_E$ follows from Fatou's lemma and that $M^0$ is a realization of $\cM$ is a consequence of Green's formula.  For endomorphisms with form small negative part, see Definition~\ref{definition:admissible endomorphisms} below, the corresponding statements follow from standard perturbation theory. 
  
  If the negative part of the endomorphism is not form small but only such that $\cQ^c_E$ is lower semi-bounded on $\ell^2(X,\mu;E)$, then the situation is more delicate. In this case, unconditionally we can only treat scalar Schrödinger operators; if $V$ is such that $\cQ^c_V$ is lower semi-bounded on $\ell^2(X,\mu)$, we prove its closability and that $H^0$ is a realization of $\cH$. Given such a potential $V$ and a bundle endomorphism $W \geq V$, we then use domination of the associated resolvents to prove the closability of $\cQ^c_{E}$ and that $M^0$ is a realization of $\cM$. 
  
  If the underlying weighted graph satisfies the finiteness condition of Definition~\ref{definition:finiteness condition}, which depends on $b$ and $\mu$, no condition on $W$ except that $\cQ^c_E$ is lower semi-bounded is necessary. In this case,  the form $\cQ^c_E$ is induced by a symmetric lower semi-bounded operator and we employ Friedrichs' extension theorem to obtain a realization of $\cM$. Locally finite graphs   satisfy the finiteness condition for all choices of $\mu$.

\subsection{Nonnegative endomorphisms and small perturbations} \label{subsection:Nonnegative endomorphisms and small perturbations}

\begin{proposition}[Realization for nonnegative endomorphisms]\label{proposition:nonnegative endomorphisms}
If $W \geq 0$, then $\cQ^c_E$ is nonnegative and closable. Moreover, $Q^0_E$ is a restriction of $\cQ_E$  and $M^0$ is a realization of $\cM$.
\end{proposition}
\begin{proof}
Since $W \geq 0$, the form $\cQ^c_E$ is nonnegative and, in particular, lower semi-bounded. We prove that $\cQ^c_E$ is closable by showing that it is lower semicontinuous on its domain $\Gamma_c(X; E)$ with respect to $\ell^2(X,\mu; E)$-convergence. Let $\varphi, \varphi_n \in \Gamma_c(X; E)$ with $\varphi_n \to \varphi$ in $\ell^2(X,\mu; E)$. Since $W_x :E_x \to E_x$ and $\Phi_{x,y} :E_y \to E_x$ are continuous, for $x,y \in X$ we have $|\varphi_n (x) - \Phi_{x,y} \varphi_n(y)| \to   |\varphi (x) - \Phi_{x,y} \varphi(y)|$ and $\as{W_x \varphi_n(x),\varphi_n(x)} \to \as{W_x \varphi(x),\varphi(x)}$, as $n \to \infty$. With this and $W \geq 0$ at hand, we infer from Fatou's lemma 
$$\cQ^c_E(\varphi)  \leq \liminf_{n\to \infty} \cQ^c_E(\varphi_n).$$
This shows closability of $\cQ^c_E$.

Next we prove $D(Q^0_E) \subseteq \mathcal{D}_E$ and $Q^0_E = \cQ_E$ on $D(Q^0_E)$. To this end, let $f \in D(Q^0_E)$ and let $\varphi_n \in \Gamma_c(X; E)$ with $\varphi_n \to f$ with respect to the form norm. In particular, $\cQ^c_E(\varphi_m-\varphi_n)$ is arbitrarily small for large enough $n,m$ and $\cQ^c_E(\varphi_n) \to Q^0_E(f)$, as $n \to \infty$. Let $f_n := f-\varphi_n$.  Another application of Fatou's lemma as above shows
\begin{align*}
 \frac{1}{2} \sum_{x,y \in X}b(x,y)|f_n(x) - \Phi_{x,y} f_n(y)|^2 + \sum_{x \in X} \as{W_xf_n(x),f_n(x)}\mu(x) \leq \liminf_{m\to \infty} \cQ^c_E(\varphi_m - \varphi_n).
\end{align*}
Therefore, $f \in \cD_E$ and the left hand side of the previous inequality equals $\cQ_E(f - \varphi_n)$. Since the square roots of $\cQ^c_E$ and $\cQ_E$ are semi-norms that agree on $\Gamma_c(X; E)$, the above also implies
$$|\cQ_E(f)^{1/2} - \cQ_E^c(\varphi_n)^{1/2}| \leq \cQ_E(f - \varphi_n)^{1/2} \leq \liminf_{m \to \infty}\cQ_E^c(\varphi_m - \varphi_n)^{1/2}. $$
This shows 
$$Q^0_E(f) = \lim_{n \to \infty} \cQ^c_E(\varphi_n) = \cQ_E(f). $$
Since the off-diagonal values of $\cQ_E$ and $Q^0_E$ can be recovered from on-diagonal values via polarization, we obtain that $Q^0_E$ is a restriction of $\cQ_E$.

It remains to prove the claim about the operator $M^0$. Since $Q^0_E$ is nonnegative, $M^0$ is lower semi-bounded and therefore satisfies (R1). By definition we have $\Gamma_c(X; E) \subseteq D(Q^0_E)$ such that (R2) holds also.   Moreover, the general inclusion $D(M^0) \subseteq D(Q^0_E)$ and the already proven $D(Q^0_E) \subseteq \cD_E$ combined with  Lemma~\ref{lemma:greens formula forms} yield $D(M^0) \subseteq \cF_E$.  With all this at hand, another application of Lemma~\ref{lemma:greens formula forms}  shows that for $f \in D(M^0)$ and $\varphi \in \Gamma_c(X; E) \subseteq D(Q^0_E)$ we have
$$\langle\varphi,M^0 f\rangle_{2;\, E} = Q^0_E(\varphi, f) = \cQ_E(\varphi,f) = (\varphi, \cM f)_E.$$
This implies $M^0 f = \cM f$ and finishes the proof.
\end{proof}
As already mentioned at the beginning of this section, we can only treat special classes of  bundle endomorphism that are not nonnegative.  This is discussed next. 

To a self-adjoint bundle endomorphism $W$ we  associate the quadratic form $q_W$ on $\ell^2(X,\mu; E)$ with domain 
$$D(q_W) = \{f \in \ell^2(X,\mu; E) \mid \sum_{x \in X} \as{|W_x|f(x),f(x)} \mu(x) < \infty\},$$
on which it acts by
$$q_W(f) = \sum_{x \in X} \as{W_xf(x),f(x)} \mu(x).$$ 
Recall that $q_{W_-}$ is {\em relatively form bounded} with respect to $\cQ^c_{\Phi, W_+;\, E}$ on $\ell^2(X,\mu; E)$ with relative bound $\alpha \geq 0$,  if there exists $C \in \R$ such that
$$q_{W_-} (f) \leq \alpha \cQ^c_{\Phi, W_+;\, E}(f) + C\|f\|_{2;\, E}^2 \text{ for all }f \in \Gamma_c(X; E).$$
We shall deal with the following classes of endomorphisms.
\begin{definition}[Admissible endomorphisms]\label{definition:admissible endomorphisms}
A self-adjoint bundle endomorphism $W$ is called {\em admissible} with respect to $\mu$ and  $\Phi$, if $\cQ^c_{\Phi, W;\, E}$ is lower semi-bounded on  $\ell^2(X,\mu; E)$. The class of all admissible bundle endomorphisms with respect to $\mu$ and $\Phi$ is denoted by $\cA_{\mu, \Phi; E}$.  We say that an admissible endomorphism $W$ has a {\em form small negative part} if $q_{W_-}$ is relatively form bounded  with respect to $\cQ^c_{\Phi, W_+;\, E}$ on $\ell^2(X,\mu; E)$ with relative bound $\alpha < 1$. The class of endomorphisms with form small negative part is denoted by $\cS_{\mu, \Phi; E}$.
\end{definition}
In the scalar case when $E = (\IC)_{x \in X}$ and $\Phi = (\mathrm{Id})_{x,y \in X}$ we drop the subscripts for the connection and the bundle and write $\cA_{\mu}$ for the class of {\em admissible potentials} $\cA_{\mu, \mathrm{Id}; (\IC)}$ and we write $\cS_\mu$ for the class of potentials with form small negative part $\cS_{\mu, \mathrm{Id}; (\IC)}$.

\begin{remark}
The name admissible potentials is borrowed from \cite{KLVW}.      
\end{remark}
To a self-adjoint bundle endomorphism $W$ on $E$ we associate the function $\Wm:X \to \R$  that is  given by
$$\Wm(x) := \inf\{ \as{W_x\xi,\xi} \mid \xi \in E_x \text{ with } |\xi| = 1\}. $$
Since $W_x$ is bounded and self-adjoint, $\Wm(x)$ is the minimum of the spectrum of $W_x$.  If $W$ acts on sections by multiplication with a real-valued   $V \in C(X)$, then $\Wm(x) = V(x)$. 

In general it is nontrivial to determine whether  a given endomorphism belongs to one of the discussed classes. However, if there is some information in the scaler case, Kato's inequality provides some information in the magnetic case. More precisely, the following proposition  characterizes relative boundedness  (uniform in the connection) of magnetic Schrödinger forms with endomorphism $W$  in terms of the scalar Schrödinger form with potential $\Wm$. 

\begin{proposition} \label{proposition:semiboundedness magnetic forms and domination}
The following assertions are equivalent.
\begin{enumerate}[(i)]
 \item  $\Wm \in \cA_{\mu}$, i.e., $\cQ^c_{\Wm}$ is lower semi-bounded on $\ell^2(X,\mu)$.
 \item  There exists $C \in \R$ such that for all unitary connections $\Phi$ on $E$ and all $\varphi \in \Gamma_c(X; E)$  we have
 $$C\|\varphi \|_{2;\, E}^2 \leq \cQ^c_{\Phi, W;\, E}(\varphi ).$$
\end{enumerate}
\end{proposition}
\begin{proof}
 (i) $\Rightarrow$ (ii):  Let $\varphi \in \Gamma_c(X; E)$. We use Green's formula (Lemma~\ref{lemma:greens formula forms}) and Kato's inequality together with the bound $W \geq \Wm$ to obtain
 \begin{align*}
  \cQ^c_{\Phi,W; E}(\varphi) &= (\varphi, \cM_{\mu,\Phi,W} \varphi)_E \geq (|\varphi|, \cH_{\mu, \Wm} |\varphi|)_E  = \cQ^c_{\Wm}(|\varphi|).
 \end{align*}
  Since $\||\varphi| \|_2 = \|\varphi\|_{2; \, E}$, assertion (ii) follows from  this inequality and $\Wm \in \cA_\mu$.
 
 (ii) $\Rightarrow$ (i): Let $\varepsilon > 0$. Choose $\eta \in \Gamma(X;E)$ with $|\eta| = 1$  such that for all $x \in X$ it satisfies $\Wm(x) \geq \as{W_x\eta(x),\eta(x)} - \varepsilon$. Moreover, choose a unitary connection $\Phi$ such that for all $x,y \in X$ we have $\Phi_{x,y} \eta (y) = \eta (x)$. By assumption there exists $C \in \R$ such that for all $\varphi \in C_c(X)$ we have 
 \begin{align*}
  C\|\varphi\|^2_2  = C\|\varphi\eta\|_{2;\, E}^2 \leq\cQ^c_{\Phi,W; E}(\varphi\eta) = \cQ^c_{0}(\varphi) + q_{W}(\varphi \eta),
 \end{align*}
 where in the last step we used
 $$|\varphi(x) \eta(x) - \varphi(y) \Phi_{x,y} \eta(y)|^2= |\varphi(x) \eta(x) - \varphi(y)\eta(x)|^2 = |\varphi(x) - \varphi(y)|^2.$$
 Since 
 $$q_W(\varphi\eta) = \sum_{x \in X} |\varphi(x)|^2 \as{W_x\eta(x),\eta(x)} \mu(x) \leq \sum_{x \in X} |\varphi(x)|^2 (\Wm + \varepsilon) \mu(x)$$
 and $\varepsilon > 0$ was arbitrary, this shows that $\cQ^c_{\Wm}$ is lower semi-bounded on $\ell^2(X,\mu)$. 
\end{proof}

\begin{remark}
\begin{enumerate}
 \item The previous proposition shows that $\Wm \in \cA_{\mu}$ implies $W \in  \cA_{\mu, \Phi; E}$ for all unitary connections $\Phi$ on $E$. For scalar magnetic Schrödinger operators also the converse implication is true, but for general magnetic Schrödinger operators the situation is unclear.
%
 \item Even if  $\Wm \not \in \cA_{\mu}$, there can still exist connections $\Phi$ for which  $\cQ^c_{\Phi,W; E}$ is lower semi-bounded on $\ell^2(X,\mu; E)$, see Example~\ref{example:adjacency matrix}. 
\end{enumerate}
\end{remark}

The following proposition shows that closability of forms with nonnegative endomorphisms is preserved under perturbations by small negative endomorphisms.

\begin{proposition}[Realization for form small negative parts] \label{proposition:small negative part}
 Let $W \in \cS_{\Phi;E}.$ Then $\cQ^c_{E} = \cQ^c_{\Phi,W; E}$ is lower semi-bounded and closable on $\ell^2(X,\mu; E)$. Its closure $Q^0_E$ is the restriction of $\cQ_E$ to $D(Q^0_{\mu,\Phi, W_+; E})$ and the associated operator $M^0$ is a realization of $\cM$.
\end{proposition}
\begin{proof}
Since $W \in \mathcal{S}_{\mu, \Phi; E}$, there exists an $\varepsilon > 0$ and constants $K' > K > 0$ such that for all $\varphi \in \Gamma_c(X;E)$ we have
$$0\leq \varepsilon \cQ_{\Phi, W_+; E}^c (\varphi) + K\|\varphi\|_{2; E}^2 \leq \cQ_{E}^c (\varphi) + K'\|\varphi\|_{2; E}^2 \leq  \cQ_{\Phi, W_+; E}^c (\varphi) + K'\|\varphi\|_{2; E}^2. $$
Therefore, the form norms of $\cQ_{E}^c$ and $\cQ_{\Phi, W_+; E}^c$ are equivalent. Since $\cQ_{\Phi, W_+; E}^c$ is closable on $\ell^2(X,\mu; E)$ with closure $Q^0_{\mu,\Phi, W_+; E}$, it follows that $\cQ_{E}^c$ is closable and that its closure $Q^0_E$ satisfies $D(Q^0_E) = D(Q^0_{\mu,\Phi, W_+; E})$. 

Next we prove that $Q^0_E$ is a restriction  of $\cQ_E$. Let $f \in D(Q^0_E) = D(Q^0_{\mu,\Phi, W_+; E})$ and choose a sequence $(\varphi_n)$ in $\Gamma_c(X; E)$ that converges to $f$ with respect to the form norm of $Q^0_E$. We have $\varphi_n \to f$ in $\ell^2(X,\mu; E)$, $\cQ^c_E(\varphi_n)  \to Q^0_E(f)$ and, by the equivalence of the form norms, $\cQ_{\Phi, W_+; E}^c(\varphi_n) \to Q^0_{\mu,\Phi, W_+; E}(f)$, as $n \to \infty$. 
Fatou's lemma and the properties of $(\varphi_n)$ imply
\begin{align*}
 \sum_{x \in X} \as{W_-(x)f(x),f(x)} &\leq \liminf_{n\to \infty} q_{W_-}(\varphi_n)\\
 &\leq  \liminf_{n\to \infty} \left(\cQ_{\Phi, W_+; E}^c(\varphi_n) + K' \|\varphi_n\|_{2; E}^2\right)\\
 &=  Q^0_{\mu,\Phi, W_+; E}(f) +  K' \|f\|_{2; E}^2.
\end{align*}
This shows $f \in D(q_{W_-})$ and, since $f \in D(Q^0_E)$ was arbitrary, also
\begin{align}
q_{W_-}(f - \varphi_n) \leq Q^0_{\mu,\Phi, W_+; E}(f- \varphi_n) +  C \|f - \varphi_n\|_{2; E}^2.\label{inequality:negatice part of form} \tag{$\triangle$}
\end{align}
 In Proposition~\ref{proposition:nonnegative endomorphisms} we proved $D(Q^0_{\mu,\Phi, W_+; E}) \subseteq \mathcal{D}_{\Phi, W_+; E}$. Since $\mathcal{D}_{\Phi, W_+; E} \cap D(q_{W_-}) \subseteq \mathcal{D}_{\Phi, W; E}$, we obtain $f \in  \mathcal{D}_{\Phi, W; E}$. The properties of $(\varphi_n)$ and that $q_{W_-}$ is a quadratic form yield 
\begin{align*}
 |Q^0_E(f)^{1/2} &- \cQ_E(f)^{1/2}| = \lim_{n\to \infty} |\cQ^c_E(\varphi_n)^{1/2} - \cQ_E(f)^{1/2}| \\
 &\leq \limsup_{n\to \infty} |\cQ^c_{\Phi, W_+; E}(\varphi_n)^{1/2} -  \cQ_{\Phi, W_+; E}(f)^{1/2}| + \limsup_{n\to \infty} |q_{W_-}(\varphi_n)^{1/2} - q_{W_-}(f)^{1/2}|.\\
 &\leq |Q^0_{\mu,\Phi, W_+; E}(f)^{1/2} - \cQ_{\Phi, W_+; E}(f)^{1/2}| + \limsup_{n\to \infty} q_{W_-}(\varphi_n - f)^{1/2}=0. 
\end{align*}
The last equality follows from the fact that $Q^0_{\mu,\Phi, W_+; E}$ is a restriction of $\cQ_{\Phi, W_+; E}$, see Proposition~\ref{proposition:nonnegative endomorphisms}, and Inequality~\eqref{inequality:negatice part of form}. Therefore, $Q^0_E$ is a restriction of $\cQ_E$. With this at hand, the statement on the operator $M^0$ follows from Green's formula (Lemma~\ref{lemma:greens formula forms}), cf. the end of the proof of Proposition~\ref{proposition:nonnegative endomorphisms}.
\end{proof}

If the endomorphism only belongs to $\cA_{\mu, \Phi; E}$ and not to $\cS_{\mu, \Phi; E}$, the closability of $\cQ^c_E$ and the existence of realizations of $\cM$ is more delicate. Even if $\cQ^c_E$ is closable, the existence of realizations of $\cM$ is nontrivial. This is due to the following observation.  Suppose that $\cQ^c_E$  is closable. In this case it can happen   that there exist $\varphi_n \in  \Gamma_c(X; E)$ and $f \in D(Q^0_E)$ such that $\varphi_n \to f$  with respect to the form norm, $q_{W_-}(\varphi_n) \to \infty$ and $\cQ_{\Phi, W_+; E}^c(\varphi_n) \to \infty$, as $n \to \infty$, see Example~\ref{example:hardy}. For such a sequence we still have 
$$Q^0_E(f) = \lim_{n \to \infty} \cQ^c_E(\varphi_n),$$
but  $f$ does not belong to $\cD_E$. In particular, $Q^0_E$ is not a restriction if  $\cQ_E$ and we cannot infer from Green's formula  that $M^0$ is a realization of $\cM$. There are essentially two ways for dealing with this situation, which are discussed in the subsequent subsections.
%
\subsection{Admissible endomorphisms and domination} \label{subsection:admissible endomorphisms and domination}

In this subsection we use domination of resolvents of magnetic operators by resolvents of scalar operators to prove closability of $\cQ^c_E$ and to show that the  operator associated to its closure is a realization of $\cM$. The main result of this subsection is the following.

\begin{theorem}[Realization of discrete magnetic Schrödinger operators]\label{theorem:existence of realizations for magnetic operators}
 Let $W_{\mathrm{min}}\in \cA_{\mu}$. Then $\cQ^c_E$ is lower semi-bounded and closable on $\ell^2(X,\mu; E)$ and $M^0$ is a realization of $\cM$.
\end{theorem}

 In order to employ domination for proving closability we need several technical lemmas, which may be of interest on their own right.  For a background on the Beurling-Deny criteria we refer to Appendix~\ref{appendix:Beruling Deny}.
 
\begin{lemma}\label{lemma:qc satisfies first beurling-deny criterion}
Let $V \in \cA_{\mu}$. If $\cQ^c$ is closable on $\ell^2(X,\mu)$, its closure $Q^0$ satisfies the first Beurling-Deny criterion. In particular, for $\alpha > - \lambda_0(H^0)$ the resolvent $(H^0 + \alpha)^{-1}$ is positivity preserving, i.e., $f \geq 0$ implies $(H^0 + \alpha)^{-1} f\geq 0$.
\end{lemma}
\begin{proof}
Let $f \in D(Q^0)$ and let $(\varphi_n)$ a sequence in  $C_c(X)$ that converges to $f$ with respect to the form norm. It follows from the definition of $\cQ^c$ that
$$\cQ^c(|\varphi_n|) \leq \cQ^c(\varphi_n), \quad n\in \N.$$
The $\ell^2$-lower semicontinuity of $Q^0$ implies
$$Q^0(|f|) \leq \liminf_{n\to \infty} Q^0(|\varphi_n|) = \liminf_{n\to \infty}\cQ^c(|\varphi_n|) \leq \liminf_{n\to \infty}\cQ^c(\varphi_n) = Q^0(f) < \infty. $$
This shows $|f| \in D(Q^0)$ and the desired inequality. The statement on the resolvent follows from Proposition~\ref{propostion:beurling deny}.
\end{proof}

  A {\em self-adjoint  restriction of $\cH$} is a self-adjoint operator $H$ on $\ell^2(X,\mu)$ with $D(H) \subseteq \cF$ and $H f = \cH f$ for all $f \in D(H)$. The following lemma shows that self-adjoint restrictions whose forms satisfy the first Beurling-Deny criterion are indeed realizations in the sense of Definition~\ref{definition:discrete magnetic Schrödinger operators}.

\begin{lemma}\label{lemma:q and extension of qc}
 Let $V \in C(X)$ real-valued. Let $H$ be a lower semi-bounded self-adjoint restriction of $\cH$ such that the associated quadratic form $Q$ satisfies the first Beurling-Deny criterion. Then $Q$ is an extension of $\cQ^c$, $V \in \cA_{\mu}$ and $H$ is a realization of $\cH$. 
\end{lemma}
\begin{proof}
 For $\alpha > -\lambda_0(H)$ let $G_\alpha:= (H + \alpha)^{-1}$ the associated resolvent. Since $H$ is a restriction of $\cH$, $D(H) \subseteq \cF$ so that $G_\alpha$ maps $\ell^2(X,\mu)$ to $\cF$.
 
 We show that for $\varphi \in C_c(X)$ the convergence $\cH \alpha G_\alpha \varphi \to \cH \varphi$  holds pointwise, as  $\alpha \to \infty$. Once this is proven, Green's formula (Lemma~\ref{lemma:greens formula forms}) and the characterization of $Q$ via approximating forms (cf. Appendix~\ref{appendix:basics}) yield  
 $$\cQ^c(\varphi) = (\varphi, \cH \varphi) = \lim_{\alpha \to \infty} (\varphi, \cH \alpha G_\alpha \varphi) =  \lim_{\alpha \to \infty} \as{\varphi, H \alpha G_\alpha \varphi}_2 = \lim_{\alpha \to \infty} \alpha \as{\varphi - \alpha G_\alpha \varphi,\varphi}_2 = Q(\varphi).$$
 This shows $\varphi \in D(Q)$ and that $Q$ is an extension of $\cQ^c$. Since $Q$ is lower semi-bounded, so is $\cQ^c$ and we obtain $V \in \cA_{\mu}$. Therefore, $H$  is a realization of $\cH$ in the sense of Definition~\ref{definition:discrete magnetic Schrödinger operators}. 
 
 Let now $\varphi \in C_c(X)$. The strong continuity of the resolvent implies $\alpha G_\alpha \varphi \to \varphi$ pointwise, as $\alpha \to \infty$. We construct a function $f \in \cF$ that dominates $ \alpha G_\alpha \varphi$ for all $\alpha$ large enough. The pointwise convergence $\cH \alpha G_\alpha \varphi \to \cH \varphi$, as  $\alpha \to \infty$,  then follows from Lebesgue's theorem. 
 
 We choose a nonnegative $\psi \in \ell^2(X,\mu)$ and $\beta > -\lambda_0(H)$ such that $G_\beta \psi \geq |\varphi|$. Such  $\psi$ and $\beta$ exist because $(G_\alpha)$ is positivity preserving (Proposition~\ref{propostion:beurling deny}) and strongly continuous and  $\varphi$ has finite support.  Using the resolvent identity and that $G_\alpha$ is positivity preserving, for $\alpha > 2\beta$ we obtain
 $$ \alpha/2 |G_\alpha \varphi| \leq (\alpha - \beta) G_\alpha |\varphi| \leq    (\alpha - \beta) G_\alpha G_\beta \psi = G_\beta \psi - G_\alpha \psi \leq G_\beta \psi.$$
 This shows that for $\alpha > 2 \beta$ the function $\alpha G_\alpha \varphi$ is dominated by  $f = 2 G_\beta \psi \in \cF$ and finishes the proof.
 \end{proof}
 
 \begin{remark}
  It is important to note that in the lemma we do not assume that $H$ is a realization of $\cH$. We only assume that $H$ satisfies (R1) and (R3). The lemma says that they imply (R2) if $Q$ satisfies the first Beurling-Deny criterion.  
 \end{remark}

\begin{lemma} \label{lemma:h0 a realization}
Let $V \in \cA_{\mu}$. If $\cQ^c$ is closable on $\ell^2(X,\mu)$, the self-adjoint operator $H^0$ associated with the closure $Q^0$ is a realization of $\cH$.
\end{lemma}
\begin{proof}
  By definition $H^0$ is lower semi-bounded and $C_c(X) \subseteq D(Q^0)$. Hence, it suffices to prove $D(H^0) \subseteq \cF$ and $H^0 f = \cH f$ for $f \in D(H^0)$.  

Let $f \in D(H^0)$ and let $\alpha > - \lambda_0(H^0)$. Then $f = (H^0 + \alpha)^{-1}g$ for some $g \in \ell^2(X,\mu)$. Since $(H^0 + \alpha)^{-1}$ is positivity preserving, this shows that there exist nonnegative $f_i\in D(H^0),$ $i = 1,\ldots,4,$ with $f = f_1 - f_2 + i (f_3 - f_4)$. We can therefore assume $f \geq 0$. Let now $(\varphi_n)$ a sequence in $C_c(X)$ that converges to $f$ with respect to the form norm. Then $\psi_n :=  (\varphi_n \wedge f) \vee 0$ belongs to $C_c(X)$ and converges to $f$ in $\ell^2(X,\mu)$. Since by Lemma~\ref{lemma:qc satisfies first beurling-deny criterion} the form $Q^0$ satisfies the first Beurling-Deny criterion, we obtain
$$\|\psi_n\|_{Q^0}  \leq \|\varphi_n\|_{Q^0} + \|f\|_{Q^0}$$
from Lemma~\ref{lemma:maxima and minima energy inequality}. Thus (after taking a suitable subsequence) we can assume $\psi_n \to f$ weakly with respect to the form inner product. Using Green's formula (Lemma~\ref{lemma:Green's formula general}), for $x \in X$ we obtain
\begin{align}
 \mu(x) H^0 f(x) = Q^0(\delta_x,f) = \lim_{n \to \infty} Q^0(\delta_x,\psi_n) = \lim_{n \to \infty} \cQ^c(\delta_x,\psi_n) = \lim_{n \to \infty} \mu(x) \cH \psi_n(x). \label{equation:H idenity} \tag{$\diamondsuit$}   
\end{align}
In particular, this shows that  $\lim_n  \mu(x)\cH \psi_n(x)$ exists. Moreover,
$$ \mu(x) \cH \psi_n(x) = \deg(x) \psi_n(x) - \sum_{y \in X} b(x,y) \psi_n(y) + \mu(x) V(x) \psi_n(x).$$
Since the $\psi_n$ also converge pointwise towards $f$,  we obtain that $\sum_{y \in X} b(x,y) \psi_n(y)$ converges, as $n \to \infty.$ Fatou's lemma yields
$$0 \leq \sum_{y \in X}b(x,y) f(y) \leq \liminf_{n \to \infty} \sum_{y \in X}b(x,y) \psi_n(y) < \infty. $$
Since $x \in X$ was arbitrary, we infer $f \in \cF$. By construction we have $|\psi_n| \leq f$. Therefore, Lebesgue's dominated convergence theorem yields $\cH \psi_n(x) \to \cH f(x)$, as $n\to \infty.$ With Equation~\eqref{equation:H idenity} we arrive at  $H^0f(x) = \cH f(x)$ and the lemma is proven.
\end{proof}

\begin{lemma}[Domination]\label{lemma:domination}
Let $V \in \cA_{\mu}$ and let $W \geq V$. Suppose that $\cQ^c$ and $\cQ^c_E$ are closable. For any $\alpha > - \lambda_0(H^0)$ the value $- \alpha$ belongs to the resolvent set of $M^0$ and the resolvent $(H^0 + \alpha)^{-1}$ dominates $(M^0 + \alpha)^{-1}$, i.e., for any $f \in \ell^2(X,\mu; E)$ the following inequality holds
$$|(M^0 + \alpha)^{-1} f| \leq (H^0 + \alpha)^{-1}|f|.$$
\end{lemma}
\begin{proof}
We employ the theory of domination developed in \cite{MVV}. In contrast to our situation this paper deals with (not necessarily densely defined) forms on functions taking values in a fixed Hilbert space. If we let $\widetilde{E} = \bigoplus_{x \in X} E_x$ the direct sum of Hilbert spaces, then $\ell^2(X,\mu; E)$ isometrically embeds into the Hilbert space of square summable $\widetilde{E}$-valued functions $\ell^2(X,\mu; \widetilde{E})$. Therefore, $\cQ_E^0$ can be viewed as a not densely defined closed form on the Hilbert space $\ell^2(X,\mu; \widetilde{E})$ and the theory of \cite{MVV} can be applied.

According to \cite[Theorem~4.1]{MVV} we need to prove that there are form norm dense subspaces $U \subseteq D(Q^0_E)$ and $V \subseteq D(Q^0)$ such that the following holds:
\begin{enumerate}[(a)]
 \item $U$ is a generalized ideal in $V$, i.e.,  
 \begin{itemize}
  \item $f \in U$ implies $|f| \in V$,
  \item $f \in U, \varphi \in V$ and $|\varphi| \leq |f|$ implies $\varphi\, \mathrm{sgn}\, f \in U$.
 \end{itemize}
 \item For all $f \in U$ and $\varphi \in V$ with $0 \leq \varphi \leq |f|$ we have
 $$\mathrm{Re}\, Q_E^0( \varphi\, \mathrm{sgn}\, f, f) \geq Q^0(\varphi,|f|).$$
\end{enumerate}
We choose $U = \Gamma_c(X; E)$ and $V = C_c(X)$. By definition they are form dense subspaces of the forms $Q^0_E$ respectively $Q^0$. Moreover, $\Gamma_c(X; E)$ is  a generalized ideal in $C_c(X)$. The required inequality follows from Kato's inequality and Green's formula. More precisely, for $f \in \Gamma_c(X; E)$ and $\varphi \in C_c(X)$ with $0 \leq \varphi \leq |f|$ we have $\as{f(x), \varphi(x)\, \mathrm{sgn}\, f(x)} = |\varphi(x)||f(x)|$ and $|\varphi(x)\, \mathrm{sgn}\, f(x)| = |\varphi(x)| = \varphi(x)$. Since also $W \geq V$, Kato's inequality (Lemma~\ref{lemma:Katos inequality}) implies
$$\mathrm{Re}\,(\varphi\, \mathrm{sgn}\, f,\cM f )_E \geq (\varphi, \cH |f|).$$
With this at hand Green's formula (Lemma~\ref{lemma:greens formula forms}) yields the desired statement.
\end{proof}

\begin{theorem}[Realization of discrete Schrödinger operator]\label{theorem:existence of realizations for scalar operators}
 Let $V \in \cA_{\mu}$. Then $\cQ^c$ is lower semi-bounded and closable on $\ell^2(X,\mu)$ and $H^0$ is a realization of $\cH$. Moreover, the associated quadratic form $Q^0$ satisfies the first Beurling-Deny criterion.
\end{theorem}
\begin{proof}
 Let $\lambda_0 := \lambda_0(\cQ^c)$ the largest lower bound of $\cQ^c$ on $\ell^2(X,\mu)$. For $n \in \N_0$ we let $V_n := \max\{V,-n\}$. Since $(V_n)_-$ is bounded, it is readily verified that $V_n$ has a form small negative part so that $V_n \in \cS_\mu$. By Proposition~\ref{proposition:small negative part} the form $\cQ^c_{V_n}$ is closable; we denote its closure by $Q_n^0$ and the associated operator by $H^0_n$.  Since $V \leq V_n$, we have $\lambda_0 \leq \lambda_0(Q^0_n)$, so that for $\alpha > -\lambda_0$ the resolvent $G_\alpha^n:= (H^0_n + \alpha)^{-1}$ exists. As a first step we prove that for each $\alpha > -\lambda_0$ the sequence $(G_\alpha^n)_n$ converges strongly to some operator $G_\alpha$ and that $(G_\alpha)_{\alpha > -\lambda_0}$ is a strongly continuous resolvent family.

 If $m\geq n,$ we have $V_+ = V_0\geq V_n \geq V_m$.  For  $f\geq 0$  and  $\alpha > -\lambda_0$ Lemma~\ref{lemma:domination} applied to the scalar situation yield  
 $$0 \leq G^0_\alpha f \leq G^n_\alpha f \leq G^m_\alpha f.$$
 The bound $\lambda_0 \leq \lambda_0(H^0_n)$ implies the  bound for the operator norm
 $\|G^n_\alpha\| \leq  (\lambda_0 + \alpha)^{-1}.$
 Hence, it follows from the monotone convergence theorem that for each $f \in \ell^2(X,\mu)$ and $\alpha > -\lambda_0$ the limit
 $$G_\alpha f := \lim_{n \to \infty} G^n_\alpha f$$
 exists. It is readily verified that $(G_\alpha)$ is a family of self-adjoint bounded operators with $\|G_\alpha\| \leq  (\lambda_0 + \alpha)^{-1}$ and that it satisfies the resolvent identity. Next we prove that it is strongly continuous. To this end, we let $f \geq 0$ and use domination to estimate 
 \begin{align*}
  \|\alpha G_\alpha f - \alpha G^0_\alpha f\|^2 &= \|\alpha G_\alpha f\|^2 - 2 \as{\alpha G_\alpha f,\alpha G^0_\alpha f} + \|\alpha G^0_\alpha f\|^2\\
  &\leq \frac{\alpha^2}{(\lambda_0 + \alpha)^2} \|f\|_2^2 -   \|\alpha G^0_\alpha f\|^2.
 \end{align*}
Since $(G^0_\alpha)$ is strongly continuous, this inequality shows $\alpha G_\alpha f \to f$, as $\alpha \to \infty$, i.e., the strong continuity of $(G_\alpha)$.

Let now $Q$ be the lower semi-bounded closed quadratic form that is associated with $(G_\alpha)$ and denote by $H$ the associated self-adjoint operator. We prove that $H$ is a restriction of $\cH$ (not a realization cf. the remark after Lemma~\ref{lemma:q and extension of qc}).
  Since $G_\alpha$ is surjective onto $D(H)$, it suffices to show that $G_\alpha$ maps $\ell^2(X,\mu)$ to $\cF$ and to verify the equality $\cH G_\alpha f = f - \alpha G_\alpha f$. To this end, let $f \in \ell^2(X,\mu)$ nonnegative  and set $\mathrm{Deg} = \mu^{-1} \deg$. We already know from Proposition~\ref{proposition:small negative part} that $H^0_n$ is a restriction of $\cH_{\mu, V_n}$. For $x \in X$ the monotone convergence theorem yields
\begin{align*}
 \frac{1}{\mu(x)} \sum_{y \in X} b(x,y) G_\alpha f(y) &= \lim_{n \to \infty} \frac{1}{\mu(x)} \sum_{y \in X} b(x,y) G^n_\alpha f(y) \\
 &= \lim_{n \to \infty} \left( - \cH_{\mu, V_n} G^n_\alpha f(x) + \mathrm{Deg}(x)  G^n_\alpha f(x) + V_n(x) G^n_\alpha f(x) \right)\\
 &= \lim_{n\to \infty} \left( - f(x) + \alpha G_\alpha^n f(x) + \mathrm{Deg}(x)  G^n_\alpha f(x) + V_n(x) G^n_\alpha f(x) \right)\\
 &= - f + \alpha G_\alpha f(x) + \mathrm{Deg}(x)  G_\alpha f(x) + V(x) G_\alpha f(x).
\end{align*}
This computation implies $G_\alpha f \in \cF$ and  $\cH G_\alpha f = f - \alpha G_\alpha f$. 

The operator $H$ is a restriction of $\cH$ and as a monotone limit of positivity preserving resolvents, its resolvent is positivity preserving. Therefore, the associated form $Q$ satisfies the first Beurling-Deny criterion, see Proposition~\ref{propostion:beurling deny}. Lemma~\ref{lemma:q and extension of qc} yields that  $Q$ is an extension of $\cQ^c$ and, in particular, that $\cQ^c$ is closable. According to Lemma~\ref{lemma:h0 a realization} the corresponding operator $H^0$ is a realization of $\cH$. This finishes the proof.
 \end{proof}
\begin{remark}
 \begin{enumerate}
  \item The arguments for proving that $(G_\alpha^n)$ converges monotone to a strongly continuous resolvent are taken from \cite{KLVW}. 
  \item The form $Q$ constructed in the proof is indeed the closure of $\cQ^c$.  Proving this would have saved us from using Lemma~\ref{lemma:q and extension of qc} and Lemma~\ref{lemma:h0 a realization}. We chose this alternative presentation for two reasons. Lemma~\ref{lemma:q and extension of qc} is used below and Lemma~\ref{lemma:h0 a realization} might be interesting on its own right. If anyone comes up with a different proof of the closability of $\cQ^c,$ Lemma~\ref{lemma:h0 a realization} shows that the associated operator is a realization of $\cH.$ One such alternative proof, which uses the Feynman-Kac formula, can be found in \cite{GKS}.
 \end{enumerate}

\end{remark}
\begin{corollary} \label{corollary:realization with first beurling deny}
The operator $\cH$ has a realization whose associated quadratic form satisfies the first Beurling-Deny criterion if and only if $V \in \cA_{\mu}$. 
\end{corollary}
\begin{proof}
If $V \in \cA_{\mu}$, the previous theorem and Lemma~\ref{lemma:qc satisfies first beurling-deny criterion} show that $H^0$ is a realization of $\cH$ whose associated form satisfies the first Beurling-Deny criterion.

If $H$ is a realization of $\cH$ whose associated quadratic form $Q$ satisfies the first Beurling-Deny criterion, then Lemma~\ref{lemma:q and extension of qc} implies that $Q$ is an extension of $\cQ^c$. Hence, $\cQ^c$ is lower semi-bounded, i.e., $V \in \cA_{\mu}$.
\end{proof}

\begin{proof}[Proof of Theorem~\ref{theorem:existence of realizations for magnetic operators}]
By  assumption we have $\Wm \in \cA_{\mu}$. Therefore, Theorem~\ref{theorem:existence of realizations for scalar operators} shows that $H^0 := H^0_{\mu, \Wm}$ is a realization of $\cH_{\mu, \Wm}$. For $\alpha > - \lambda_0(H^0)$ we denote the associated resolvent by $R_\alpha:= (H^0 + \alpha)^{-1}$.

Let 
$K_n := \{x \in X \mid \lambda_0(W_x) \geq -n\}$
and let  $W_n := W_+ - 1_{K_n} W_-$. Since $1_{K_n} W_-$ is uniformly bounded, it is readily verified that $W_n \in \mathcal{S}_{\mu, \Phi; E}$. We denote the closure of $\cQ^c_{\Phi, W_n; E}$ on $\ell^2(X,\mu; E)$ by $Q^n_E$ and write $(G_\alpha^n)$ for the corresponding resolvent; their existence follows from Proposition~\ref{proposition:small negative part}. Since $W_n \geq \Wm$, Proposition~\ref{proposition:semiboundedness magnetic forms and domination} shows that all the  $Q^n_E$ have  a common lower bound. Hence, the forms $(Q^n_E)$ fulfill the assumptions of Lemma~\ref{lemma:monotone convergence of forms} on monotone convergence of quadratic forms. We obtain that the  resolvents $(G_\alpha^n)$ strongly converge to a resolvent $(G_\alpha)$ with associated lower semi-bounded closed quadratic form $Q_E$  and self-adjoint operator $M$. 

We prove that $M$ is a realization of $\cM$. Condition (R1) is trivially satisfied since $Q_E$ is lower semi-bounded. For $\varphi \in \Gamma_c(X; E)$ it follows from Lemma~\ref{lemma:monotone convergence of forms} that 
$$Q_E(\varphi) \leq \liminf_{n \to \infty} Q_E^n(\varphi) = \cQ^c(\varphi) < \infty,$$
showing (R2). Since $W_n \geq \Wm,$  Lemma~\ref{lemma:domination} yields that the resolvent $R_\alpha$ dominates $G_\alpha^n$  and hence also $G_\alpha$, i.e., 
$$|G^n_\alpha f|, |G_\alpha f| \leq R_\alpha |f|, \quad  f \in \ell^2(X,\mu; E). $$
 The operator $H^0$ is a realization of $\cH$ so that $R_\alpha |f| \in D(H^0) \subseteq \cF$ for all $f \in \ell^2(X,\mu; E)$.  These two observations imply that the image of $(G_\alpha)$ is contained in $\cF_E$ so that $D(M) = G_\alpha \ell^2(X,\mu; E) \subseteq \cF_E$. Moreover, domination of the resolvents and Lebesgue's dominated convergence theorem show that for $x \in X$  and $f \in \ell^2(X,\mu; E)$ in the space $E_x$  we have
 $$\sum_{y \in X } b(x,y)\Phi_{x,y} G_\alpha f(y) =  \lim_{n\to \infty} \sum_{y \in X} b(x,y) \Phi_{x,y} G^n_\alpha f(y). $$
 For $x \in X$ this implies 
 \begin{align*}
  \cM G_\alpha f(x)  &= \lim_{n\to \infty} \cM G^n_\alpha f(x) \\
  &=\lim_{n\to \infty} \left(\cM_{\mu, \Phi, W_n} G_\alpha^n f (x) + (W_x - (W_n)_x)G_\alpha^n f(x) \right)\\
  &= \lim_{n\to \infty} \left(f(x) - \alpha G^n_\alpha f(x) +   (1_{K_n}(x) - 1) (W_-)_x G_\alpha^n f(x) \right)\\
  &= f(x) - \alpha G_\alpha f(x)\\
  &= M G_\alpha f (x).
 \end{align*}
 For the third to last equality we used the definition of $W_n$ and that the operator associated with $Q^n_E$ is a realization of $\cM_{\mu, \Phi, W_n}$, see Proposition~\ref{proposition:small negative part}. From this computation it follows that $M$ is a restriction of $\cM$ so that $M$ satisfies (R3).

It remains to show $Q_E = Q^0_E$.  \cite[Corollary~2.4]{LSW2} implies that $\Gamma_c(X; E)$ is dense in  $D(Q_E)$. More precisely, the resolvent $(G_\alpha)$ of $Q_E$ is dominated by the resolvent $(R_\alpha)$ of $Q_{\mu, \Wm}^0$ and $C_c(X) = \ell^2_c(X,\mu)$ (the $\ell^2$-functions with compact support in $X$) is dense in $D(Q^0_{\mu,\Wm})$ with respect to the form norm. It then follows from \cite[Corollary~2.4]{LSW2} that $ \Gamma_c(X; E) = \ell_c^2(X,\mu; E)$ (the $\ell^2$-sections with compact support in $X$) is dense in $Q_E$ with respect to the form norm.  

Let now $\varphi \in \Gamma_c(X; E)$. Below we prove that for every $x \in X$ we have $\cM \alpha G_\alpha \varphi(x)  \to \cM \varphi(x)$ in $E_x$, as $\alpha \to \infty$.  Since $M$ is a realization of $\cM$ and $\varphi$ has finite support, this implies
$$Q_E(\varphi) = \lim_{\alpha \to \infty} Q_E(\varphi, \alpha G_\alpha  \varphi) = \lim_{\alpha \to \infty} \as{\varphi, \cM \alpha G_\alpha \varphi}_{2; E} = (\varphi,\cM \varphi)_E = \cQ^c_E(\varphi),$$
where for the last equality we used Green's formula (Lemma~\ref{lemma:greens formula forms}). This shows that $Q_E$ is an extension of $\cQ^c$. Together with $\Gamma_c(X;E)$ being dense in $D(Q_E$) we arrive at  $Q_E = Q_E^0$.

To finish the proof we show $\cM \alpha G_\alpha \varphi(x)  \to \cM \varphi(x)$, as $\alpha \to \infty$, with almost the same arguments as in the proof of Lemma~\ref{lemma:q and extension of qc}.     Let $\psi \in \ell^2(X,\mu)$ nonnegative and $\beta > - \lambda_0(H^0)$ such that $R_\beta \psi \geq |\varphi|$. It follows from the domination of $G_\alpha$ by $R_\alpha$ and the resolvent identity for $R_\alpha$, that for $\alpha > \max \{2\beta, -\lambda_0(M)\}$ we have
$$\alpha/2 |G_\alpha \varphi| \leq \alpha/2 R_\alpha |\varphi| \leq R_\beta \psi,$$
cf. the proof of Lemma~\ref{lemma:q and extension of qc}. Since $R_\beta \psi \in \cF$ and $(G_\alpha)$ is strongly continuous, an application of Lebesgue's dominated convergence theorem yields
$$\sum_{y \in X} \Phi_{x,y}\alpha G_\alpha \varphi(y) \to \sum_{y \in X} \Phi_{x,y} \varphi(y), \text{ as } \alpha \to \infty.  $$
This implies the desired convergence $\cM \alpha G_\alpha \varphi(x)  \to \cM \varphi(x)$ in $E_x$, as $\alpha \to \infty$, and finishes the proof.
\end{proof}

\begin{remark}
 In the previous proof we showed that the form $Q_E$ constructed there equals $Q^0_E$. As a short cut we employed the theory developed in \cite{LSW2} to show that $\Gamma_c(X; E)$ is dense in $D(Q_E)$. However, for proving  Theorem~\ref{theorem:existence of realizations for magnetic operators} this is not necessary. It is also possible to argue the same way as at the end of the proof of Theorem~\ref{theorem:existence of realizations for scalar operators}: First show that the operator associated with $Q_E$ is a restriction of $\cM$ and then use domination to obtain that $Q_E$ is an  extension of $\cQ^c_E$ (the latter statement is a version of Lemma~\ref{lemma:q and extension of qc} for magnetic operators, which  is one step in the presented proof of Theorem~\ref{theorem:existence of realizations for magnetic operators}). After that prove a version of Lemma~\ref{lemma:h0 a realization} for magnetic forms with the help of domination instead of using that resolvents are positivity preserving.    
 
 The known proofs for Theorem~\ref{theorem:existence of realizations for magnetic operators} (the discussed one  and a probabilistic one for  scalar magnetic Schrödinger operators in \cite{GKS}, cf. Subsection~\ref{subsection:summary and examples}) have in common that they use domination of forms.  Such a perturbative approach has the drawback that it always has two steps: One needs to first prove existence of realizations for scalar Schrödinger operators before one can treat the magnetic case.
\end{remark}

\subsection{Admissible endomorphisms and graphs with a finiteness condition} \label{subsection:admissible endomorphisms and fc}

The last existence result of realizations deals with the situation when $\cQ^c_E$ is induced by  a symmetric operator, i.e., when $\cM$ maps $\Gamma_c(X; E)$ to $\ell^2(X,\mu; E)$. We first put this condition into perspective. 
\begin{lemma}\label{lemma:characterizing (FC)}
 The following assertions are equivalent.
 \begin{enumerate}[(i)]
  \item $\cM \Gamma_c(X; E) \subseteq \ell^2(X,\mu; E)$.
  \item For all $x \in X$ the function $X \to \R, y \mapsto b(x,y)/\mu(y)$ belongs to $\ell^2(X,\mu)$.
 \end{enumerate}
 In this case, $\ell^2(X,\mu; E) \subseteq \cF_E$. In particular, both assertions are satisfied if the graph $(X,b)$ is locally finite.
\end{lemma}
\begin{proof}
For $x  \in X$ and $\xi \in E_x$ with $|\xi| = 1$ we let $\delta_{x,\,\xi}$ the compactly supported section with $\delta_{x,\,\xi}(x) = \xi$ and $\delta_{x,\,\xi}(y) = 0$ if $y \neq x$. It follows  from the definitions that $\cM \delta_{x,\,\xi} \in  \ell^2(X,\mu)$ if and only if $X \to \R, y \mapsto b(x,y)/\mu(y)$ belongs to $\ell^2(X,\mu)$.  This shows the equivalence of (i) and (ii).

Moreover, if (ii) holds, the Cauchy-Schwarz inequality implies 
$$\sum_{y \in X} b(x,y) |f(y)| \leq \|b(x,\cdot)/\mu\|_2 \left(\sum_{y\in X} |f(y)|^2 \mu(y)\right)^{1/2}.$$
This proves the inclusion $\ell^2(X,\mu; E) \subseteq \cF_E$.
\end{proof}

The previous lemma shows that the inclusion $\cM \Gamma_c(X; E) \subseteq \ell^2(X,\mu; E)$ only depends on $(X,b)$ and the weight $\mu$ and not on the connection nor on the endomorphism. If it is satisfied, the restriction of $\cM$ to $\Gamma_c(X; E)$ is a densely-defined operator on $\ell^2(X,\mu; E)$.  
\begin{definition}[Finiteness condition and the minimal restriction of $\cM$] \label{definition:finiteness condition}
The triplet $(X,b,\mu)$ satisfies the {\em finiteness condition} (FC) if for all $x \in X$ the function $X \to \R, y \mapsto b(x,y)/\mu(y)$ belongs to $\ell^2(X,\mu)$. In this case, the operator $M^\mathrm{min}: D(M^\mathrm{min}) \to \ell^2(X, \mu; E)$ with $D(M^\mathrm{min}) = \Gamma_c(X; E)$ and $M^\mathrm{min} f = \cM f$ for $f\in D(M^\mathrm{min})$  is called the {\em minimal restriction of $\cM$}. 
\end{definition}
The following proposition is the main result of this subsection.  
\begin{proposition}[Realization under finiteness condition]\label{proposition:realization finiteness condition}
Suppose that  (FC) holds. 
\begin{enumerate}[(a)]
 \item $(M^\mathrm{min})^* = M^\mathrm{max}$. In particular, $(M^\mathrm{min})^*$ is the restriction of $\cM$ to
 $$D((M^\mathrm{min})^*) = \{f \in \ell^2(X,\mu; E) \mid \cM f \in \ell^2(X,\mu; E)\}.$$
 \item  If $W \in \cA_{\mu,\Phi; E}$, the form $\cQ^c_E$ is lower semi-bound and closable, and the operator $M^0$ is a realization of $\cM$.
 \item If $M$ is a realization of $\cM$, then $M$ is an extension of $M^\mathrm{min}$ and the associated quadratic form  is an extension of $\cQ^c_E$. 
\end{enumerate}
\end{proposition}
\begin{proof}
(a):  Let $f \in D((M^\mathrm{min})^*)$. By Lemma~\ref{lemma:characterizing (FC)} we have $f \in \cF_E$. For $\varphi \in \Gamma_c(X; E)$ we infer from Green's formula (Lemma~\ref{lemma:Green's formula general}) and the definition of $M^\mathrm{min}$ that
\begin{align*}
 \langle \varphi, (M^\mathrm{min})^* f \rangle_{2; E} = \langle M^\mathrm{min} \varphi,  f \rangle_{2; E} = \sum_{x \in X} \as{\cM \varphi(x),f(x)} \mu(x) = (\varphi, \cM f)_E.
\end{align*}
This shows $\cM f = (M^\mathrm{min})^* f \in \ell^2(X,\mu; E)$. Therefore, $(M^\mathrm{min})^*$ is a restriction of $M^\mathrm{max}$.

Let now $f \in D(M^\mathrm{max})$. Since by definition $D(M^\mathrm{max}) \subseteq \cF_E$,  Green's formula (Lemma~\ref{lemma:Green's formula general}) implies that for $\varphi \in D(M^\mathrm{min}) = \Gamma_c(X; E)$ we have
$$\langle M^\mathrm{min} \varphi,  f \rangle_{2; E} = \sum_{x \in X} \as{\cM \varphi(x),f(x)} \mu(x) = (\varphi, \cM f)_E = \langle \varphi,  M^\mathrm{max} f \rangle_{2; E}. $$
This shows $f \in D((M^\mathrm{min})^*)$ and that $M^\mathrm{max}$ is a restriction of $(M^\mathrm{min})^*$.

The ``In particular''-statement follows from the definition of $M^\mathrm{max}$ and Lemma~\ref{lemma:characterizing (FC)}.

(b): For $\varphi,\psi \in \Gamma_c(X; E)$ Green's formula (Lemma~\ref{lemma:Green's formula general}) and the definition  of $M^\mathrm{min}$ yield
$$\langle M^\mathrm{min} \varphi,\psi\rangle_{2; E} = \sum_{x \in X} \as{\cM \varphi(x),\psi(x)} \mu(x) =  \cQ^c_E(\varphi,\psi) = (\varphi, \cM \psi)_E = \langle  \varphi, M^\mathrm{min}\psi\rangle_{2; E}.$$
Therefore, $\cQ^c_E$ is the quadratic form of the symmetric operator $M^\mathrm{min}$. Since $W \in \cA_{\mu, \Phi; E}$, it is also lower semi-bounded. It follows from Friedrichs' extension theorem that $\cQ^c_E$ is closable and that the self-adjoint operator $M^0$ that is associated with the closure $Q^0_E$ is an extension of $M^\mathrm{min}$. With this at hand, (a) implies that $M^0$ is a restriction of $M^\mathrm{max}$ and therefore a restriction of $\cM$. Since  $\Gamma_c(X; E) \subseteq D(Q^0_E)$, the  operator $M^0$ is a realization of $\cM$.

(c): We first prove that the associated quadratic from, which we denote by  $Q$, is an extension of $\cQ^c_E$. Let $\varphi \in \Gamma_c(X; E)$. By the definition of realizations we have $\varphi \in D(Q)$. The domain of $M$ is dense in  $D(Q)$ with respect to the form norm. Hence, there exists a sequence $f_n \in D(M)$ such that $f_n \to \varphi$ with respect to the form norm.  Since $\cM \varphi \in \ell^2(X,\mu; E)$ and $f_n \in D(M) \subseteq \cF_E$, we obtain with the help of Green's formula (Lemma~\ref{lemma:Green's formula general}) that
$$Q(\varphi) = \lim_{n \to \infty} Q(\varphi, f_n) = \lim_{n \to \infty} \langle \varphi, Mf_n \rangle_{2; E} =  \lim_{n \to \infty} \langle \cM \varphi, f_n \rangle_{2; E} = \langle \cM \varphi, \varphi\rangle_{2; E} = \cQ_E^c(\varphi).$$
It remains to prove the statement about $M$.  Let $\varphi \in \Gamma_c(X; E)$. Since $\ell^2(X,\mu;E) \subseteq \cF_E$, for $f \in D(M)$  Green's formula (Lemma~\ref{lemma:Green's formula general}) yields
$$ \langle   \cM \varphi, f\rangle_{2; E} = (\varphi,\cM f)_E = \langle  \varphi, Mf\rangle_{2; E}.$$
holds. Since $\cM \varphi \in \ell^2(X,\mu ;E)$, this implies $\varphi \in D(M^*)$ and $M^* \varphi = \cM \varphi$. Now the claim follows because $M$ is self-adjoint. 
\end{proof}

\begin{remark}
\begin{enumerate}
 \item  This proposition allows more general endomorphisms than the corresponding results  for graphs without (FC) in Subsection~\ref{subsection:admissible endomorphisms and domination} and its proof is much simpler. The reason for this is that under (FC) Green's formula is valid for for sections with compact support, i.e., one has 
 $$\as{\varphi, \cM \psi}_{2; E} = \cQ^c_E(\varphi,\psi),\quad \varphi,\psi \in \Gamma_c(X; E).$$
 \item The proposition also shows that under (FC) the operator $M^\mathrm{max}$ is closed. It would be interesting to known whether or not this is true for graphs which do not satisfy (FC).
\end{enumerate}
\end{remark}

 The following lemma shows why we formulated condition (R2) in the definition of realizations for the domain of the associated quadratic form and not for the domain of the operator. Otherwise, we could have only dealt with graphs $(X,b)$ and weights $\mu$ that satisfy (FC).

\begin{corollary}\label{corollary:operator domain (FC)}
Let $W \in \cA_{\mu, \Phi; E}$. The following assertions are equivalent.
\begin{enumerate}[(i)]
 \item (FC) holds.
 \item For any realization $M$ of $\cM$ we have $\Gamma_c(X; E) \subseteq D(M)$.
 \item There exists a realization $M$ of $\cM$ with $\Gamma_c(X; E) \subseteq D(M)$.
\end{enumerate}
\end{corollary}
\begin{proof}
 (i) $\Rightarrow$ (ii): Proposition~\ref{proposition:realization finiteness condition}~(c) shows that any realization of $\cM$ is an extension of $M^\mathrm{min}$ so that $\Gamma_c(X; E) = D(M^\mathrm{min}) \subseteq D(M)$.
 
 (ii) $\Rightarrow$ (iii): Proposition~\ref{proposition:realization finiteness condition}~(b) yields the existence of a realization.
 
 (iii) $\Rightarrow$ (i): Let $M$ a  realization of $\cM$ with $\Gamma_c(X; E) \subseteq D(M)$. Then $\cM \Gamma_c(X; E) = M \Gamma_c(X; E) \subseteq \ell^2(X, \mu; E)$, i.e., (FC) holds.
\end{proof}

\subsection{Summary and examples} \label{subsection:summary and examples}

In this subsection we summarize the results of this section and put them into perspective of the existing literature. Moreover, we discuss some examples that show the optimality of our results.

The whole section was devoted to proving the existence of realizations of $\cM$ (and $\cH$) under the following conditions.  
\begin{itemize}
 \item[(a)] $W \in \cA_{\mu, \Phi; E}$ and  (FC)   - admissible endomorphisms \& graphs with a finiteness condition.
 \item[(b)] $W \in \cS_{\mu, \Phi; E}$  - endomorphisms with small negative part. 
 \item[(c)] $W_{\mathrm{min}} \in  \cA_{\mu}$ - admissible endomorphisms dominated by an admissible potential.
\end{itemize}

\begin{remark}[Existing literature] 
\begin{enumerate}[(a)]
      \item  The condition (FC),  Lemma~\ref{lemma:characterizing (FC)}, Proposition~\ref{proposition:realization finiteness condition} and their proofs are basically taken from \cite{KL}, which contains versions of these results  for Schrödinger operators with nonnegative potentials. With the same arguments a version of Proposition~\ref{proposition:realization finiteness condition} is proven in \cite{GKS}  for scalar magnetic Schrödinger operators with  admissible potentials.  
\item The stability of closability under form small perturbations is a standard result in perturbation theory of quadratic forms. Thus, the closability of  $\cQ^c_E$ in  Proposition~\ref{proposition:small negative part} is well-known. The arguments for proving that in this case $M^0$ is a realization of $\cM$ are taken from \cite{GKS}, which treats scalar magnetic Schrödinger operators. That they can be extended to general magnetic Schrödinger operators has also been observed in \cite{GMT}.

\item For scalar magnetic Schrödinger operators the closability of $\cQ^c_{\theta,V}$ on $\ell^2(X,\mu)$ when $V \in \cA_{\mu}$ is one of the main results of \cite{GKS}. There, the proof is based on domination and explicit computations involving a Feynman-Kac-Ito formula for the corresponding semigroups. Our approach to proving  the closability statements in Theorem~\ref{theorem:existence of realizations for scalar operators} and in Theorem~\ref{theorem:existence of realizations for magnetic operators}  also uses domination but, in contrast,  is entirely analytic. The statement that $\Wm \in \cA_{\mu}$ implies that $M^0$ is a realization of $\cM$ is new even for scalar (magnetic) Schrödinger operators. In \cite{GKS,KPP} the authors could only prove this result under additional conditions.  
\end{enumerate}
\end{remark}

\begin{remark}[Optimality of the results]
  \begin{enumerate}[(a)] 
  \item   For graphs with (FC) the condition $W \in \cA_{\mu, \Phi; E}$ is optimal. In this case, the quadratic form of any realization of $\cM$ is an extension of $\cQ^c_E$, see Proposition~\ref{proposition:realization finiteness condition}. Thus, if (FC) holds and $\cM$ has a realization, then $\cQ^c_E$ is necessarily lower semi-bounded, i.e., $W \in \cA_{\mu, \Phi; E}$. 
  
  \item[(c)]    For general graphs it is unclear whether or not $W_{\mathrm{min}} \in  \cA_{\mu}$ is optimal for the existence of realizations of $\cM$. In particular, the existence of realizations of $\cM$ remains unresolved when the graph does not satisfy (FC) and $W \in \cA_{\mu, \Phi; E}$ but $W_{\mathrm{min}} \not \in \cA_{\mu}$.   Even in the case of scalar Schrödinger operators it is unclear whether $V \in \cA_{\mu}$ is necessary for the existence of realizations of $\cH$. We only proved that $V \in \cA_{\mu}$ is equivalent to the existence of a realization whose associated quadratic form satisfies the first Beurling-Deny criterion, cf.  Corollary~\ref{corollary:realization with first beurling deny}.  The problem is that in general we were not able prove that the quadratic form of an arbitrary realization of $\cH$ is an extension of $\cQ^c$; we needed to assume the first Beurling-Deny criterion, cf. Lemma~\ref{lemma:q and extension of qc}.
  \end{enumerate}
\end{remark}

The following example shows that Proposition~\ref{proposition:semiboundedness magnetic forms and domination} is only valid with some uniform control over all magnetic fields. It is taken from \cite{Gol}.

\begin{example}\label{example:adjacency matrix} \label{lemma:adjacency matrix}
 Let $K_n = (X_n,b_n)$ be the complete graph on $n$-vertices $X_n$, i.e., $|X_n| = n$ and $b_n(x,y) = 1$ for all $x,y \in X_n$, and let $\mu_n$ the counting measure on $X_n$. The corresponding adjacency operator
 $$A_n f(x) := \sum_{y \in X_n} b_n(x,y) f(x) = \sum_{y \in X_n} f(y)$$
 has the eigenvalues $n$ and $-1$.
 
 Let $(X,b)$ be the direct sum of $(X_n,b_n), n \in \N$, i.e.,  $X = \bigsqcup_{n \geq 1} X_n$ and $b(x,y) = 1$, if $x,y \in X_n$ for some $n \in \N$, and $b(x,y) = 0$, else. Moreover, let $\mu$ be the counting measure on $X$ and let $A:C(X) \to C(X)$ the formal adjacency operator
 $$A f(x) = \sum_{y \in X} f(y).$$
 It is equals the scalar formal magnetic Schrödinger operator $\cM_{\mu,-\pi,-\deg}$. Since all the $A_n$ are lower semi-bounded by $-1$ on $\ell^2(X_n,\mu_n)$,  for $\varphi \in C_c(X)$ we obtain
 $$\cQ^c_{-\pi,-\deg}(\varphi) = \sum_{n = 1}^\infty \as{A_n (\varphi 1_{X_n}),\varphi 1_{X_n}}_2 \geq   \sum_{n = 1}^\infty - \| 1_{X_n} \varphi\|_2^2 = -\|\varphi\|^2.$$
 This shows that  $\cQ^c_{-\pi,-\deg}$ is semi-bounded from below so that $-\deg \in \cA_{\mu, (-\mathrm{Id}); (\IC)}$.  However, $\cQ^c_{-\pi,-\deg}$ is not bounded from above (test e.g. with the sequence of normalized eigenfunctions $f_n = n^{-1/2} 1_{X_n}$). Hence, $\cQ^c_{-\deg} = \cQ^c_{0,-\deg} = - \cQ^c_{-\pi,-\deg}$ is not bounded from below on $\ell^2(X,\mu)$, so that $-\deg \not \in \cA_{\mu} = \cA_{\mu, (\mathrm{Id}); (\IC)}$.
 
 The graph constructed in this example is not connected. However, one can modify the graph to make it connected  as follows. If for each $n \in \N$ one adds a single edge of weight $1$ from some vertex in $X_n$ to some vertex in $X_{n+1}$, then the adjacency operator $A'$ of the resulting connected graph is a bounded perturbation of $A$. Therefore, the discussed boundedness properties of $A$ are passed on to $A'$.
\end{example}

When the endomorphism $W$ has a small negative part (i.e. $W \in \cS_{\mu,\Phi; E}$) we do not only obtain that $M^0$ is a realization of $\cM$ but also gain some information about the domain of the associated quadratic form; it is contained in $\mathcal{D}_{\Phi, W_+; E}$ and $Q^0_E$ is a restriction of $\cQ_E$. In the  following example we construct a  scalar potential for which this fails. Typically this is the case for optimal Hardy weights. As a consequence we also obtain that the classes $\cA_{\mu,\Phi; E}$ and $\cS_{\mu, \Phi; E}$ are different. 

\begin{example}\label{example:hardy}
  A function $w:X \to [0,\infty)$ is called a {\em Hardy weight} for the graph $(X,b)$ if 
$$\cQ^c_0(\varphi) \geq \sum_{x \in X} |\varphi(x)|^2 w(x), \quad \varphi \in C_c(X).$$
 In \cite{KPP} a Hardy weight $w$ is called optimal, if there exists a sequence $(e_n)$ in $C_c(X)$ with the following properties.
 \begin{itemize}
  \item $\lim\limits_{n \to \infty} \left( \cQ^c_0(e_n) - \sum\limits_{x \in X} |e_n (x)|^2 w(x)\right) = 0. $
  \item The sequence $(e_n)$ converges pointwise to a nonnegative function $G$.
  \item $G \not \in \ell^2(X,w)$, i.e., 
  $$\sum_{x \in X} |G(x)|^2 w(x) = \infty.$$
 \end{itemize}
 If $(X,b)$ is connected, the function $G$ is unique up to multiplication by a constant. It is called the {\em Agmon ground state} for $(X,b)$ and the weight $w$. The existence of optimal Hardy weights is established in \cite{KPP}. On the graph $\Z^d$  with weight $b:\Z^d \times \Z^d \to \{0,1\}$ given by  $b(x,y) = 1$ if $|x-y| = 1$ and $b(x,y) = 0$ else,  they construct an optimal Hardy weight provided that $d \geq 3$.
 
 Let $(X,b)$ be connected and suppose that $w$ is an optimal Hardy weight. Let $\mu:X \to (0,\infty)$ a weight such that the Agmon ground state $G$ satisfies $G \in \ell^2(X,\mu)$ and let $V:= - w \mu^{-1}$. Since $w$ is a Hardy weight, the form 
 $$Q^c_V(\varphi) = \frac{1}{2} \sum_{x,y \in X} b(x,y) |\varphi(x) - \varphi(y)|^2 -  \sum\limits_{x \in X} |\varphi (x)|^2w(x), \quad \varphi \in C_c(X),$$
 is nonnegative.  According to Theorem~\ref{theorem:existence of realizations for scalar operators} it is closable on $\ell^2(X,\mu)$. We prove that its closure $Q^0 = Q^0_{\mu, V}$ is not a restriction of $\cQ_V$. Since $G \not \in \ell^2(X,w)$, it suffices to prove $G \in D(Q^0)$ and $Q^0(G) = 0$.
 
 Let $(e_n)$ a sequence in $C_c(X)$ as in the definition of optimal Hardy weights and consider $f_n : = (e_n \vee0) \wedge G $, which also has compact support. The choice of $\mu$ and Lebesgue's dominated convergence theorem imply that $(f_n)$ converges in $\ell^2(X,\mu)$ towards $G$.  Moreover, the $\ell^2$-lower semicontinuity of $Q^0$ and  it satisfying the first Beurling-Deny criterion imply
 $$Q^0(f_n)^{1/2} \leq \liminf_{m \to \infty} Q^0((e_n \vee 0) \wedge e_m)^{1/2} \leq \liminf_{m \to \infty} \left(Q^0(e_n)^{1/2} + Q^0(e_m)^{1/2} \right) =  Q^0(e_n)^{1/2},$$
 see Lemma~\ref{lemma:maxima and minima energy inequality}. Hence, $(f_n)$ is also a sequence as in the definition of optimal Hardy weights, which additionally converges in $\ell^2(X,\mu)$ towards $G$. These properties and the inequality $Q^0(f_n - f_m)^{1/2} \leq Q^0(f_n)^{1/2} +   Q^0(f_m)^{1/2}$  show that $(f_n)$ is Cauchy with respect to the form norm. Since $Q^0$ is closed, it follows that $G \in D(Q^0)$ and 
 $$Q^0(G) = \lim_{n \to \infty} Q^0(f_n) = 0.$$
 This finishes the proof.

%
\end{example}

\section{Bounded realizations of $\cM$ and $\cH$} \label{section:bounded realizations}
In this  section we discuss when $\cM$ and $\cH$ have bounded realizations.  For this  the  function $B = B_{\mu,W}:X \to [0,\infty)$ that is given by  
$$B(x)  =  \sup\{|\mu(x)^{-1} \deg(x) + \langle W_x \xi,\xi\rangle| \mid  \xi \in E_x \text{ with } |\xi| = 1\}, \quad x \in X, $$
plays an important role. Our main theorem regarding bounded realizations reads as follows.
\begin{theorem}
   The following assertions are equivalent.
 \begin{enumerate}[(i)]
  \item The function $B$ is bounded and  $W \in \cA_{\mu, \Phi; E} \cap \cA_{\mu, - \Phi; E}$.
  \item The form $\cQ_E^c = \cQ^c_{\Phi, W; E}$ is bounded on $\ell^2(X,\mu; E)$.
  \item $\cM$ has a bounded realization.
  \item  $D(M^{\mathrm{max}}) = \ell^2(X,\mu; E)$.
  \item The operator $M^{\mathrm{max}}$ is  a bounded   realization of $\cM$. 
 \end{enumerate}
 If the above are satisfied, then (FC) holds and $M^\mathrm{min}$ is essentially self-adjoint.

\end{theorem}
\begin{proof}
We let $\mathrm{Deg} = \mu^{-1} \deg$. With the same symbol we denote the bundle endomorphism that acts upon $\Gamma(X; E)$ by pointwise multiplication with $\mathrm{Deg}$.    For $\varphi \in \Gamma_c (X; E)$ we note the identity
\begin{align*}
 \cQ^c_{\Phi,W; E}(\varphi)  = 2q_{\mathrm{Deg} + W}(\varphi) -  \cQ^c_{-\Phi,W; E}(\varphi).\tag{$\heartsuit$} \label{equation:magnetic identity}
\end{align*}

(i) $\Rightarrow$ (ii): The identity \eqref{equation:magnetic identity} and the lower semi-boundedness of $\cQ^c_{\Phi,W; E}$ and $\cQ^c_{-\Phi,W; E}$ imply the existence of $C \geq 0$ such that
$$- C\|\varphi\|_{2;\, E}^2 \leq \cQ^c_{\Phi,W; E}(\varphi) \leq  2q_{\mathrm{Deg} + W}(\varphi)  + C \|\varphi\|_{2;\, E}^2. $$
Since $q_{\mathrm{Deg} + W}(\varphi) \leq q_B(|\varphi|) \leq \sup B \|\varphi\|_{2;\, E}^2$, we arrive at (ii).

(ii) $\Rightarrow$ (i): Clearly, the boundedness of forms implies  lower semi-boundedness. Hence, it suffices to prove that the function $B$ and the form $\cQ^c_{-\Phi, W; E}$ are bounded. For $x \in X$ and $\xi \in E_x$ we denote by $\delta_{x,\xi}$ the finitely supported section with $\delta_{x,\xi}(x) = \xi$ and $\delta_{x,\xi}(y) = 0$ for $y \neq x$.  If $|\xi| = 1$, then
$$\cQ^c_{\Phi, W; E}(\delta_{x,\xi}) = \deg(x) + \as{W_x\xi,\xi} \mu(x).$$
From the boundedness of $\cQ^c_{\Phi, W; E}$ we infer the existence of $C \geq 0$ such that
$$|\cQ^c_{\Phi, W; E}(\delta_{x,\xi})| \leq C\|\delta_{x,\xi}\|_{2; E}^2 = C \mu(x).$$
Combining both inequalities yields that $B$ is bounded. With this at hand, the  boundedness of $\cQ^c_{-\Phi, W; E}$ follows from the boundedness of $\cQ^c_{\Phi, W; E}$ and the identity \eqref{equation:magnetic identity}.

(ii) $\Rightarrow$ (iv): Since $\cQ^c_{\Phi, W; E}$ is bounded, its closure $Q^0$ is a continuous quadratic form on $\ell^2(X,\mu; E)$ and the associated self-adjoint operator $M^0$ is bounded. For $\varphi,\psi \in \Gamma_c(X; E)$ Green's formula (Lemma~\ref{lemma:Green's formula general}) implies
$$\langle \varphi,M^0 \psi \rangle_{2; E} = Q^0(\varphi,\psi) = \cQ^c_{\Phi, W; E}(\varphi,\psi) = (\varphi, \cM \psi)_E.$$
This shows  (FC), i.e., $\cM \Gamma_c(X; E) \subseteq \ell^2(X,\mu; E)$,  and that $M^0$ is an extension of $M^\mathrm{min}$. Therefore, $M^0 = (M^0)^*$ is a restriction of $(M^\mathrm{min})^* = M^\mathrm{max}$; for the last equality we used Proposition~\ref{proposition:realization finiteness condition}. We arrive at $\ell^2(X,\mu; E) = D(M^0) \subseteq D(M^\mathrm{max})$.

(iv) $\Rightarrow$ (v): Assertion (iv) implies $\cM \Gamma_c(X; E) \subseteq \ell^2(X,\mu; E)$ and so (FC) holds. It follows from Proposition~\ref{proposition:realization finiteness condition} that $M^\mathrm{max}$ is closed. Since by assumption $D(M^\mathrm{max}) = \ell^2(X,\mu; E)$, the closed graph theorem implies that $M^\mathrm{max}$ is continuous. It remains to prove that $M^\mathrm{max}$ is self-adjoint. Green's formula (Lemma~\ref{lemma:Green's formula general}) and (FC) yield
$$\langle M^\mathrm{max} f,g \rangle_{2; E} = \langle f,  M^\mathrm{max}  g \rangle_{2; E} $$
for $f,g \in \Gamma_c(X; E)$.  By continuity this identity extends to $f,g \in \ell^2(X,\mu; E)$.

(v) $\Rightarrow$ (iii): This is trivial.


(iii) $\Rightarrow$ (ii): Let $M$ be a bounded realization of $\cM$. For $\varphi \in \Gamma_c(X; E)$ Green's formula (Lemma~\ref{lemma:Green's formula general}) implies
$$|\cQ^c_{\Phi, W; E}(\varphi)| = |(\varphi, \cM \varphi)_E|  = |\langle \varphi, M \varphi \rangle_{2; E}| \leq \|M\| \|\varphi\|_{2; E}^2.$$
This proves (ii).

Suppose now that one of the assertions holds. That they imply (FC) was proven along the way. Moreover, (v) shows that $M^\mathrm{max}$ is self-adjoint. According to Proposition~\ref{proposition:realization finiteness condition} it satisfies $M^\mathrm{max} = (M^\mathrm{min})^*$, so that $M^\mathrm{min}$ is essentially self-adjoint.
\end{proof}
\begin{remark}
   As remarked in Subsection~\ref{subsection:admissible endomorphisms and domination}, in general it is hard to determine whether a given endomorphism  $W$ belongs to $\cA_{\mu, \Phi; E} \cap \cA_{\mu, - \Phi; E}$ or not. Proposition~\ref{proposition:semiboundedness magnetic forms and domination} gives the sufficient condition $\Wm \in \cA_{\mu}$, which might be easier to check. Note that this is always satisfied if $W \geq 0$.
\end{remark}
For scalar  Schrödinger operators $\cH_{\mu,V}$ the function $B$ can be easily computed.  It is given by $B  = |\mu^{-1} \deg  + V|.$ Moreover, the lower bound on the spectrum of the endomorphism $\Wm$ is given by $V$ itself, cf. the discussion after the definition of $\Wm$ in Subsection~\ref{subsection:Nonnegative endomorphisms and small perturbations}.  Therefore, the theorem and the previous remark yield the following.
\begin{corollary}
 The following assertions are equivalent.
 \begin{enumerate}[(i)]
  \item The function $\mu^{-1} \deg  + V$ is bounded and $V \in \cA_{\mu}$.
  \item $\cQ^c = \cQ_V^c$ is bounded on $\ell^2(X,\mu)$
  \item $\cH$ has a bounded realization.
  \item $H^\mathrm{max}$ is a bounded realization of $\cH$.
  \item  $D(H^\mathrm{max})= \ell^2(X,\mu)$.
 \end{enumerate}
\end{corollary}

\begin{remark} 
 \begin{enumerate}
  \item  For Schrödinger operators with nonnegative potentials this characterization of boundedness is contained in \cite{KL2}.
  \item Example~\ref{example:adjacency matrix} shows that the assumption  $W \in \cA_{\mu, \Phi; E} \cap \cA_{\mu, -\Phi; E}$ in (i) cannot be weakened.  In the example   we constructed a graph  with counting measure $\mu$ (i.e. $\mu(x) = 1, x \in X$) such that the scalar magnetic Schrödinger operator with magnetic field $\theta = -\pi$ and potential $-\deg$ (the adjacency operator)  is bounded from above but not from below on $\ell^2(X,\mu)$. As discussed there, this means that $- \deg \in  \cA_{\mu, (-\mathrm{Id});(\IC)}$ but $-\deg \not \in \cA_{\mu} = \cA_{\mu, (\mathrm{Id}); (\IC)}$. Moreover, it satisfies $B(x) = |\mu(x)^{-1}\deg(x) - \deg(x)| = 0$.
 \end{enumerate}

\end{remark}

\section{Uniqueness of realizations} \label{section:uniqueness}

In this section we discuss two criteria that guarantee the uniqueness of  realizations of $\cM$ and $\cH$. More precisely, we prove  the absence of nonnegative subsolutions for scalar operators and then extend this to magnetic operators with the help of Kato's inequality. That this in turn yields uniqueness of realizations is guaranteed by the following lemma.

\begin{lemma}[Abstract criterion for uniqueness] \label{lemma:uniqueness abstract tool}
 Let $V \in C(X)$ real-valued. Assume that there exists $C \in \R$ such that all nonnegative $f \in \ell^2(X,\mu) \cap \cF$ with $\cH f \leq C f$ satisfy $f = 0$.
 \begin{enumerate}[(a)]
  \item If $W \geq V$, then $\cM$ has at most one realization. If, moreover, $\Wm \in \cA_{\mu}$, then $\cM$ has exactly one realization.
  \item If $W \in \cA_{\mu, \Phi; E}$ with $W \geq V$ and (FC) holds, then $M^\mathrm{min}$ is essentially self-adjoint.
 \end{enumerate}

\end{lemma}
\begin{proof}
(a): Suppose that $\cM$ has two realizations $M_1$ and $M_2$. Since both are by definition lower semi-bounded, their resolvents  $(M_1 - \lambda)^{-1}$ and $(M_2 - \lambda)^{-1}$ exist for $\lambda$ small enough. Let such a $\lambda$ with $\lambda \leq C$ be given and let $g \in \ell^2(X,\mu; E)$. Since both are realizations, the function 
$$f:= (M_1 - \lambda)^{-1}g - (M_2 - \lambda)^{-1}g$$
belongs to $\ell^2(X,\mu; E) \cap \cF_E$ and satisfies $(\cM - \lambda)f = 0$.  Kato's inequality (Lemma~\ref{lemma:Katos inequality}) and $W \geq V$ imply  $|f| \in  \ell^2(X,\mu) \cap \cF$ and $\cH |f| \leq \lambda |f|$. Since $\lambda \leq C$ and $|f| \geq 0$, we obtain $ \cH |f| \leq C |f|$. By our assumption this implies $|f| = 0$, i.e., $f = 0$. Hence, the resolvents agree and we conclude $M_1 = M_2$. If $\Wm \in \cA_{\mu}$, the existence of realizations is guaranteed by Theorem~\ref{theorem:existence of realizations for magnetic operators}.

(b) We need to prove that $(M^\mathrm{min})^*$ is self-adjoint. According to Proposition~\ref{proposition:realization finiteness condition} we have $(M^\mathrm{min})^* = M^\mathrm{max}$, where $M^\mathrm{max}$ is the restriction of $\cM$ to $D(M^\mathrm{max}) = \{f \in \ell^2(X,\mu; E) \mid \cM f \in \ell^2(X,\mu; E)\}$. Since $M^0$ is a self-adjoint and $M^\mathrm{max}$ is an extension of $M^0$, see Proposition~\ref{proposition:realization finiteness condition}, it suffices to prove $D(M^\mathrm{max}) \subseteq D(M^0)$ to settle the claim.  Let $f \in D(M^\mathrm{max})$ and for $\lambda$ small enough consider $g:= (M^0-\lambda)^{-1} (\cM-\lambda)f$, which exists since $\cM f \in \ell^2(X,\mu; E)$. Since $M^0$ is a realization of $\cM$, we conclude that
$$(\cM - \lambda)(f-g) = 0.$$
With the same arguments as in (a) we obtain $f = g \in D(M^0)$ and the claim is proven.
\end{proof}

\begin{remark} 
 For graphs satisfying (FC) essential self-adjointness is a stronger property than uniqueness of realizations. This is because realizations are always semi-bounded. Indeed, there are graphs and magnetic Schrödinger operators where $M^\mathrm{min}$ is not semi-bounded (neither from above nor from below) but essentially self-adjoint. For example, similar to Example~\ref{example:adjacency matrix} one can consider the adjacency operator on the disjoint union $\bigsqcup_{n = 2}^\infty K_{n,n}$ of  complete bipartite graphs  on $n$ vertices  $K_{n,n}$. We leave the details to the reader.
 
 It is unclear whether there are graphs with (FC) and $W \in \cA_{\mu, \Phi; E}$ such that the associated magnetic Schrödinger operator has a unique realization but $M^\mathrm{min}$ is not essentially-self-adjoint.  
\end{remark}

\subsection{A measure space criterion}\label{subsection:a measure space criterion}

In this subsection we present a uniqueness criterion that is based on combinatorics and the discreteness of the measure space. It seems to have no counterpart  for operators on smooth spaces.

\begin{theorem}[Measure space criterion for uniqueness] \label{theorem:measure space uniqueness}
Suppose $(X,b)$ has no isolated vertices and  $\Wm \in \cA_{\mu}$. If there exists $\alpha \in \R$ such that for each infinite path $(x_n)$ we have
 \begin{align}\label{equation:divergent sum}
  \sum_{n = 1}^\infty \mu(x_n) \prod_{j = 0}^{n-1} \left(1 + \frac{\mu(x_j)(\Wm(x_j) - \alpha)}{\deg(x_j)}\right)^2 = \infty, \tag{$\clubsuit$}
 \end{align}
 then $\cM$ has exactly one realization. If, additionally, (FC) holds, then $M^{\mathrm{min}}$ is essentially self-adjoint.
\end{theorem}
\begin{proof}
 To simplify notation we set $\mathrm{Deg} := \mu^{-1} \deg$. We start with proving the following observation. Under the given assumptions there exists an $\alpha \in \R$ such that  $\mathrm{Deg} + \Wm \geq \alpha.$ and \eqref{equation:divergent sum} holds for all infinite paths. 
 
Since $\cQ^c_{\Wm}$ is lower semi-bounded on $\ell^2(X,\mu)$, the function $\mathrm{Deg} + \Wm$ is bounded from below. This can be easily inferred from the identity 
$$\cQ^c_{\Wm}(\delta_x) = \deg(x) + \Wm(x) \mu(x), \quad x \in X.$$
Thus, the bound $\mathrm{Deg} + \Wm \geq \alpha$ is satisfied for small enough $\alpha \in \R$ . We show that also \eqref{equation:divergent sum} holds for $\alpha$ small enough. Let $\alpha_0 \in \R$ for which \eqref{equation:divergent sum} holds for all infinite paths.  It suffices to show that for all small enough $\alpha$  and all $x \in X$ we have 
$$(\Deg(x) + \Wm(x) - \alpha)^2 \geq (\Deg(x) + \Wm(x) - \alpha_0)^2.$$
This however is a consequence of $\Deg + \Wm$ being bounded from below.

 We now use Lemma~\ref{lemma:uniqueness abstract tool} to deduce uniqueness. Let $\alpha \in \R$ such that \eqref{equation:divergent sum} holds for all infinite paths and $\mathrm{Deg} + \Wm \geq \alpha$, and let $f \in \ell^2(X,\mu) \cap \cF$ nonnegative with $\cH_{\mu, \Wm} f \leq \alpha f$. The definition of $\cH_{\mu,\Wm}$ then shows that for each $x \in X$ we have
 $$(\mathrm{Deg}(x) + \Wm(x) - \alpha) f(x) \leq \frac{1}{\mu(x)} \sum_{y \in X} b(x,y) f(y).$$
 Assume that there exists some $x_0 \in X$ with $f(x_0) > 0$. By the previous inequality there exists some $x_1 \in X$ with $x_1 \sim x_0$ and
 $$\frac{\mathrm{Deg}(x_0) + \Wm(x_0) - \alpha}{\mathrm{Deg}(x_0)} f(x_0) \leq  f(x_1).$$
Iterating this argument and using $\Deg + \Wm - \alpha \geq 0$ yields an infinite path $(x_n)$ such that for each $n \in \N$ we have 
$$f(x_n) \geq f(x_0)\prod_{j = 0}^{n-1} \left(1 + \frac{\Wm(x_j) - \alpha}{\mathrm{Deg}(x_j)}\right). $$
Since  \eqref{equation:divergent sum} holds, this inequality and $f(x_0) > 0$  contradict $f \in \ell^2(X,\mu)$.
%
%
%
%
\end{proof}

The assumption on the divergence of the sum in the previous theorem is a bit technical. For nonnegative endomorphisms it reduces to infinite paths having infinite measure, which is satisfied if $\inf_{x \in X} \mu(x) > 0$.

\begin{corollary}\label{corollary:measure space criterion}
 If $W \geq 0$ and every infinite path $(x_n)$ satisfies
 $$\sum_{n = 1}^\infty \mu(x_n) = \infty,$$
 then $\cM$ has exactly one realization. If, additionally, (FC) holds, then $M^{\mathrm{min}}$ is essentially self-adjoint.
\end{corollary}

\begin{remark}\begin{enumerate}
  \item The most important application of Corollary~\ref{corollary:measure space criterion} is when $\inf_{x \in X} \mu(x) > 0$. In this case, it shows   uniqueness of  realizations for all magnetic Schrödinger operators with nonnegative bundle endomorphism.  
  
  If one path has infinite measure, the whole space has infinite measure. Therefore, Corollary~\ref{corollary:measure space criterion} cannot be used to infer uniqueness of realizations or essential self-adjointness when $\mu(X) < \infty$. 
               \item 
The given abstract criterion for essential self-adjointness in Lemma~\ref{lemma:uniqueness abstract tool} and the basic idea for the proof of Theorem~\ref{theorem:measure space uniqueness} go back to \cite{Woj},  which treats graphs with weights $b \in \{0,1\}$ and the counting measure. In the presented form Corollary~\ref{corollary:measure space criterion} is taken from \cite{KL}, which treats Schrödinger operators with nonnegative potentials. Variants of Theorem~\ref{theorem:measure space uniqueness} are contained in  \cite{Gol2,GKS} for scalar magnetic Schrödinger operators. 

Compared to \cite{GKS}, our Theorem~\ref{theorem:measure space uniqueness} is a bit stronger because we also obtain the existence of realizations by means of Theorem~\ref{theorem:existence of realizations for magnetic operators}.

In \cite{Gol2} only locally finite graphs are considered, but essential self-adjointness of $M^\mathrm{min}$ is proven without assuming it to be lower semi-bounded. In this case (and more generally if (FC) holds), for establishing essential self-adjointness of $M^\mathrm{min}$ without assumptions on $W$ it suffices to prove that for some $\alpha \in \R, \gamma > 0$ all solutions to 
$$((M^\mathrm{min})^* + \alpha \pm \gamma i)f = (M^\mathrm{max} + \alpha \pm \gamma i)f = 0$$
satisfy $f = 0$, see e.g. \cite[Theorem~X.1]{RSII}. If  \eqref{equation:divergent sum} holds for all infinite paths, the vanishing of such solutions can be proven along the same lines as in \cite{Gol2}, where scalar magnetic Schrödinger operators are treated. Since our paper focuses on lower semi-bounded realizations, we refrain from giving details.

              \end{enumerate}

\end{remark}

\subsection{A metric space criterion}\label{subsection:a metric space criterion}

In this subsection we prove a criterion on uniqueness of realizations that is based on intrinsic metrics. The philosophy, which is inspired by corresponding results on manifolds, is the following. A magnetic Schrödinger operator can have unique realizations for two reasons: 1. The space has no boundary so that it is impossible to have different realizations from imposing different boundary conditions. 2. The space has a boundary but a strong growth of the potential (or the endomorphism) forces functions in the domain of the operator to vanish at the boundary. Also in this case boundary conditions can  not lead to different realizations. It turns out that a possible boundary to make this work is the Cauchy boundary with respect to an intrinsic metric. Theorem~\ref{theorem:metric space uniqueness} is a unified approach to both perspectives and Corollary~\ref{corollary:completeness} is a precise form of the first. The vanishing of the boundary (completeness) is replaced by balls with respect to an intrinsic metric being finite.  For path metrics on locally finite graphs this equivalent to completeness by a discrete version of the Hopf-Rinow theorem, see Proposition~\ref{proposition:properties of path metrics}.

A {\em pseudo metric} on $X$ is a symmetric function $\rho:X \times X \to [0,\infty)$ that vanishes on the diagonal and satisfies the triangle inequality. We let $\overline{X}^\rho$ be the completion of $X$ with respect to $\rho$ and $\partial_\rho X := \overline{X}^\rho \setminus X$ the corresponding {\em Cauchy boundary}.  By $D_\rho:X \to [0,\infty]$ we denote the distance to the boundary, i.e., 
$$D_\rho(x) := \rho(x, \partial_\rho X) := \inf\{ \rho(x,z) \mid z \in \partial_\rho X \}.$$
Here we use the convention $D_\rho =\infty$ if $\partial_\rho X = \emptyset$. Note that $D_\rho(x) > 0$ for all $x \in X$ if and only if $\partial_\rho X$ is closed in $\overline{X}^\rho$.

For a graph $(X,b)$ and a weight $\mu$ a pseudo metric $\rho$ on $X$ is called {\em intrinsic (with respect to $b$ and  $\mu$)} if 
$$\sum_{y \in X} b(x,y)\rho(x,y)^2 \leq \mu(x), \text{ for all } x \in X.$$
\begin{remark}
  For regular Dirichlet forms intrinsic metrics were introduced and systematically studied in \cite{FLW}. For graphs and other non-local operators related concepts, so-called adapted metrics, were independently introduced in \cite{Fol1,Hua,MU}. In recent years they have been used to solve several open problems in global analysis on graphs. We refer to the survey \cite{Kel} for a detailed discussion.
\end{remark}

   The possibility that $D_\rho$  is infinite is implicit in the statement of following theorem,  where by convention dividing by infinity yields zero.

\begin{theorem}\label{theorem:metric space uniqueness}
 Let $\rho$ be an intrinsic pseudo metric with the following properties.
 \begin{itemize}
  \item  $\partial_\rho X$ is closed in $\overline{X}^\rho$.
  \item For all $\varepsilon > 0$  all $\rho$-bounded subsets of  $\{x \in X \mid D_\rho(x) \geq \varepsilon\}$ are finite.
 \end{itemize}
 If $\Wm   \geq  \frac{1}{2 D_\rho^2} + V$ with $V \in \cA_{\mu}$, then $\cM$ has exactly one realization. If, additionally, (FC) holds, then $M^{\mathrm{min}}$ is essentially self-adjoint on $\ell^2(X,\mu; E)$.
\end{theorem}

\begin{proof}
We use Lemma~\ref{lemma:qc satisfies first beurling-deny criterion} to prove the statement. Since constant functions belong to $\cA_{\mu}$, the assumption implies that
$$\Wm \geq \frac{1}{2} \max \left\{1,   \frac{1}{D_\rho^{2}}\right\} + V $$
for some $V \in \cA_{\mu}$ (which is not the same as in the statement of the theorem). We consider the discrete Schrödinger operator $\cH':= \cH_{\mu, V'}$ with potential $V' := \frac{1}{2} \max\{1,   D_\rho^{-2}\} + V$. Using $V \in \cA_{\mu}$ we choose $\lambda \in \R$ such that $\lambda \|\varphi\|^2 \leq \cQ_V^c(\varphi)$ for all $\varphi \in C_c(X)$ and let $C: = \lambda -1$. Furthermore, we let $f \in \cF \cap \ell^2(X,\mu)$ nonnegative with $\cH' f \leq C f$. By Lemma~\ref{lemma:qc satisfies first beurling-deny criterion} it suffices to show $f = 0$.

For $\varphi \in C_c(X)$ the ground state transform (Lemma~\ref{lemma:ground state transform}) yields
\begin{align*}\cQ^c_{V'}(f \varphi) - C \| \varphi f\|^2 &\leq \frac{1}{2} \sum_{x,y \in X} b(x,y) f(x)f(y) |\varphi(x) - \varphi(y)|^2 \\
&\leq \frac{1}{2} \sum_{x \in X}f(x)^2 \sum_{y \in X} b(x,y)  |\varphi(x) - \varphi(y)|^2. 
\end{align*}
Moreover, from the choice of $C$ and the definition of $V'$ we obtain
\begin{align*} 
 \frac{1}{2} \sum_{x \in X} |\varphi(x) f(x)|^2 \max \left\{1, D_\rho(x)^{-2}\right\}\mu(x) + \|\varphi f\|^2\leq  \cQ^c_{V'}(f \varphi) - C \| \varphi f\|^2.
\end{align*}
We choose $\varphi$ to obtain the desired statement from these estimates. We fix a point $o \in X$. Let $0 < \varepsilon, R$ and set 
$$X_{\varepsilon, R} := \{x \in X \mid D_\rho(x) \geq \varepsilon \text{ and } \rho(o,x) \leq R\}.$$
By  assumption these sets are finite. We consider the piecewise affine functions $F: \R \cup\{\infty\} \to \R$, $F(t) = (t - \varepsilon)_+ \wedge 1$  and $G:\R_+ \to \R$, $G(t) = (2 - t/R)_+ \wedge 1$ and we define $\varphi:X \to \R$ by $\varphi(x) := F(D_\rho(x)) G(\rho(o,x)).$ It is straightforward that $\varphi$ is $\rho$-Lipschitz with Lipschitz constant $1 + 1/R$. Moreover, it is supported in $X_{\varepsilon, 2R}$, which is finite. Combining these observations with the above inequalities and using that $\rho$ is intrinsic we arrive at 
$$\frac{1}{2} \sum_{x \in X} |\varphi(x) f(x)|^2 \max \left\{1, D_\rho(x)^{-2}\right\}\mu(x) + \|\varphi f\|^2 \leq \frac{1}{2} \left(1 + \frac{1}{R} \right)^2 \|f\|^2. $$
For $x \in X_{\varepsilon, R}$ we have $\varphi(x) = (D_\rho(x) - \varepsilon) \wedge 1$ and hence
$$\max \left\{1, D_\rho(x)^{-2}\right\}|\varphi(x)|^2 \geq (1 - \varepsilon/D_\rho(x))^2.$$
This amounts to
$$\frac{1}{2} \sum_{x \in X_{\varepsilon, R}} (1 - \varepsilon/D_\rho(x))^2 |f(x)|^2 \mu(x) + \|\varphi f\|_2^2 \leq \frac{1}{2} \left(1 + \frac{1}{R} \right)^2 \|f\|^2_2.  $$
Now we let $\varepsilon \to 0+$ and $R \to \infty$. Since $\partial_\rho X$ is closed, we have $D_\rho(x) > 0$ for all $x \in X$, such that under these limits $X_{\varepsilon, R} \nearrow X$ and  $\varphi \to  D_\rho \wedge 1$ pointwise.  From this we obtain $\|(D_\rho \wedge 1) f\|_2 = 0$ and   we arrive at $f = 0$, which was to be proven. 
\end{proof}

\begin{corollary}\label{corollary:completeness}
  Let $\rho$ be an intrinsic pseudo metric. Assume that all $\rho$-balls are finite. If $\Wm \in  \cA_{\mu}$, then $\cM$ has exactly one realization. If, moreover, (FC) holds, then $M^\mathrm{min}$ is essentially self-adjoint.
\end{corollary}
\begin{proof}
 Finiteness of $\rho$-balls implies completeness of $(X,\rho)$ so that $\partial_\rho X = \emptyset$ and $D_\rho \equiv \infty$. With this at hand, the statement follows from the previous theorem. 
\end{proof}

\begin{remark}
\begin{enumerate}

 \item For path metrics (see below) on locally finite graphs finiteness of balls and completeness  coincide, see Proposition~\ref{proposition:properties of path metrics}.  Hence, Corollary~\ref{corollary:completeness} is a discrete version of a classical theorem of Strichartz on manifolds \cite{Str}, which says that on a complete Riemannian manifold Laplacians on functions, forms and tensors are essentially self-adjoint. For Schrödinger operators on open subsets of Euclidean space variants of Theorem~\ref{theorem:metric space uniqueness} are well known (with $D_\rho$   replaced by the distance to the topological boundary of the domain), see e.g. \cite{NN}.
 
 \item  For discrete Schrödinger operators with vanishing potential Corollary~\ref{corollary:completeness} was first proven in \cite{HKMW} and then extended to scalar magnetic Schrödinger operators in \cite{GKS}. Related but somewhat weaker results with additional assumptions on $\mu^{-1}\deg$ or particular metrics are contained in \cite{Tor,Mil}. For magnetic Schrödinger operators on bundles over locally finite graphs Theorem~\ref{theorem:metric space uniqueness} is proven in \cite{MT}, which was inspired by results on Schrödinger operators in \cite{TTV}. In \cite{MT} only path metrics (see below) are considered and the potentials need to be uniformly bounded from below.  Moreover, a somewhat stronger assumption than finiteness of bounded subsets of $\{x \in X \mid D_\rho(x) \geq \varepsilon\}$, called ``regularity of the graph'', is needed. The presented proof of Theorem~\ref{theorem:metric space uniqueness} is a simplified version of the one given in \cite{MT}. 

 \item If $\inf_{x \in X} \mu(x) > 0$ and $W \geq 0$, the criteria  in Subsection~\ref{subsection:a measure space criterion} always yield uniqueness of realizations while Theorem~\ref{theorem:measure space uniqueness} may not be applicable. For finite measures the situation is opposite. In this case, Theorem~\ref{theorem:measure space uniqueness} and its corollary may give uniqueness results while the criteria from Subsection~\ref{subsection:a measure space criterion} fail, see e.g. Example~\ref{example:essential self adjointness on Z} below. As a rule of thumb one can say that Theorem~\ref{theorem:metric space uniqueness} and its corollary are most interesting for the finite measure case. 
\end{enumerate}
\end{remark}

In the remainder of this section we put the assumptions on closedness of $\partial_\rho X$, finiteness of bounded subsets of $\{x \in X \mid D_\rho(x) \geq \varepsilon\}$ and finiteness of bounded subsets of $X$ into perspective. There are basically two kinds of (pseudo) metrics on $X$ for which this is possible; path metrics on locally finite graphs and  metrics  induced from embeddings of $X$ into Euclidean spaces (or more generally complete Riemannian manifolds). 

Let $\sigma: X \times X \to [0,\infty)$ be a symmetric function with $\sigma(x,y) > 0$ if and only if $x \sim y$. The length of a finite path $\gamma = (x_0,x_1,\ldots,x_n)$ with respect to $\sigma$ is defined by
$$L_\sigma(\gamma) := \sum_{i=1}^n \sigma(x_{i-1},x_i).$$
If the graph $(X,b)$ is connected, the associated {\em path pseudo metric $\rho_\sigma:X \times X \to [0,\infty)$} is defined by
$$\rho_\sigma(x,y) := \inf\{ L_\sigma(\gamma) \mid \gamma = (x_0,\ldots,x_n) \text{ is a path with } x_0 = x, x_n = y\}.$$
Pseudo metrics that arise in this way are called {\em path pseudo metrics}. One way to guarantee that   $\rho_\sigma$ is intrinsic (with respect to $b$ and $\mu$) is to demand that 
$$\sum_{y \in X} b(x,y) \sigma(x,y)^2\leq \mu(x), \quad \text{ for all } x \in X.$$
If  $\sigma$ satisfies this property the path pseudo metric $\rho_\sigma$ is called {\em strongly intrinsic}.  One edge weight for which this assumption holds is  $\sigma_H:X \times X \to [0,\infty)$ with $\sigma_H(x,y) := 0$ if $x \not \sim y$ and 
$$\sigma_H(x,y) :=  \min \left\{\frac{\mu(x)}{\deg(x)},\frac{\mu(y)}{\deg(y)} \right\}^{1/2}, \text{ if } x\sim y.$$
The corresponding path pseudo metric  $\rho_{\sigma_H}$ was introduced in \cite{Hua}.

The following lemma  characterizes some of the required properties to apply Theorem~\ref{theorem:metric space uniqueness} and Corollary~\ref{corollary:completeness} for path metrics on locally finite graphs.

\begin{proposition}\label{proposition:properties of path metrics}
Let $(X,b)$ a locally finite connected graph and let $\sigma:X \times X \to [0,\infty)$ symmetric with $\sigma(x,y) > 0$ if and only if $x \sim y$.
\begin{enumerate}[(a)]
 \item $\rho_\sigma$ is a metric on $X$ that induces the discrete topology  and $\partial_{\rho_\sigma} X$ is closed in $\overline{X}^{\rho_\sigma}$.
 \item The following assertions are equivalent.
 \begin{enumerate}[(i)]
  \item All $\rho_\sigma$-balls are finite.
  \item $(X,\rho_\sigma)$ is complete. 
 \end{enumerate}
\end{enumerate}
\end{proposition}
\begin{proof}

(a): For $x \in X$ we  let $\sigma_x := \inf\{\sigma(x,y) \mid y \sim x\}.$ Since $(X,b)$ is locally finite, it satisfies $\sigma_x > 0$. If $y \neq x$ we have $\rho_\sigma(x,y) \geq \sigma_x> 0$. Hence, $\rho_\sigma(x,y) = 0$ implies $x = y$ so that $\rho_\sigma$ is a metric.  Moreover, it follows from the definitions that any $\rho_\sigma$-ball of radius less than $\sigma_x$ around $x$ only contains $x$.  This shows that $\{x\}$ is open in $(X,\rho_\sigma)$ and in $(\overline{X}^{\rho_\sigma},\rho_\sigma)$. Thus, $\rho_\sigma$ induces the discrete topology on $X$ and $\partial_{\rho_\sigma} X$ is closed in $\overline{X}^{\rho_\sigma}$.

(b): This is contained in \cite[Theorem~A.1]{HKMW}.
%
\end{proof}

The assumption that bounded subsets of $\{x \in X \mid D_\rho(x) \geq \varepsilon\}$ are finite may or may not be satisfied for path  metrics $\rho$ on locally finite graphs. In \cite{TTVa,MT} it is claimed that all locally finite weighted trees (and more generally graphs of finite first Betti number) have this property. However, the following example shows that this is not true. It was communicated to us by Matthias Keller.
\begin{example}
 On $X: = \{(n,i) \mid n \in \N, i = 1,2\}$  let the graph $b:X \times X \to [0,\infty)$ be given by 
 $$b((n,i),(m,j)) = \begin{cases}
                    1 &\text{if } n = m \text{ and } i \neq j\\
                    \min\{n,m\}^{-2} &\text{if } |n-m| = 1 \text{ and } i = j = 0\\
                    0 &\text{else}
                    \end{cases}.
$$
W                             e consider the path metric $\rho_b$ that is induced by $b$. Then $(X,b)$ consists of the infinite path of finite length $((n,0))_{n \in \N}$ with an edge of length $1$ attached to each of the vertices in the path. In particular, $(X,b)$ is a tree. 

The diameter of $(X,\rho_b)$ is bounded by $1 + \frac{\pi^2}{6}$, hence any of its subsets is bounded. The Cauchy boundary is one point $\partial$ and the distance to it satisfies
$$D_{\rho_b}((n,i)) = \rho_b((n,i),\partial) = i +  \sum_{k = n}^\infty \frac{1}{k^2}.$$
This shows that for $0 < \varepsilon < 1$ we have $\{(n,1) \mid n \in \N\} \subseteq \{(n,i) \in X \mid D_{\rho_b}((n,i)) \geq \varepsilon\}$. Since $\{(n,1) \mid n \in \N\}$ is bounded, bounded subsets of  $\{(n,i) \in X \mid D_{\rho_b}((n,i)) \geq \varepsilon\}$ are not necessarily finite.
\end{example}

 We finish this section by discussing metrics that arise from embeddings into Euclidean spaces.  Let $\iota:X \to \R^n$ be an injective function.  We define the metric $d_\iota:X \times X \to [0,\infty)$ by $d_\iota(x,y) := |\iota(x) -\iota(y)|$. Then $\iota$ is an isometry from $(X,d_\iota)$ to $(\R^n,|\cdot |)$ that maps $X$ to $\iota(X)$.  It is readily verified that it extends uniquely to a surjective isometry $\hat \iota : (\overline{X}^{d_\iota},d_\iota) \to (\overline{\iota(X)},|\cdot|)$,  where $\overline{\iota(X)}$ is the closure of $\iota(X)$ in $\R^n$. Under this map the Cauchy boundary $\partial_{d_\iota} X$ is one-to-one with $\overline{\iota(X)} \setminus  \iota(X)$. In particular,
 $$D_{d_\iota}(x) = \inf \{|\iota(x) - a| \mid a \in \overline{\iota(X)} \setminus \iota(X)\}.$$
 We say that $\lim_{|x| \to \infty} \iota(x) = \infty$ if for every $R > 0$ there exists a finite $K \subseteq X$ such that $|\iota(x)| \geq R$ for all $x \in X \setminus K$.  Recall that $D \subseteq \R^n$ is called {\em discrete} if every $x \in D$ has an open neighborhood $U$ such that $U \cap D = \{x\}$. The following lemma summarizes properties of the metric space $(X,d_\iota)$, which are relevant for an application of Theorem~\ref{theorem:metric space uniqueness}.

\begin{proposition}\label{proposition:properties of metrics from embeddings}
 Let $\iota:X \to \R^d$ injective.

 \begin{enumerate}[(a)]
  \item The following assertions are equivalent.
  \begin{enumerate}[(i)]
   \item $\partial_{d_\iota} X$ is closed in $\overline{X}^{d_\iota}$.
   \item $\iota(X)$ is open in $\overline{\iota(X)}$.
  \end{enumerate}

  \item  The following assertions are equivalent. 
 \begin{enumerate}[(i)]
   \item $\iota(X)$ is discrete.
   \item $\partial_{d_\iota} X$ is closed and for every $\varepsilon > 0$ all $d_\iota$-bounded subsets of
    $$\{x \in X \mid D_{d_\iota}(x) \geq \varepsilon\} = \{x \in X \mid  |\iota(x) - a| \geq \varepsilon \text{ for all } a \in \overline{\iota(X)} \setminus  \iota(X)\}$$
 are finite.
 \end{enumerate}
 \item The following assertions are equivalent.
 \begin{enumerate}[(i)]
        \item All $d_\iota$-balls are finite.
        \item $\lim\limits_{|x| \to \infty} \iota(x) = \infty$.
        \item $\iota(X)$ is discrete and $(X,d_\iota)$ is complete.
       \end{enumerate} 
 \end{enumerate}
\end{proposition}
\begin{proof}
(a): (i) $\Leftrightarrow$ (ii):  As mentioned above, $\partial_{d_\iota} X$ and $\overline{\iota(X)} \setminus \iota(X)$ are one-to-one under the surjective isometry $\hat \iota:\overline{X}^{d_\iota} \to \overline{\iota(X)}$. Therefore, $\partial_{d_\iota} X$ is closed if and only if  $\overline{\iota(X)} \setminus \iota(X)$ is closed in $\overline{\iota(X)}$. This in turn is equivalent to  $\iota(X)$ being open in $\overline{\iota(X)}$.

(b): (i) $\Rightarrow$ (ii): The discreteness of $\iota(X)$ implies that singleton sets $\{\iota(x)\}$ are open in $\iota(X)$ in the relative topology. Hence, $\iota(X)$ is open in $\overline{\iota(X)}$ and (a) implies that $\partial_{d_\iota} X$ is closed. 

Let now $B$ be a $d_\iota$-bounded subset of $\{x \in X \mid D_{d_\iota}(x) \geq \varepsilon\}$.  By assumption $\iota(B)$ is a bounded discrete set in $\R^d$ that has positive distance from  $\overline{\iota(X)} \setminus  \iota(X)$. If $\iota(B)$ were infinite, it would contain a sequence $(a_n)$ of pairwise different points. Since $\iota(B)$ is bounded,  without loss of generality we can assume that $(a_n)$ converges to some point $a \in \overline{\iota(B)}$.    However, since $\iota(X)$ is discrete, we also have $a \in \R^d \setminus \iota(X)$.  This contradicts the fact that $\iota(B)$ has positive distance to $\overline{\iota(X)} \setminus  \iota(X)$.

%
%
%
%

(ii) $\Rightarrow$ (i): Suppose that $\iota(X)$ is not discrete. Then there exists  $o \in X$ such that every Euclidean ball around $f(o)$ contains infinitely many elements of $\iota(X)$.  According to (a) our assumption implies that $\iota(X)$ is open in $\overline{\iota(X)}$. Hence,  there exists an $\varepsilon > 0$ such that  $|\iota(o) - a| \geq 2\varepsilon$ for every $a \in \overline{\iota(X)} \setminus \iota(X)$. Now consider the $d_\iota$-bounded set
$$B:= \{x \in X \mid d_\iota(o,x) \leq \varepsilon\} = \{x \in X \mid |\iota(o) - \iota(x)| \leq \varepsilon \}.$$
By the choice of $o$ the set $B$ is infinite. Moreover, $x \in B$ satisfy
$$|\iota(x) - a| \geq |\iota(o) - a| - |\iota(x) - \iota(o)| \geq \varepsilon$$
for every $a \in \overline{\iota(X)} \setminus \iota(X)$ so that $B \subseteq \{x \in X \mid D_{d_\iota}(x) \geq \varepsilon\}$.  This contradicts (ii).

(c): (i) $\Leftrightarrow$ (ii): This is straightforward from the definitions.

(i) $\Rightarrow$ (iii): Completeness follows from the fact that Cauchy sequences are bounded and finite sets are compact. The discreteness of $\iota(X)$ follows from (b).

(iii) $\Rightarrow$ (ii): Let $o \in X$ and for $r > 0$ let $B_r(o) := \{x \in X \mid d_\iota(o,x) \leq r\}$. Then  $\iota(B_r(o))$ is obviously bounded in $\R^d$. Since $(X,d_\iota)$ is complete, it is also complete and hence even compact in $\R^d$. Thus, if $\iota(B_r(o))$ were not finite, there would be an infinite sequence $(a_n)$ of pairwise different elements of  $\iota(B_r(o))$ that converges to some $a \in \overline{\iota(B_r(o))} = \iota(B_r(o))$. This however contradicts the discreteness of $\iota(X)$. 
\end{proof}

\begin{remark}
 \begin{enumerate}
  \item The injectivity of $\iota$ is not so important. It would have been possible to deal with functions for which the preimages of singleton sets are finite. In this case, the resulting distance functions are only pseudo metrics and not metrics.
  
 \item In view of Proposition~\ref{proposition:properties of path metrics}~(b) one could ask whether finiteness of $d_\iota$ balls alone is equivalent to completeness of $(X,d_\iota)$. This is not the case. Consider e.g. $X = \N_0$ and $\iota:\N \to \R$ given by $\iota(0) = 2$ and $\iota(n) = 2 - 1/n$, $n \in \N$. Then $(\N,d_\iota)$ is complete but $\N$ is bounded with respect to $d_\iota$. In particular, not all $d_\iota$-balls are finite.
 
  \item Instead of functions with values in Euclidean spaces we could have considered functions with values in complete metric spaces whose bounded sets are precompact. For example, this is the case for complete Riemannian manifolds.
 \end{enumerate}
\end{remark}

For a given graph $(X,b)$ without isolated vertices and an injective  function $\iota:X \to \R^n$ there is a smallest weight $\mu_\iota:X \to (0,\infty)$ such that $d_\iota$ is intrinsic with respect to $b$ and $\mu_\iota$. It is given by
$$\mu_\iota(x) = \sum_{y \in X} b(x,y)|\iota(x) - \iota(y)|^2$$
and has the property that for all weights $\mu:X \to (0,\infty)$ that satisfy $\mu \geq \mu_\iota$ the metric $d_\iota$ is intrinsic with respect to $b$ and $\mu$. Even though this observation is elementary, it allows us to construct many interesting examples of graphs with finite measures for which Theorem~\ref{theorem:metric space uniqueness} can be applied. We finish this subsection with two such examples.

The following example is taken from \cite{Puc}, which is based on \cite{BS}.

\begin{example}[Nerves of circle packings]
A {\em circle packing} in $\R^2$ (or $\mathbb C$) is a collection of circles $\cC =  (C_j)_{j \in J}$, with $C_j = \{x \in \R^2 \mid |x - x_j| = r_j\}$ for some $x_j \in \R^2$ and $r_j > 0$, such that for $i \neq j$ the interiors of the circles $C_i$ and $C_j$ do not intersect.  We say that $\cC$ is {\em bounded} if $\cup_{j \in J} C_j$ is bounded in  $\R^2$.  The {\em nerve} (or contact graph) of a circle packing $\cC = (C_j)_{j \in J}$ is the graph $(X_\cC,b_\cC)$, where $X_\cC = \{x_j \mid j \in J\}$ and 
$$b_\cC(x_i,x_j)  =\begin{cases}
                1 &\text{if } S_j \cap S_i \neq \emptyset\\
                0 &\text{else}
               \end{cases}.
$$
The graph $(X_\cC,b_\cC)$ satisfies the condition (b2) if and only if it is locally finite, i.e., $\deg(x_i) = \# \{j \in J \mid C_i \cap C_j \neq \emptyset\} < \infty$ for all $x_i \in X_\cC$. We say that $\cC$ is {\em connected} if $(X_\cC,b_\cC)$ is a connected graph.

Let $\iota:X_\cC \to \R^2$ the identity, i.e., $\iota(x_j) = x_j$ for $j\in J$. Then $d_\iota(x_i,x_j) = |x_i - x_j| = r_i + r_j$ so that $\iota(X_\cC) = X_\cC$ is discrete. If $\cC$ is bounded and connected and the function $\deg$ is bounded, then the measure $\mu_\iota$ is finite. Indeed, if $K$ is an upper bound for $\deg$ we obtain
\begin{align*}
 \mu_\iota(X) &=  \sum_{x \in X_\cC} \sum_{y \in X_\cC} b_\cC(x,y)|\iota(x) - \iota(y)|^2\\
 &= \sum_{j \in J} \sum_{i \in J : S_j \cap S_i \neq \emptyset} (r_i + r_j)^2 \\
 &\leq 4K \sum_{j \in J} r^2_j.
\end{align*}
Since $\pi r_j^2$ equals the area of the circle $C_j$ and different circles have disjoint interiors, the above computation yields $ \mu_\iota(X) \leq \frac{4K}{\pi} \mathrm{Area}(\cC)$. Here, $\mathrm{Area}(\cC)$ is the area covered by the interiors of the circles in $\cC$. This shows that for the nerve of a bounded connected circle packing with bounded degree the Euclidean metric is an intrinsic metric with respect to a finite measure. Moreover, it satisfies the assumptions of Theorem~\ref{theorem:metric space uniqueness}.
\end{example}

 The following example shows how perturbation by a large potential forces a non essentially self-adjoint operator to become essentially self-adjoint. It also shows that for finite measures Theorem~\ref{theorem:metric space uniqueness} is stronger than Theorem~\ref{theorem:measure space uniqueness}. The first part of the discussion without the potential is taken from \cite{HKMW}.
 
\begin{example} \label{example:essential self adjointness on Z}
 Consider the graph $(\Z,b)$ with $b(k,l) = 1$ if $|k-l| = 1$. For a  weight $\mu:\Z \to (0,\infty)$ the associated formal Schrödinger operator $\cH_{\mu, 0}$ acts on $C(\Z)$ by
 $$\cH_{\mu, 0} f(k) = \frac{1}{\mu(k)} \left( 2f(k) - f(k+1) - f(k-1) \right), \quad k \in \Z.$$
 Consider the weight $\nu:\Z \to (0,\infty)$, $\nu(k) = 2k^{-4}$ for $k \neq 0$ and $\nu(0) = 2$. Then $H^\mathrm{min}_{\nu,0}$  is not essentially self-adjoint. This can be seen as follows. The function $h:\Z \to \R$ with $h(k) = k$ belongs to $\ell^2(X,\nu)$ and satisfies $\cH_{\nu,0}h = 0$. Hence,  $ h \in D(H^\mathrm{max}_{\nu,0})$. However, $h$ has infinite energy because
 $$\sum_{k,l \in \Z} b(k,l) (h(k) - h(l))^2 = \sum_{k \in \Z} 2 = \infty.$$
 Since all functions in  $D(H^0_{\nu,0})$ have finite energy, see Proposition~\ref{proposition:nonnegative endomorphisms}, this implies $h \not \in D(H^0_{\nu,0})$. If $H^\mathrm{min}_{\nu,0}$ were essentially self-adjoint, all self-adjoint extensions of $H^\mathrm{min}_{\nu,0}$ would coincide with $(H^\mathrm{min}_{\nu,0})^* = H^\mathrm{max}_{\nu,0}$. The previous discussion shows that this is not the case.
 
 Let now $\iota: \Z \to \R$ given by $\iota(k) = 2 - 1/k$ if $k \neq 0$ and $\iota(0) = 0$. Then $\iota(\Z)$ is discrete in $\R$. The Cauchy boundary of $\Z$ with respect to $d_\iota$ consists of exactly one point $\partial$ and the isometric extension of $\iota$ to the completions is given by $\hat \iota: \Z \cup \{ \partial\} \to \overline{\iota(\Z)} = \iota(\Z) \cup\{2\}$ with $\hat \iota (k)  = \iota(k)$ for $k \in \Z$ and $\hat \iota (\partial) = 2$. We obtain that 
 $$D_{d_\iota}(k) = |\iota(k) - 2| = \begin{cases} 1/|k| &\text{for } k \neq 0\\ 2 &\text{for } k = 0 \end{cases}. $$
 Moreover, the metric $d_\iota$ is intrinsic with respect to $\nu$, since 
 $$\mu_\iota(k) = \sum_{l \in \Z} b(k,l) d_\iota(k,l)^2 = \sum_{l \in \Z} b(k,l) |\iota(k) - \iota(l)|^2 = \begin{cases}
                                                                                 \frac{2}{k^2(k^2 -1)}  &\text{if } |k| \geq 2\\
                                                                                 \frac{5}{4} &\text{if } |k| = 1\\
                                                                                 2 &\text{if } k = 0
                                                                                \end{cases}.
 $$
 Then Theorem~\ref{theorem:metric space uniqueness} and Proposition~\ref{proposition:properties of metrics from embeddings} imply that for any $V:\Z \to \R$ with $V(k) \geq k^2/2 $ the operator $H^\mathrm{min}_{\nu,V}$ is essentially self-adjoint $\ell^2(\Z,\nu).$ Moreover, for the case $V:\Z \to \R$,  $V(k) = k^2/2 $ the assumptions of Theorem~\ref{theorem:measure space uniqueness} are not satisfied, since for all $\alpha \in \R$  the infinite path $(x_n)_{n \geq 0} = (n)_{n \geq 0}$ satisfies
 $$\sum_{n  = 1}^\infty \nu(n) \prod_{k = 0}^{n-1} \left(1 + \frac{\nu(k)(k^2/2 - \alpha)}{2}\right)^2 = 2 \sum_{n  = 1}^\infty \frac{1}{n^4} \prod_{k = 0}^{n-1} \left(1 + \frac{(k^2/2 - \alpha)}{2k^4}\right)^2< \infty. $$
\end{example}

\section{Markovian realizations of $\cH$ and Markov uniqueness} \label{section:Markovian realizations}
In this section we study realizations of $\cH$ whose associated quadratic forms are Dirichlet forms. The semigroups generated by these realizations are Markovian and therefore correspond to Markov processes on $X$ through the Feynman-Kac formula.  Such realizations on possibly non locally finite graphs have received quite some attention in recent years, see e.g.  \cite{KL2,KL,HKLW,Schmi} and references therein. Here we discuss their basic structure and give criteria for their uniqueness. It turns out that if the potential is nonnegative, there always is a minimal and a maximal Markovian realization in the sense of quadratic forms, see Theorem~\ref{theorem:structure of Markovian realizations}. Since Markovian realizations are special realizations, theorems that ensure their uniqueness tend to have weaker assumptions than the ones for essential self-adjointness or uniqueness of realizations, which we discussed in Section~\ref{section:uniqueness}. This is the content of Subsection~\ref{subsection:markov uniqueness}.

\subsection{The structure of Markovian realizations}\label{subsection:structure of Markovian realizations}

 In this subsection we construct a minimal and a maximal Markovian realization of $\cH$ and show that all other Markovian realization lie between them in the sense of quadratic forms. 

For the definition and basic properties of Dirichlet forms we refer to Appendix~\ref{appendix:Beruling Deny}. Since Dirichlet forms are real quadratic forms, for the purpose of this section it suffices to consider real-valued functions.  We use the following convention. 

\begin{convention}
 In this whole section all functions are real-valued. In particular, we abuse notation and denote by $C(X)$ the real-valued functions on $X$, by $C_c(X)$ the real-valued functions on $X$ with finite support and by $\ell^2(X,\mu)$  the Hilbert space of real-valued square summable functions.
\end{convention}

\begin{definition}[Markovian realization]
 A realization  of $\cH$ is called {\em Markovian} if the associated quadratic form is a Dirichlet form.
\end{definition}

The following lemma shows that for studying Markovian realizations it suffices to consider the case  $V \geq 0$.

\begin{lemma}
If $\cH$ has a Markovian realization, then $V \geq 0$.
\end{lemma}
\begin{proof}
 Let $H$ be a Markovian realization of $\cH$ with associated quadratic form $Q$. It follows from Lemma~\ref{lemma:q and extension of qc} that $Q$ is an extension of $\cQ^c$. This implies  $\cQ^c(\varphi \wedge 1) \leq \cQ^c(\varphi)$ for all $\varphi  \in C_c(X)$. Let $K \subseteq X$ finite and let $x \in K$.  Lemma~\ref{lemma:postivity of form} shows $ \cQ^c(1_K,\delta_x) \geq 0$ and, therefore, 
 $$0 \leq \cQ^c(1_K,\delta_x) = \sum_{y \in X \setminus K} b(x,y)  + V(x).$$
 Letting $K \nearrow X$ yields $V \geq 0$. 
\end{proof}

Recall the definition of $\cQ$ and $\cD$ from Subsection~\ref{subsection:Schrödinger forms}. The following well-known proposition gives the two most prominent Markovian realizations of $\cH$ when $V \geq 0$, see e.g. \cite{KL,Schmi}.

\begin{proposition}[Existence of Markovian realizations]
 Let $V \geq 0$. Then $Q^0$ and the restriction of $\cQ$ to $\mathcal{D} \cap \ell^2(X,\mu)$ are Dirichlet forms and the associated operators are Markovian realizations of $\cH$.
\end{proposition}

\begin{proof}
 Let $Q^{(N)}$ denote the restriction of $\cQ$ to $\mathcal{D} \cap \ell^2(X,\mu)$. We first prove that $Q^{(N)}$ is a Dirichlet form. Its lower semicontinuity with respect to $\ell^2(X,\mu)$-convergence follows from Fatou's lemma. Hence, it is closed by Lemma~\ref{lemma:characterization closedness}. The definition of $\cQ$ and $V \geq 0$ imply that for $f \in \mathcal{D}$ and any normal contraction $C:\R \to \R$ we have $C \circ f \in \cD$ and $\cQ(C \circ f) \leq \cQ(f)$. Since $f \in \ell^2(X,\mu)$ also implies $C \circ f \in \ell^2(X,\mu)$, we obtain that $Q^{(N)}$ is a Dirichlet form. Let $H^{(N)}$ be the associated self-adjoint operator. For $f \in D(H^{(N)})$ and $\varphi \in C_c(X) \subseteq D(Q^{(N)})$ Green's formula (Lemma~\ref{lemma:greens formula forms}) shows
 $$\langle \varphi, H^{(N)}f \rangle_2 = Q^{(N)}(\varphi, f) = \cQ(\varphi,f) = (\varphi, \cH f).$$
 This implies that $H^{(N)}$ is a realization of $\cH$.
 
 We have already seen in Proposition~\ref{proposition:nonnegative endomorphisms} that $H^0$ is a realization of $\cH$. It remains to show that $Q^0$ is a Dirichlet form. To this end, let $f \in D(Q^0)$ and choose a sequence $(\varphi_n)$ in $C_c(X)$ that converges to $f$ with respect to the form norm.  For a normal contraction $C:\R \to \R$  the sequence $(C \circ \varphi_n)$ belongs to $C_c(X)$ and converges in $\ell^2(X,\mu)$ to $f$. Therefore, the lower semicontinuity of $Q^0$ and its definition yield
 $$Q^0(C \circ f) \leq  \liminf_{n \to \infty} Q^0( C \circ \varphi_n) =  \liminf_{n \to \infty} \cQ^c( C \circ \varphi_n) \leq  \liminf_{n \to \infty} \cQ^c(  \varphi_n) = Q^0(f).$$
 This shows that $Q^0$ satisfies the first Beurling-Deny criterion and finishes the proof.
 \end{proof}

\begin{definition}[Neumann realization]
Let $V \geq 0$. The Dirichlet form on $\ell^2(X,\mu)$ that is the restriction of $\cQ$ to the domain $\mathcal{D}\cap \ell^2(X,\mu)$ is denoted by $Q^{(N)}  = Q^{(N)}_{\mu, V}$. The associated self-adjoint operator is called $H^{(N)} = H^{(N)}_{\mu, V}$.
\end{definition}

\begin{remark}
 We already noted that $H^0$ can be thought of being the realization of $\cH$ with ``Dirichlet boundary conditions at infinity''.  As the notation suggests, $H^{(N)}$ should be thought of as the realization with ``Neumann boundary conditions at infinity''. On appropriate compactifications of $X$ this intuition can be made precise, see \cite{KLSS, HS}.
\end{remark}

Markovian realizations are ordered in terms of their quadratic forms, cf. Appendix~\ref{appendix:basics}.  The following theorem shows that in this sense $H^0$ is the minimal Markovian realization and $H^{(N)}$ is the maximal Markovian realization.

\begin{theorem}[Structure of Markovian realizations]\label{theorem:structure of Markovian realizations}
 Let $H$ be a Markovian realization of $\cH$ with associated Dirichlet form $Q$. Then $V \geq 0$ and $Q^0 \leq Q \leq Q^{(N)}$.  In particular, $Q$ is an extension of $Q^0$.
\end{theorem}

\begin{corollary}[Markov uniqueness]\label{corollary:markov uniqueness}
 The operator $\cH$ has exactly one Markovian realization if and only if $V \geq 0$ and $Q^0 = Q^{(N)}$. 
\end{corollary}

\begin{remark}
 \begin{enumerate}
 \item We would like to stress that in our convention $Q^0 \leq Q \leq Q^{(N)}$ means $D(Q^0) \subseteq D(Q) \subseteq D(Q^{(N)})$ and $Q^0(f) \geq Q(f) \geq Q^{(N)}(f)$ for all $f \in \ell^2(X,\mu)$ (with the forms being equal $\infty$ outside their domains). It emphasizes the size of the form domain; forms with large domains and small values are large in the sense of this ordering. This convention seems to be standard in Dirichlet form theory, see e.g. \cite{FOT,CF}.

  \item For locally finite graphs this theorem and its corollary are contained in \cite{HKLW}. It is new for general graphs. The proof given below is based on results of \cite{Schmi3}.
  \item Since the space is discrete, any Dirichlet form $Q$ that extends $Q^0$ is a Silverstein extension, i.e.,   for $f \in D(Q^0) \cap \ell^\infty(X)$ and $g \in D(Q)\cap \ell^\infty(X)$ we have $fg \in D(Q^0) \cap \ell^\infty(X)$.   This is clear whenever $f \in C_c(X)$, the case of general $f$ follows by an approximation argument. 
  
  It also follows from abstract results in Dirichlet form theory that when $V = 0$ the form $Q^{(N)}$ is the maximal (Silverstein) extension of $Q^0$, see e.g. \cite[Theorem~6.6.9]{CF}. In the case $V \neq 0$ the abstract results contained in the literature are wrong; they claim that any (Silverstein) extension $Q$ of  $Q^0$ satisfies $Q \leq Q^{(N)}$. However, in general there are extensions of $Q^0$ for which $Q^0 \leq Q^{(N)}$ does no hold; for an example see \cite[Section~5]{KLSS}. Note that the class of  Dirichlet forms treated by Theorem~\ref{theorem:structure of Markovian realizations} is a bit smaller. We only consider extensions $Q$ of $Q^0$ whose associated operator is a Markovian realization of $\cH$.

  \item Dirichlet forms $Q$ on $\ell^2(X,\mu)$ with $Q^0 \leq Q \leq Q^{(N)}$ are parametrized by certain Dirichlet forms on the Royden boundary of the graph. This is discussed in \cite{KLSS}. 
 \end{enumerate}

\end{remark}

Let $H$ be a Markovian realization of $\cH$ and let $Q$ be the associated Dirichlet form. For $f \in D(Q) \cap \ell^\infty(X)$ and $\varphi \in D(Q)$ with $0 \leq \varphi \leq 1$  we define the concatenated form
$$Q_\varphi (f) := Q(\varphi f) - Q(\varphi f^2, \varphi).$$
%
Since $D(Q) \cap \ell^\infty(X)$ is an algebra, see e.g. \cite[Theorem~1.4.2]{FOT}, this is well-defined. For $f \in D(Q)\cap \ell^\infty(X)$  we define the {\em main part of $Q$} by 
$$Q^m(f) := \sup \{Q_\varphi(f) \mid \varphi  \in D(Q) \text{ with } 0 \leq \varphi \leq 1\},$$
and the {\em killing part of $Q$} by
$$Q^k(f) := Q(f) - Q^m(f).$$
The following lemma shows that both are well-defined. It is a special case of the theory developed in \cite[Chapter~3]{Schmi3}.
\begin{lemma}\label{lemma:fundamental lemma structure of markovian restrictions}
Let $f,g\in D(Q) \cap \ell^\infty(X)$ and let  $\varphi,\psi \in D(Q) $ with $0 \leq \varphi,\psi \leq 1$. Then the following holds.
\begin{enumerate}[(a)]
 \item $0 \leq Q_\varphi(f) \leq Q^m(f) \leq Q(f)$ and $0 \leq Q^k(f) \leq Q(f)$.
 \item $\varphi \leq \psi$ implies $Q_\varphi(f) \leq Q_{\psi}(f)$.
 \item $|f| \leq |g|$ implies  $Q^k(f) \leq Q^k(g)$.
\end{enumerate}
\end{lemma}

\begin{proof}

 (a) \& (b): We show the inequalities $0 \leq Q_\varphi(f) \leq Q_\psi(f) \leq Q(f)$, the rest then follows from the definitions. We denote by $G_\alpha := (H + \alpha)^{-1}$ the resolvents and we use the approximating forms, see Appendix~\ref{appendix:basics}, to compute
 \begin{align*}
  Q_\varphi(f) &= \lim_{\alpha \to \infty} \alpha ( \as{(I -\alpha G_\alpha) (\varphi f),\varphi f}_2 - \as{(I - \alpha G_\alpha)(\varphi f^2), \varphi}_2 )\\
  &= \lim_{\alpha \to \infty} \alpha (  \as{ \alpha G_\alpha (\varphi f^2), \varphi}_2 - \as{\alpha G_\alpha (\varphi f) ,\varphi f}_2  )
 \end{align*}
 and
 $$ Q(f) = \lim_{\alpha \to \infty} \alpha \as{(I -\alpha G_\alpha) f ,f}_2.$$  
 Since the involved quantities are continuous in $f$, it suffices to prove 
 $$Q^\alpha_\varphi(f) := \as{ \alpha G_\alpha (\varphi f^2), \varphi}_2 - \as{\alpha G_\alpha (\varphi f) ,\varphi f}_2 \leq  \as{(I -\alpha G_\alpha) f ,f}_2 =: Q^\alpha(f)$$
 and the monotonicity in $\varphi$ for $f \in C_c(X)$.  Let $S$ be the finite support of $f$. An elementary computation shows
 $$Q_\varphi^\alpha(f)  = \sum_{x,y \in S} b^\alpha_\varphi(x,y) (f(x) - f(y))^2 + \sum_{x \in S} c_\varphi^\alpha(x) f(x)^2,  $$
 where for $x,y \in S$ the coefficients satisfy
 $$b^\alpha_\varphi(x,y) \quad  =  \quad \begin{cases} - Q_\varphi^\alpha(\delta_x,\delta_y) &\text{if } x \neq y\\
                                          0 &\text{if }x = y
                                         \end{cases}
                                     \quad    = \quad
 \begin{cases}
                                          \as{\alpha G_\alpha (\varphi\delta_x),\varphi \delta_y}_2 &\text{if } x \neq y\\
                                          0 &\text{if }x = y
                                         \end{cases},
$$
and
$$c_\varphi^\alpha(x) = Q^\alpha_\varphi(1_S, \delta_x) =  \as{\alpha G_\alpha ( 1_{X\setminus S} \varphi), \varphi \delta_x}_2. 
$$
 Similarly, we obtain 
  $$Q^\alpha(f)  =  \sum_{x,y \in S} b^\alpha(x,y) (f(x) - f(y))^2 + \sum_{x \in S} c^\alpha(x) f(x)^2, $$
 where for $x,y \in S$ the coefficients satisfy
 $$b^\alpha(x,y) \quad  =  \quad \begin{cases} - Q^\alpha(\delta_x,\delta_y) &\text{if } x \neq y\\
                                          0 &\text{if }x = y
                                         \end{cases}
                                     \quad    = \quad
 \begin{cases}
                                          \as{\alpha G_\alpha\delta_x,\delta y }_2 &\text{if }  x \neq y\\
                                          0 &\text{else}
                                         \end{cases},
$$
and
$$c^\alpha(x) = Q^\alpha(1_S, \delta_x) =  \as{1_S - \alpha G_\alpha 1_S, \delta_x}_2.  
 $$
 Since resolvents of Dirichlet forms are positivity preserving, these computations immediately yield $0 \leq b_\varphi^\alpha \leq b_\psi^\alpha \leq b^\alpha$ and $0 \leq c_\varphi^\alpha \leq c_\psi^\alpha$. For $x \in S$ we obtain
 \begin{align*}
  c^\alpha(x) - c_\varphi^\alpha(x) &= \as{1_S - \alpha G_\alpha 1_S,\delta_x}_2 - \as{\alpha G_\alpha ( 1_{X \setminus S} \varphi),\varphi \delta_x}_2\\
  &\geq \as{1_S - \alpha G_\alpha 1_S,\delta_x}_2 - \as{\alpha G_\alpha ( 1_{X \setminus S} \varphi), \delta_x}_2\\
  &=  \as{1_S - \alpha G_\alpha (1_S + \varphi 1_{X \setminus S}),\delta_x}_2.
 \end{align*}
 Since $\alpha G_\alpha$ is Markovian and $0 \leq \varphi \leq 1$, we have $\alpha G_\alpha (1_S + \varphi 1_{X \setminus S}) \leq 1$. Therefore, the above computation shows $c^\alpha_\varphi(x) \leq c^\alpha(x)$ for $x \in S$ and the claims are proven.
 
 (c): We prove this statement in two steps. First, we additionally assume that there exists $\psi \in D(Q)$ such that $1_{\{|g| > 0\}} \leq \psi \leq 1$. Let $\varepsilon > 0$. We use the definition of $Q^k$ and that $Q_\varphi$ is monotone in $\varphi$ to choose $\varphi \in D(Q)$ with $\psi \leq \varphi \leq 1$ such that 
 $$|Q^m(f) - Q_\varphi(f)| \leq \varepsilon \text{ and } |Q^m(g) - Q_\varphi(g)| \leq \varepsilon. $$
 Note that $\varphi$ equals one on the supports of $f$ and $g$. We obtain
 \begin{align*}
  Q^k(g) - Q^k(f) &\geq  Q(g) - Q_\varphi(g) - Q(f) + Q_\varphi(f) - 2\varepsilon \\
  &= Q(\varphi g) - Q_\varphi(g) - Q(\varphi f) + Q_\varphi(f) - 2 \varepsilon\\
  &= Q(\varphi g^2,\varphi) - Q(\varphi f^2,\varphi)\\
  &= Q(g^2 - f^2,\varphi) - 2 \varepsilon.\\
 \end{align*}
 Since $g^2 - f^2$ is nonnegative and  $\varphi$ equals one on the support of $g^2 - f^2$, Lemma~\ref{lemma:postivity of form} yields $Q(g^2 - f^2,\varphi) \geq 0$ and the claim is proven for these special $f$ and $g$.
 
 Let now $f,g \in D(Q) \cap \ell^\infty(X)$ with $|f| \leq |g|$ arbitrary. For $\alpha > 0$ consider the functions $f_{\alpha} := f - (f \wedge \alpha) \vee (-\alpha)$ and $g_{\alpha} := g - (g \wedge \alpha) \vee (-\alpha)$. They satisfy $|f_\alpha| \leq |g_\alpha|$ and \cite[Theorem~1.4.2]{FOT} shows $f_\alpha,g_\alpha \in D(Q)$ and $g_{\alpha} \to g$ and $f_\alpha \to f$ with respect to the form norm, as $\alpha \to 0+$.   Since $Q$ is a Dirichlet form, the function 
 $\psi_\alpha :=  (\alpha^{-1} |g|) \wedge 1$   
 belongs to $D(Q)$ and since  $\{|g_\alpha| > 0\} = \{|g| > \alpha\}$, it equals one on the support of $g_\alpha$. We can therefore apply the already proven inequality to the functions $f_\alpha$ and $g_\alpha$. Moreover, the quadratic form $Q^k$ is smaller than $Q$ and therefore continuous with respect to $Q$-convergence. With all of these properties we conclude
 $$Q^k(f) = \lim_{\alpha \to 0+} Q^k(f_\alpha) \leq \lim_{\alpha \to 0+} Q^k(g_\alpha) = Q^k(g). $$
 This finishes the proof.
\end{proof}

%
%
%

For the  following proof recall that  $\cQ_0$ is the Schrödinger form with respect to the potential $V = 0$.

\begin{proof}[Proof of Theorem~\ref{theorem:structure of Markovian realizations}]

 By definition we have $Q = Q^m + Q^k$ on $D(Q) \cap \ell^\infty(X)$ and $Q^{(N)} = \cQ_0 + q_V$ on $D(Q^{(N)})$. Since bounded functions are dense in the domains of Dirichlet forms, see e.g. \cite[Theorem~1.4.2]{FOT}, it suffices to prove $\cQ_0(f) \leq Q^m(f)$ and $q_V(f) \leq Q^k(f)$ for $f \in D(Q) \cap \ell^\infty(X)$.

  Let $K \subseteq X$  finite and let  $f \in D(Q) \cap \ell^\infty(X)$. By Lemma~\ref{lemma:q and extension of qc} the form $Q$ is an extension of $\cQ^c$ and  $1_K$, $1_K f$ and $1_K f^2$ have finite support. Therefore, Lemma~\ref{lemma:fundamental lemma structure of markovian restrictions} yields
 $$Q^m(f) \geq Q_{1_K}(f) =   \cQ^c(1_K f) - \cQ^c(1_K f^2, 1_K) = \frac{1}{2} \sum_{x,y \in K} b(x,y) (f(x) - f(y))^2.$$
  Letting $K \nearrow X$ yields $Q^m(f) \geq \cQ_0(f)$ for all $f \in D(Q) \cap \ell^\infty(X)$.

 It remains to prove the inequality $Q^k(f) \geq q_V(f)$ for $f \in D(Q) \cap \ell^\infty(X)$. To this end, let $K \subseteq X$ finite and let $\varepsilon > 0$. According to Lemma~\ref{lemma:fundamental lemma structure of markovian restrictions}, we can choose $\psi \in D(Q)$ with $1_K \leq \psi \leq 1$ such that for each $\varphi \in D(Q)$ with $\psi \leq \varphi \leq 1$ we have $Q^k(1_K f) \geq Q(1_K f) - Q_{\varphi} (1_K f) - \varepsilon$. The monotonicity of $Q^k$ (Lemma~\ref{lemma:fundamental lemma structure of markovian restrictions}~(c)) then implies
 $$Q^k(f) \geq Q^k(1_K f) \geq Q(1_K f) - Q_{\varphi} (1_K f) - \varepsilon = Q(1_K f^2,\varphi) - \varepsilon. $$
 Since $0 \leq \varphi \leq 1$ is bounded and the resolvent $G_{\alpha} : = (H + \alpha)^{-1}$ is Markovian, we have $0 \leq \alpha G_\alpha \varphi \leq 1$. It follows from Lebesgue's dominated convergence theorem that $\cH \alpha G_\alpha \varphi \to \cH \varphi$ pointwise, as $\alpha \to \infty.$ Moreover,  $\cH$ is a realization of $H$ and $1_Kf^2$ has finite support so that we obtain
 $$Q(1_K f^2,\varphi) = \lim_{\alpha \to \infty} Q(1_K f^2,\alpha G_\alpha \varphi) = \lim_{\alpha \to \infty}  (1_K f^2, \cH \alpha G_\alpha \varphi) = (1_K f^2, \cH \varphi). $$
 Combining these computations we arrive at 
 $$Q^k(f) \geq (1_K f^2, \cH \varphi) - \varepsilon,$$
 whenever $\varphi \in D(Q)$ with $\psi \leq \varphi \leq 1$. Letting $\varphi \nearrow 1$ pointwise  and using Lebesgue's dominated convergence theorem yields $\cH \varphi \to V$ pointwise. This shows
 $$Q^k(f) \geq  (1_K f^2,V) - \varepsilon = \sum_{x \in K} f(x)^2 V(x) \mu(x) - \varepsilon.$$
 Since $K$ and $\varepsilon$ were arbitrary, this finishes the proof. 
\end{proof}

%
%
%
%
%

\subsection{Markov uniqueness} \label{subsection:markov uniqueness}

In this subsection we discuss criteria for the uniqueness of Markovian realizations of $\cH$, which by Theorem~\ref{theorem:structure of Markovian realizations} is equivalent to $Q^0 = Q^{(N)}$. Of course, the results of Section~\ref{section:uniqueness} can always be applied, but in general their assumptions are stronger than what is actually needed for Markov uniqueness.  There are various abstract characterization for $Q^0 = Q^{(N)}$ in the context of graphs, see e.g. \cite{HKLW,Schmi}. In this text the focus is  a bit different. We present characterizations of Markov uniqueness that are related to the geometry of the graph and therefore have the same spirit as the results in Section~\ref{section:uniqueness}.

The first criterion deals with capacities of boundaries. As a preparation we discuss some of their elementary properties. Let $V \geq 0$. A function $h:X \to [0,\infty)$ is called {\em $1$-excessive} (with respect to the form $Q^{(N)}$) if for every $\beta > 0$ we have $\beta(H^{(N)} + \beta + 1)^{-1} h \leq h$. Here the resolvent is extended to arbitrary nonnegative functions by monotonicity, i.e., 
$$(H^{(N)} + \beta  + 1)^{-1} h := \sup \{(H^{(N)} + \beta + 1)^{-1} f \mid f \in \ell^2(X,\mu) \text{ with } 0 \leq f\leq h\}.$$
Since the resolvent of $Q^{(N)}$ is Markovian, the constant function $1$ is always $1$-excessive. If $\mu(X) <\infty$  and $(V\cdot \mu)(X) < \infty$, it belongs to $D(Q^{(N)})$. We say that a function $f:X \to \R$ is {\em strictly positive} if $f(x) > 0$ for all $x \in X$. We need the following characterizations of excessive functions.

\begin{lemma}[Characterization of excessive functions]\label{lemma:characterization of excessive functions}
 Let $h:X \to [0,\infty)$. The following assertions are equivalent.
 \begin{enumerate}[(i)]
  \item $h$ is $1$-excessive.
  \item For every $f \in D(Q^{(N)})$ we have $f \wedge h \in D(Q^{(N)})$ and
  $$Q^{(N)}(f \wedge h) + \|f \wedge h\|^2_2 \leq Q^{(N)}(f) + \|f\|^2_2.$$
 \end{enumerate}
 If, additionally, $h \in D(Q^{(N)})$, then these   are equivalent to the following.
 \begin{enumerate}[(i)]  \setcounter{enumi}{2}
  \item $Q^{(N)}(h,f) + \langle h, f \rangle_2 \geq 0$ for all nonnegative $f \in D(Q^{(N)})$.
 \end{enumerate}
In particular, there always exists a strictly positive $1$-excessive function in $D(Q^{(N)})$. 
\end{lemma}
\begin{proof}
The equivalence of the assertions (i) - (iii) follows from \cite[Proposition~III.1.2]{MR} and \cite[Proposition~4]{Kaj}.

For the ``In particular''-part we let $g \in \ell^2(X,\mu)$ strictly positive and consider $h := (H^{(N)} + 1)^{-1} g$. By (iii) it is $1$-excessive. We show that it is strictly positive. Suppose that there exists an $x \in X$ with $h(x) = 0$. For $\beta > 0$ let $G_\beta := (H^{(N)} + 1)^{-1}$.  The resolvent identity and that $G_\beta$ is positivity preserving yield for $\beta > 1$ 
$$0 = h(x) = G_1 g(x) = (\beta -1) G_1 G_\beta g(x) + G_\beta g(x) \geq G_\beta g(x) \geq 0.  $$
Since $(G_\beta)$ is strongly continuous, we obtain
$$0 = \lim_{\beta \to \infty} \beta G_\beta g(x) = g(x) > 0,$$
a contradiction.
\end{proof}

Let now $h$ be a  $1$-excessive  function.  We define the {\em capacity} of a set $U \subseteq X$ (with respect to $h$ and $Q^{(N)})$ by
$$\mathrm{cap}_h(U) := \inf \{\|f\|_{Q^{(N)}}^2  \mid f \in D(Q^{(N)}) \text{ with } f \geq 1_U h\},$$
where $\|f\|_{Q^{(N)}}^2 = Q^{(N)}(f) + \|f\|_2^2$ is the square of the form norm. We use the convention $\mathrm{cap}_h(U) = \infty$ if the infimum is taken over an empty set.  If the convex set $\{f \in D(Q^{(N)}) \mid f \geq 1_U h\}$ is nonempty, by the Hilbert space projection theorem there exists a unique minimizer $h_U \geq 1_U h$ such that $\mathrm{cap}_h(U)  =Q^{(N)}(h_U) + \|h_U\|^2_2$. It is called the {\em equilibrium potential of $U$} (with respect to $h$). Since by Lemma~\ref{lemma:characterization of excessive functions} we have $\||h_U| \wedge h\|_{Q^{(N)}}^2 \leq \|h_U\|_{Q^{(N)}}^2$, it  satisfies $0 \leq h_U \leq h$. In particular, $h_U = h$ on $U$.

The {\em capacity of the boundary of X}  is defined by
$$\mathrm{cap}_h(\partial X) := \inf \{\mathrm{cap}_h(X \setminus K) \mid  K\subseteq X \text{ finite}  \}.$$
Let $\rho$ be a pseudo  metric. The {\em capacity of the Cauchy boundary}  is defined by
$$\mathrm{cap}_h(\partial_\rho X) := \inf \{\mathrm{cap}_h(X \cap O) \mid O \subseteq \overline{X}^\rho \text{ open neighborhood of } \partial_\rho X \}.$$
\begin{remark}
 Let  $\hat{X}$ be a compactification of $X$. For an open neighborhood $U$ of $\hat{X} \setminus X$ in $\hat X$ the set $\hat X \setminus U = X \setminus U$ is compact in $X$. Since $X$ carries the discrete topology, this means that $X \setminus U$ is finite. This implies that  $X \cap U = X \setminus K$ for some finite $K \subseteq X$.  On the other hand, if $K\subseteq X$ is finite, the set $\hat X \setminus K$ is an open neighborhood of $\hat X \setminus X$. These considerations show 
 $$\mathrm{cap}_h(\partial X) = \inf \{\mathrm{cap}_h(X \cap O) \mid O \subseteq \hat X \text{ open neighborhood of } \hat X \setminus X \}.$$
 Therefore, $\mathrm{cap}_h(\partial X)$ and $\mathrm{cap}_h(\partial_\rho X)$ have basically the same definition with the only difference that open neighborhoods of ``the boundary'' are defined by two different topologies in topological spaces containing $X$.
\end{remark}

The results of Subsection~\ref{subsection:a metric space criterion} could be read that absence of a boundary implies uniqueness of realizations. The following theorem is similar in spirit. It says that smallness of the boundary yields Markov uniqueness.

\begin{theorem}\label{theorem:capacity criterion boundary}
If $V \geq 0$, the following assertions are equivalent. 
\begin{enumerate}[(i)]
 \item $\cH$ has a unique Markovian realization.
 \item For any  $1$-excessive function $h \in D(Q^{(N)})$ we have $\mathrm{cap}_h(\partial X) = 0$. 
 \item There exists one strictly positive $1$-excessive function $h \in D(Q^{(N)})$ with $\mathrm{cap}_h(\partial X) = 0$. 
\end{enumerate}
If, additionally, $(X,b)$ is connected and locally finite, then these are equivalent to the following.
\begin{enumerate}[(i)]  \setcounter{enumi}{3}
 \item  For one/any strongly intrinsic path metric $\rho$  and for  any $1$-excessive function $h \in D(Q^{(N)})$ we have $\mathrm{cap}_h(\partial_{\rho} X) = 0$.
 \item  For one/any strongly intrinsic path metric $\rho$  and for  one strictly positive $1$-excessive function $h \in D(Q^{(N)})$ we have $\mathrm{cap}_h(\partial_{\rho} X) = 0$.
\end{enumerate}
\end{theorem}

We prove the theorem through several lemmas, which may be of interest on their own right. For some subset $U \subseteq X$ we denote by $\pi_U:\ell^2(X,\mu) \to \ell^2(U,\mu|_U)$ the restriction $f \mapsto f|_U$ and by $\iota_U = (\pi_U)^*:\ell^2(U,\mu|_U) \to \ell^2(X,\mu)$ the embedding that is given by $\iota_U f(x) = f(x)$ if $x \in U$ and $\iota_U f(x) = 0$ if $x \in X \setminus U$. Let $Q^{(N)}_U$ be the Dirichlet form on $\ell^2(U,\mu|_U)$ with domain $D(Q^{(N)}_U) = \{f \in \ell^2(U,\mu|_U) \mid \iota_U f \in D(Q^{(N)})\}$ on which it acts by $Q^{(N)}_U(f) :=   Q^{(N)}(\iota_U f)$. Moreover, let $Q^0_U$ be the restriction of $Q^{(N)}_U$ to the closure of $C_c(U)$ with respect to the form norm $\|\cdot\|_{Q^{(N)}_U}$. For $f \in D(Q^{(N)}_U)$ we have
$$Q^{(N)}_U(f) = \frac{1}{2} \sum_{x,y \in U} b(x,y) (f(x) - f(y))^2  + \sum_{x \in U} f(x)^2 (V(x) + d_U(x)) \mu(x),$$
with $d_U(x) = \mu(x)^{-1}\sum_{y \in X \setminus U} b(x,y)$. Hence, $Q^{(N)}_U$ and $Q^0_U$ are the maximal and the minimal Dirichlet form on $\ell^2(U,\mu|_U)$ associated with the graph $(U, b|_{U \times U})$ with potential $V|_U + d_U$. The associated operators are Markovian  realizations of the Schrödinger operator $\cH_{\mu|_U, V|_U + d_U}$ on the graph $b|_{U \times U}$.

\begin{lemma}\label{lemma:q0u = qnu}
 Let $(X,b)$ be a connected locally finite graph and let $\rho$ be a strongly intrinsic path metric. If $U \subseteq X$ is complete with respect to $\rho$, then $Q^0_U = Q^{(N)}_U$.
\end{lemma}
\begin{proof}
We first show that it suffices to consider the connected components of $U$.  Let  $(U_i)_{i \in I}$ the connected components of $U$. A simple computation using the aforementioned formula for $Q^{(N)}_U$ shows
$$Q^{(N)}_U(f) = \sum_{i \in I} Q^{(N)}_{U_i}(f|_{U_i}).$$
Hence, if for each $U_i$ we have $Q^{(N)}_{U_i} = Q^0_{U_i}$, then $Q^{(N)}_U = Q^0_U$.

Let now $\sigma$ be an edge weight such that  $\rho = \rho_\sigma$ and let  $W$ a connected component of $U$. By $\rho_{\sigma}^W$ we denote the path metric that $\sigma|_{W \times W}$ induces on $W$.   Since there are less paths in $W$ than in $X$, we have $\rho_\sigma^W \geq \rho_\sigma$ on $W \times W$. Moreover, since $\rho_\sigma$ is strongly intrinsic, the metric $\rho^W_\sigma$ is intrinsic with respect to $b|_{W \times W}$ and $\mu|_{W}$. 

{\em Claim:} $(W,\rho^W_\sigma)$ is complete. 

{\em Proof of the claim.} Let $(x_n)$ be a Cauchy sequence in $(W,\rho^W_\sigma)$. Since $ \rho_\sigma^W \geq \rho_\sigma$ and $W$ is complete in with respect to $\rho_\sigma$, it converges with respect to $ \rho_\sigma$ to some $x \in W$. According to Proposition~\ref{proposition:properties of path metrics} the singleton set $\{x\}$ is open in $(X,\rho_\sigma)$. Hence, $(x_n)$ is eventually constant so that it also converges in $(W,\rho^W_\sigma)$ to $x$. This proves the claim.

As discussed in Proposition~\ref{proposition:properties of path metrics} completeness of $(W,\rho^W_\sigma)$ implies that bounded sets in $(W,\rho^W_\sigma)$ are finite. Hence, by Corollary~\ref{corollary:completeness} the Schrödinger operator $\cH_{\mu|_W, V|_W + d_W}$ (with respect to the graph $b|_{W \times W}$) has a unique realization on $\ell^2(W,\mu|_W)$. The discussion preceding this theorem shows that this implies $Q^0_W = Q^{(N)}_W$.
\end{proof}

\begin{lemma}\label{lemma:alternative formula capacity}
 Let $h \in D(Q^{(N)})$ be $1$-excessive. For any $U \subseteq X$ we have
 $$\mathrm{cap}_h(U) = \inf \{\|h-f\|_{Q^{(N)}}^2 \mid f \in D(Q^{(N)}) \text{ with } f 1_U = 0 \}.$$
\end{lemma}

\begin{proof}
 Let $h_U$ be the equilibrium potential of $U$. Since $h_U = h$ on $U$, we have  $h - h_U = 0$ on $U$ so that
 $$\mathrm{cap}_h(U) = \|h - (h - h_U)\|_{Q^{(N)}}^2 \geq \inf \{\|h-f\|_{Q^{(N)}}^2 \mid f \in D(Q^{(N)}) \text{ with } f 1_U = 0 \}.$$
 Let now $\varepsilon > 0$ and choose $g \in D(Q^{(N)})$ with $g 1_U = 0$ such that 
 $$\|h-g\|_{Q^{(N)}}^2 \leq \inf \{\|h-f\|_{Q^{(N)}}^2 \mid f \in D(Q^{(N)}) \text{ with } f 1_U = 0 \} + \varepsilon.$$
 Since $h - g \geq h1_U$, this inequality and the definition of $\mathrm{cap}_h(U)$ imply
 $$\mathrm{cap}_h(U) \leq \inf \{\|h-f\|_{Q^{(N)}}^2 \mid f \in D(Q^{(N)}) \text{ with } f 1_U = 0 \} + \varepsilon.$$
 This finishes the proof.
\end{proof}

\begin{lemma}\label{lemma:equality of capacities}
 Let $(X,b)$ be a connected locally finite graph and let $\rho$ be a strongly intrinsic path metric. Let $h \in D(Q^{(N)})$ be a $1$-excessive function. Then 
 $$\mathrm{cap}_h(\partial X) = \mathrm{cap}_h(\partial_{\rho} X).$$
\end{lemma}
\begin{proof}
Since $\rho$ is a metric, points are closed. Hence, for any finite $K \subseteq X$ the set $\overline{X}^{\rho}\setminus K$ is an open neighborhood of $\partial_{\rho} X$. This implies $\mathrm{cap}_h(\partial X) \geq  \mathrm{cap}_h(\partial_{\rho} X)$.

Let $O$ an open neighborhood of $\partial_{\rho} X$ and let $h_{X\cap O}$ the corresponding equilibrium potential.  Since $h_{X \cap O} = h$ on $X \cap O$, the function $h - h_{X\cap O}$ is supported in $X \setminus O$. Moreover,  $h - h_{X\cap O} \in D(Q^{(N)})$ and  so the restriction of $h - h_{X\cap O}$ to $X \setminus O$ belongs to $D(Q^{(N)}_{X \setminus O})$.  The complement $X \setminus O$ is closed in the completion and therefore complete itself. Lemma~\ref{lemma:q0u = qnu} shows $Q^0_{X \setminus O} = Q^{(N)}_{X \setminus O}$.  Since $h - h_{X \cap O}$ is supported in $X \setminus O$, this yields that $h - h_{X \cap O}$ can be approximated with respect to $\|\cdot\|_{Q^{(N)}}$ by functions of finite support in $X \setminus O$.

Let now $\varepsilon > 0$ and let $\psi$ be a function with finite support $K \subseteq X\setminus O$ that satisfies
$ \|h-h_{X \cap O} - \psi\|_{Q^{(N)}}   < \varepsilon.$
Using Lemma~\ref{lemma:alternative formula capacity} and that $\psi 1_{X \setminus K} = 0$ we obtain 
\begin{align*}
 \mathrm{cap}_h(X \cap O)^{1/2} = \|h - (h-h_{X \cap O})\|_{Q^{(N)}} \geq  \|h - \psi\|_{Q^{(N)}} - \varepsilon
 \geq   \mathrm{cap}_h(X \setminus K)^{1/2} - \varepsilon.
 \end{align*}
This yields  $\mathrm{cap}_h(X \cap O) \geq  \mathrm{cap}_h(\partial X)$ and the claim is proven.
\end{proof}

\begin{lemma}\label{lemma:strictly positive excessive function in q0 implies q0 = qn}
 Let $h$ be a strictly positive $1$-excessive function. If $h \in D(Q^0)$, then $Q^0 = Q^{(N)}$.  
\end{lemma}
\begin{proof}
 Since $Q^{(N)}$ is a Dirichlet form, it suffices to prove that each nonnegative $f \in D(Q^{(N)})$ can be approximated by finitely supported functions with respect to the form norm.  Let $(\varphi_n)$ a sequence in $C_c(X)$ that converges to $h$ with respect to the form norm. Since $f$ is nonnegative, the functions $f_n := f \wedge \varphi_n$ have finite support and they converge in $\ell^2(X,\mu)$ to $f \wedge h$. The lower semicontinuity of $Q^0$ and that $Q^0$ and $Q^{(N)}$ agree on $C_c(X)$ yield
 $$\left(Q^0(f \wedge h) + \|f \wedge h\|_2^2\right)^{1/2} \leq \liminf_{n \to \infty} \|f_n\|_{Q^0}  = \liminf_{n \to \infty} \|f_n\|_{Q^{(N)}} \leq \|f\|_{Q^{(N)}} +  \|h\|_{Q^{(N)}} < \infty.  $$
 For the last inequality we use Lemma~\ref{lemma:maxima and minima energy inequality}. This shows $f \wedge h \in D(Q^0)$. For any $n \in \N$ the function $nh$ is also $1$-excessive and belongs to $D(Q^0)$. Hence, $f \wedge (n h) \in D(Q^0)$ for any $n \in \N$. With this at hand, the lower semi-continuity of $Q^0$ and that $Q^{(N)}$ and $Q^0$ agree on $D(Q^0)$ imply
 \begin{align*}
  Q^0(f) + \|f\|_2^2  &\leq \liminf_{n \to \infty} \left(Q^0(f \wedge (nh)) + \| f \wedge (nh) \|^2_2\right) \\
  &= \liminf_{n \to \infty} \left(Q^{(N)}(f \wedge (nh)) + \| f \wedge (nh) \|^2_2\right) \\
  &\leq Q^{(N)}(f) + \|f\|^2_2.
 \end{align*}
 For the last inequality we used that $nh$ is $1$-excessive and Lemma~\ref{lemma:characterization of excessive functions}. This shows $f \in D(Q^0)$ and finishes the proof.
\end{proof}

\begin{lemma}\label{lemma:characterizing h in q0}
 Let $h \in D(Q^{(N)})$ be a  $1$-excessive function. Then $h \in D(Q^0)$ if and only if  $ \mathrm{cap}_h(\partial X) = 0$.
\end{lemma}
\begin{proof}
   Let $h \in D(Q^{(N)})$ be $1$-excessive. Using Lemma~\ref{lemma:alternative formula capacity} we obtain 
\begin{align*}
 \mathrm{cap}_h(\partial X) &= \inf_{K \subseteq X\, \mathrm{ finite}}  \mathrm{cap}_h(X \setminus K)\\
 &= \inf_{K \subseteq X\, \mathrm{ finite}} \inf \{\|h - f\|^2_{Q^{(N)}} \mid f 1_{X \setminus K} = 0\}\\
 &= \inf \{\|h - f\|^2_{Q^{(N)}} \mid f  \in C_c(X)\}.
\end{align*}
This shows the claim.
\end{proof}

\begin{proof}[Proof of Theorem~\ref{theorem:capacity criterion boundary}]
According to Corollary~\ref{corollary:markov uniqueness} for proving the equivalence of (i) - (iii) it suffices to show that the assertions  (ii) and (iii) are equivalent to $Q^0 = Q^{(N)}$.

(i) $\Rightarrow$ (ii): Since $Q^0 = Q^{(N)}$, we have $h \in D(Q^0)$ and Lemma~\ref{lemma:characterizing h in q0} shows $\mathrm{cap}_h(\partial X) = 0$.

(ii) $\Rightarrow$ (iii): This follows from the existence of a strictly positive $1$-excessive function, see Lemma~\ref{lemma:characterization of excessive functions}.

(iii) $\Rightarrow$ (i):  Let $h \in D(Q^{(N)})$ be $1$-excessive and strictly positive. According to Lemma~\ref{lemma:characterizing h in q0} assertion (iii) implies $h \in D(Q^0)$. With this at hand $Q^0 = Q^{(N)}$ follows from Lemma~\ref{lemma:strictly positive excessive function in q0 implies q0 = qn}.

Assume now additionally that $(X,b)$ is connected and locally finite. The equivalence of the assertions (ii) and (iv) and of  the assertions (iii) and (v) follows from Lemma~\ref{lemma:equality of capacities} and the existence of strongly intrinsic path metrics, which was discussed in Subsection~\ref{subsection:a metric space criterion}.
\end{proof}

\begin{remark}
\begin{enumerate}
 \item  Under the condition $1 \in D(Q^{(N)})$, for locally finite graphs the equivalence of $Q^0 = Q^{(N)}$ and $\mathrm{cap}_1(\partial_\rho X) = 0$ for some strongly intrinsic path metric $\rho$ is contained in \cite{HKMW}.  The idea for the proof of Lemma~\ref{lemma:q0u = qnu} is also taken from \cite{HKMW}. The proofs of Lemma~\ref{lemma:alternative formula capacity}, Lemma~\ref{lemma:characterizing h in q0} and Lemma~\ref{lemma:strictly positive excessive function in q0 implies q0 = qn} are taken from \cite{Schmi3}, which contains abstract versions of these results and of the equivalence of (i) - (iii) in Theorem~\ref{theorem:capacity criterion boundary}. Lemma~\ref{lemma:equality of capacities} seems to be a new observation.
 \item If $1 \in D(Q^{(N)})$, the condition $\mathrm{cap}_1(\partial_\rho X)= 0$ can be inferred from estimates on the Minkowski dimension of  $\partial_\rho X$ with respect to the given measure $\mu$, see \cite[Theorem~4]{HKMW}.
 \item For proving $\mathrm{cap}_h(\partial X) = \mathrm{cap}_h(\partial_{\rho} X)$ it was essential that $(X,b)$ is locally finite and $\rho$ is a strongly intrinsic path metric. If $\rho$ is only an intrinsic metric that induces the discrete topology on $X$, we still have $\mathrm{cap}_h(\partial X) \geq \mathrm{cap}_h(\partial_{\rho} X)$ but we doubt that the converse inequality holds. Nevertheless, it could still happen that $\mathrm{cap}_h(\partial_{\rho} X) = 0$ implies  $\mathrm{cap}_h(\partial X)$. If this were the case, in the theorem we could drop the assumption that $\rho$ is a strongly intrinsic path metric.
\end{enumerate}

\end{remark}

In the case of finite measure it is possible to relate Markov uniqueness of $\cH$ to uniqueness of realizations of $\cH$ and the criteria that we gave in Subsection~\ref{subsection:a metric space criterion}. This is discussed next.   We denote by $\cD^0 = \cD_V^0$ the functions of finite energy $f \in \cD = \cD_V$ for which there exists a sequence $(\varphi_n)$ in $C_c(X)$ such that $\varphi_n \to f$ pointwise and $\cQ(f - \varphi_n) \to  0$, as $n \to \infty.$ In the following theorem we assume $V= 0$ for convenience. It would also be true with $(V \cdot \mu)(X) < \infty$.

\begin{theorem}\label{theorem:markov uniqueness 2}
  Suppose $V=0$.  The following assertions are equivalent. 
  \begin{enumerate}[(i)]
   \item There exists a finite measure $\mu$ such that $\cH_{\mu, 0}$ has a unique Markovian realization.
   \item For all finite measures $\mu$ the operator $\cH_{\mu, 0}$ has a unique Markovian realization.
   \item There exists a finite measure $\mu$ such that $\cH_{\mu, 0}$ has a unique realization.
   \item There exists a finite measure $\mu$ and an intrinsic metric $\rho$ with respect to $b$ and $\mu$ that has finite distance balls.
   \item $1 \in \mathcal{D}^0$.
   \item There exists a function   $f \in \mathcal{D}$ with $\lim\limits_{x\to \infty} f(x) = \infty.$
  \end{enumerate}
   If, additionally, (FC) holds, then these are equivalent to the following.
   \begin{enumerate}[(i)]\setcounter{enumi}{6}
    \item There exists a finite measure $\mu$ such that $H^\mathrm{min}_{\mu,0}$ is essentially self-adjoint.
  \end{enumerate}
\end{theorem}

To shed some light on the measures that are mentioned in (iii) and (vii) we single out the following lemma before proving the theorem.

\begin{lemma}\label{lemma:intrinsic metrics with finite balls}
 Let $f \in \cD$ injective with $\lim\limits_{x\to \infty} f(x) = \infty$. Then the metric $d_f(x,y): X \times X \to [0,\infty)$, $d_f(x,y) = |f(x) - f(y)|$ has finite distance balls. The measure $\mu_f$ that is given by  
 $$\mu_f(x) = \sum_{y \in X}b(x,y) |f(x) - f(y)|^2$$
 satisfies $\mu_f(X) \leq 2 \cQ(f) < \infty$ and $d_f$ is intrinsic with respect to any measure $\mu \geq \mu_f$.
\end{lemma}
 \begin{proof}
  According to Proposition~\ref{proposition:properties of metrics from embeddings} the metric $d_f$ has finite distance balls.   The other statements follow straightforward from the definitions.
 \end{proof}

\begin{proof}
We first prove the equivalence of (i), (ii) and (v).

 (ii) $\Rightarrow$ (i): This is trivial.
 
 (i) $\Rightarrow$ (v): Markov uniqueness, $V = 0$  and $\mu(X) < \infty$ imply $1 \in D(Q^{(N)}) = D(Q^0)$.  Moreover, it follows from the definitions that $D(Q^0) \subseteq \mathcal{D}^0 \cap \ell^2(X,\mu)$. This shows $1 \in \mathcal{D}^0$.
 
 (v) $\Rightarrow$ (ii): Let $\mu$ a finite measure on $X$. Since the constant function $1$ is $1$-excessive and belongs to $D(Q^{(N)})$, by Lemma~\ref{lemma:strictly positive excessive function in q0 implies q0 = qn} it suffices to prove $1 \in D(Q^0)$. Let $(\varphi_n)$ a sequence in $C_c(X)$ that converges pointwise to $1$ and satisfies $\cQ(\varphi_n - 1) = \cQ(\varphi_n) = \cQ^c(\varphi_n) \to 0$, as $n \to \infty$. Consider the sequence $\tilde \varphi_n := (\varphi_n \wedge 1) \vee 0$ in $C_c(X)$. It satisfies $\cQ^c(\tilde \varphi_n) \leq \cQ^c(\varphi_n)$ and, since $\mu$ is finite, it converges in $\ell^2(X,\mu)$ to $1$. The $\ell^2$-lower semicontinuity of $Q^0$ and that it is the closure of $\cQ^c$ yields
 $$Q^0(1) \leq \liminf_{n \to \infty} Q^0(\tilde \varphi_n) = \liminf_{n \to \infty} \cQ^c (\tilde \varphi_n) \leq \liminf_{n \to \infty} \cQ^c (\varphi_n) = 0 < \infty.$$
 This shows $1 \in D(Q^0)$.

 (iv) $\Rightarrow$ (iii)/(vii) : This follows from Corollary~\ref{corollary:completeness}. 
 
 (iii)/(vii) $\Rightarrow$ (i): This is trivial.
 
 (v) $\Rightarrow$ (vi): Let $\mu$ a finite measure. Let $(\varphi_n)$ be a sequence in $C_c(X)$ that converges pointwise to $1$ and satisfies $\cQ(1 - \varphi_n) = \cQ(\varphi_n)  \to 0$, as $n\to \infty$. As seen in the proof of (v) $\Rightarrow$ (ii), we can assume $0 \leq \varphi_n \leq 1$ so  that $\varphi_n \to 1$ in $\ell^2(X,\mu)$. We can further assume (after choosing a subsequence) that 
 $$\sum_{n = 1}^\infty \|1 - \varphi_n\|_{Q^{(N)}_{\mu,0}} < \infty.$$
 Hence, $f:= \sum_{n = 1}^\infty (1 - \varphi_n)$ exists in the Hilbert space $(D(Q^{(N)}_{\mu,0}),\|\cdot\|_{Q^{(N)}_{\mu,0}})$. For $K \in \N$ it satisfies $|f(x)| \geq K$ whenever $x \in X \setminus \cup_{n = 1}^K \mathrm{supp} (\varphi_n)$. This shows $\lim\limits_{x\to \infty} f(x) = \infty.$
 
 (vi) $\Rightarrow$ (iv): Let $f \in \cD$ with $\lim\limits_{x\to \infty} f(x) = \infty.$ By modifying the values of $f$ a little bit it can be chosen to be injective. With this at hand, the statement follows from Lemma~\ref{lemma:intrinsic metrics with finite balls}.
\end{proof}

\begin{remark}
\begin{enumerate}
 \item Let $V = 0$. Weighted  graphs $(X,b)$ that satisfy $1 \in \mathcal{D}^0$   are called recurrent. There is a vast amount of literature on abstract and geometric conditions ensuring recurrence. We refer to the textbooks \cite{Soa,Woe} as well as \cite{Schmi} for further details and references.
 
That recurrence ($1 \in \mathcal{D}^0$)  is equivalent to $Q^0_{\mu,0} = Q^{(N)}_{\mu,0}$ for finite measures $\mu$ is well known, see e.g. \cite{Schmi} for the discrete setting or \cite{Kuw,HKLMS} in the context of general Dirichlet forms.  Indeed, recurrence always implies Markov uniqueness for all measures.
 
 The equivalence of (iv), (v) and (vi) is taken from \cite{Puc}.
 \item Corollary~\ref{corollary:completeness} shows that the existence of intrinsic metrics with finite distance balls yields uniqueness of realizations. It is quite remarkable that the previous theorem gives some kind of converse. Uniqueness of Markovian realizations yields the existence of intrinsic metrics with finite distance balls and hence also uniqueness of realizations. However, we warn the reader that this is only true for particular measures.   In the proof we construct finite measures and metrics with finite distance balls that depend on functions as in (vi), cf. Lemma~\ref{lemma:intrinsic metrics with finite balls}. See also  Example~\ref{example:last example} below. 
\end{enumerate}
\end{remark}

We finish this section with an example of a graph where the previous theorem can be applied. Moreover, illustrate the construction of the finite measures that appear in assertion (iii)/(vii) of the theorem. 

\begin{example}\label{example:last example}
 Consider the graph $(\Z,b)$ of  Example~\ref{example:essential self adjointness on Z} and the measure $\nu_\alpha:\Z \to (0,\infty)$  with $\nu_\alpha(n) =   n^{-\alpha}$ if $n \neq 0$ and $\nu(0) = 1$. With the same arguments as in Example~\ref{example:essential self adjointness on Z} it is possible to prove that if $\alpha > 3$, the operator  $H^\mathrm{min}_{\nu_\alpha,0}$ is not essentially self-adjoint on $\ell^2(X,\nu_\alpha)$. Hence, according to Corollary~\ref{corollary:completeness}, for $\alpha > 3$ there cannot exist an intrinsic metric with finite distance balls with respect to $b$ and the measure $\nu_\alpha$.
 
 It is well known  that the graph $(\Z,b)$ is recurrent, i.e., $1 \in \cD^0$.  Indeed, for $1/2 < \alpha \leq 1$ the function $f_\alpha:\Z \to \R$ with $f_\alpha(0) = 1$ and 
 $$f_\alpha(n) = 1 + \mathrm{sgn} (n) \sum_{k = 1}^{|n|} \frac{1}{k^\alpha}, \quad n\neq 0,$$
 satisfies $f_\alpha \in \cD$ and $\lim\limits_{x\to \infty} f_\alpha(x) = \infty$. Therefore, the previous theorem shows $1 \in \cD^0$ and that for any finite measure we have Markov uniqueness. In particular, for all $\alpha > 1$  the operator $\cH_{\nu_\alpha,0}$ has a unique Markovian realization. For $\alpha \leq 1$ essential self-adjointness and Markov uniqueness can be inferred from Corollary~\ref{corollary:measure space criterion}.

  We discuss which finite measures and intrinsic metrics with finite balls are induced by $f_\alpha$, cf. Lemma~\ref{lemma:intrinsic metrics with finite balls}.  Let $g_\alpha = f_{\alpha/2}/\sqrt{2}$ and consider the measure
 $$\mu_{g_\alpha}(n)  =  \frac{1}{2} \sum_{k \in \Z} b(n,k)(f_{\alpha/2}(n) - f_{\alpha/2}(k))^2 = \begin{cases}
                                                                                            \frac{1}{2(|n| + 1)^{\alpha}} +  \frac{1}{2|n| ^{\alpha}} &\text{if } |n| \geq 1\\
                                                                                            1 &\text{if } n = 0
                                                                                           \end{cases}.
$$
 It satisfies $\mu_{g_\alpha} \leq \nu_{\alpha}$. Hence, the metric $d_{g_\alpha}$ that is given by
 $$d_{g_\alpha}(n,m) = |g_\alpha(n) - g_\alpha(m)| = \frac{1}{\sqrt{2}} \left|\mathrm{sgn} (n) \sum_{k = 1}^{|n|} \frac{1}{k^{\alpha/2}} - \mathrm{sgn} (m) \sum_{k = 1}^{|m|} \frac{1}{k^{\alpha/2}} \right| $$
 is intrinsic with respect to $\nu_\alpha$. For $\alpha \leq 2$ we have $\lim\limits_{x \to \infty}g_\alpha = \infty$  so that balls with respect to $d_{g_\alpha}$ are finite, see Proposition~\ref{proposition:properties of metrics from embeddings}. This shows that for $\alpha \leq 2$ the operator $\cH_{\nu_\alpha,0}$ is essentially self-adjoint on $\ell^2(X,\nu_\alpha)$, while for $\alpha > 1$ the measure $\nu_\alpha$ is finite.
 
 In the previous discussion we saw that  the operator $H^\mathrm{min}_{\nu_\alpha,0}$ is essentially self-adjoint for $\alpha \leq 2$ and not essentially self-adjoint for $\alpha > 3$.  It would be interesting to know what happens for $2 < \alpha \leq 3$. To the best of our knowledge, this is open. Note that $\cH_{\nu_\alpha,0}$ has a unique Markovian realization for all $\alpha \in \R$.
\end{example}

\section{Open problems} \label{section:open problems}

In this section we collect some of the open problems that arose in the text. For some problems we also comment on expected answers, turning them into conjectures.

 \begin{problem}
  Is $M^\mathrm{max}$ a closed operator on $\ell^2(X,\mu; E)$?
 \end{problem}
 
 If  (FC) is satisfied the maximal restriction   $M^\mathrm{max}$ is the adjoint of $M^\mathrm{min}$ and hence it is closed. For general graphs and finite measures however we expect the answer to be negative for two reasons. If we take a graph that is not locally finite and equip $C(X)$ with the locally convex topology of pointwise convergence, the operator $\cH_{\mu,0} :\cF \to C(X)$ is not closed as an unbounded operator on $C(X)$. For finite measures the spaces $\ell^2(X,\mu)$ and $C(X)$ and their topologies are not that different.

 \begin{problem}
  Does $W \in \cA_{\mu,\Phi; E}$ for all unitary connections $\Phi$ on $E$ imply $\Wm \in \cA_\mu$?
 \end{problem}

 At a first glance it seems unlikely that this is true, since by Proposition~\ref{proposition:semiboundedness magnetic forms and domination} $\Wm \in \cA_{\mu}$ implies a  lower bound for $ \cQ^c_{\Phi, W;\, E}$ that is uniform in the connection $\Phi$. On the other hand, it is not totally unlikely that some uniform boundedness principle yields a positive answer to the question.

 \begin{problem}
  Let $M$ be a realization of $\cM$. Is the associated quadratic form an extension of $\cQ^c_E$? In particular, does the existence of a realization of $\cH_{\mu,V}$ imply $V \in \cA_\mu$?
 \end{problem}

 We resolved this problem in the scalar case when the associated form satisfies the first Beurling-Deny criterion, see Lemma~\ref{lemma:q and extension of qc}. Moreover, by Proposition~\ref{proposition:realization finiteness condition} it is satisfied whenever (FC) holds. An answer to the ``in particular''-part of this problem would show or disprove the optimality of Theorem~\ref{theorem:existence of realizations for scalar operators}.
 
 \begin{problem}
  Let $(X,b)$ be a graph and let $\mu$ be a weight such that (FC) is not satisfied. Is there $W \in \cA_{\mu,\Phi; E}$ with $\Wm \not \in \cA_\mu$ such that $\cM_{\mu, \Phi,W; E}$ has / does not have a realization?
 \end{problem}

 This is a case that cannot be treated with domination arguments, cf. also Subsection~\ref{subsection:summary and examples}.

 \begin{problem}
  Is there a graph $(X,b)$ and a weight $\mu$ with (FC)  and $W \in \cA_{\mu, \Phi; E}$ such that $M^\mathrm{min}$ is not essentially self-adjoint but $\cM_{\mu,\Phi,W; E}$ has a unique realization?
 \end{problem}

 Realizations in the sense of Definition~\ref{definition:discrete magnetic Schrödinger operators} are always semi-bounded. Hence, uniqueness of realizations asks for uniqueness in the class of semi-bounded operators while essential self-adjointness asks for uniqueness in the class of all self-adjoint operators. For a general symmetric operator on some Hilbert space  it can of course happen that it has a unique semi-bounded extension while it is not essentially self-adjoint. Therefore, the previous questions asks whether such abstract examples can be realized within the class of magnetic Schrödinger operators. 
 
 \begin{problem} 
  Let $\rho$ be an intrinsic pseudo metric such that $X$ is closed in $\overline{X}^\rho$ and such that for all $\varepsilon > 0$ bounded subsets of $\{x \in X \mid D_\rho(x) \geq \varepsilon\}$ are finite.  What is the optimal constant $C \geq 0$ such that for all $W$ with $\Wm \geq C D_\rho^{-2} + V$ for some $V \in \cA_\mu$ the operator $\cM$ has a unique realization? 
 \end{problem}

Theorem~\ref{theorem:metric space uniqueness} shows that  the optimal constant is smaller than or equal to $\frac{1}{2}$.  For essential self-adjointness of  classical Schrödinger operators $-\Delta + V$  on open subsets of Euclidean space (with $D_\rho$   replaced by the distance to the topological boundary of the domain) the optimal constant is $\frac{3}{4}$, see e.g. \cite{NN} for the result in any dimension and \cite[Theorem~X.10]{RSII} for the optimality of the constant in one dimension. Even though this may make our result seem stronger than in the Euclidean case, it is not. Our discrete operator $\cH_{\mu,0}$ plays the same role as $-\frac{1}{2}\Delta$ (the generator of Brownian motion) in the Euclidean setting. For $\frac{1}{2}\Delta$ the optimal constant scales down to $\frac{3}{8}$, which is smaller than $\frac{1}{2}$.

 \begin{problem}
  For which $2 < \alpha \leq 3$ is the operator $H^\mathrm{min}_{\nu_\alpha,0}$ of Example~\ref{example:last example} essentially self-adjoint? 
 \end{problem}

 It was the general theme of the uniqueness sections that small measures tend to make discrete magnetic Schrödinger operators less unique. Even for the simplest infinite graph $\Z$ the role of the measure for essential self-adjointness is not fully understood. The previous problem could be a starting point for further investigations.

 \appendix
 
 \section{Quadratic forms on Hilbert spaces} \label{appendix:quadratic forms}

 In this appendix we collect basic facts about quadratic forms on Hilbert spaces. In particular, we treat their monotone limits and the Beurling-Deny criteria. The material presented here is standard. For the claims where we do not give detailed references, we refer the reader to the textbooks \cite{FOT,RSI}. 
 
 \subsection{Basics}\label{appendix:basics}

 Let $(H,\as{\cdot,\cdot})$ be a complex Hilbert space with induced norm $\|\cdot\|$. A  functional $q:H \to (-\infty,\infty]$ is called {\em quadratic form} if it is homogeneous and satisfies the parallelogram identity, i.e., if for all $f,g,h \in H$ and $\lambda \in \IC$ it satisfies
 $$q(\lambda f) = |\lambda|^2 q(f)$$
 and
 $$q(f+g) + q(f-g) = 2q(f) + 2q(g).$$
 The {\em domain} of a quadratic form $q$ on $H$ is $D(q) = \{f \in H \mid q(f) < \infty\}.$ It is called densely defined if the closure of $D(q)$ in $H$ equals $H$.
 
 By a theorem of Jordan and von Neumann \cite{JvN} any quadratic form induces a sesquilinear form on its domain via polarization, i.e.,
 $$q:D(q) \times D(q) \to \IC,\, (f,g) \mapsto q(f,g) := \frac{1}{4}\sum_{k = 1}^4 i^kq(f + i^k g)$$
 is sesquilinear. Here we abuse notation and write $q$ for the quadratic form on $H$ and the induced sesquilinear form on $D(q)$.  In this sense, we have $q(f,f) = q(f)$ for $f \in D(q)$. 
 
 A quadratic form $q'$ is called an {\em extension} of a quadratic form $q$ if $D(q) \subseteq D(q')$ and $q(f) = q'(f)$ for $f \in D(q)$. Moreover, we say that two quadratic forms $q,q'$ satisfy $q \leq q'$ if $D(q) \subseteq D(q')$ and $q(f) \geq q'(f)$ for all $f \in  D(q)$. Note that the order relation $\leq$ on quadratic forms compares the size of their domains.
 
 A quadratic form $q$ on $H$ is called {\em lower semi-bounded} or {\em semi-bounded from below} if there exists a constant $C \in \R$ such that  
 \begin{align}
 C \|f\|^2 \leq q(f) \text{ for all } f \in H.\label{inequality:lower bound} \tag{$\spadesuit$}
 \end{align}
 By the definition of $D(q)$ this can be replaced by the validity of the above inequality for  all $f \in D(q)$. If the quadratic form $q$ is semi-bounded from below, the largest possible constant $C \in \R$ for which  Inequality~\eqref{inequality:lower bound} holds is denoted by $\lambda_0(q).$ The quadratic form $q - \lambda_0(q) \|\cdot\|^2$ is a nonnegative quadratic form. Therefore, the induced sesquilinear form satisfies the Cauchy-Schwarz inequality. 
 
 For a lower semi-bounded quadratic form $q$ and $\alpha> 0$ the inner product
 $$\as{\cdot,\cdot}_q : D(q) \times D(q) \to \IC,\, (f,g) \mapsto \as{f,g}_q :=  q(f,g) + (\alpha - \lambda_0(q)) \as{f,g}$$
 is called {\em form inner product}. The induced norm on $D(q)$ is called the {\em form norm} and denoted by $\|\cdot\|_q$. Even though this definition depends on $\alpha > 0$, different $\alpha$ yield equivalent norms and this suffices for our purposes. Note that if $f_n \to  f$ with respect to the form norm, then also $f_n \to f$ in $H$, $q(f_n - f) \to 0$ and $q(f_n) \to q(f)$.

 A  densely defined lower semi-bounded quadratic form $q$ on $H$ is called {\em closed} if $(D(q),\as{\cdot,\cdot}_q)$ is a Hilbert space; it is called {\em closable} if it possesses a closed extension. Closability and closedness can be characterized by lower semicontinuity as follows, see e.g. \cite[Theorem~S.18]{RSI}.
 
 \begin{lemma}\label{lemma:characterization closedness}
  Let $q$ be a  densely defined lower semi-bounded quadratic form on $H$. The following assertions are equivalent.
  \begin{enumerate}[(i)]
   \item $q$ is closed.
   \item $q$ is lower semicontinuous, i.e., $f_n \to f$ in $H$ implies
   $$q(f) \leq \liminf_{n \to \infty} q(f_n).$$
   %
  \end{enumerate}
 \end{lemma}

 \begin{lemma}
  Let $q$ be a  densely defined lower semi-bounded quadratic form on $H$. The following assertions are equivalent.
  \begin{enumerate}[(i)]
   \item $q$ is closable.
   \item $q$ is lower semicontinuous on its domain, i.e., for all $f \in D(q)$ the convergence $f_n \to f$ in $H$ implies
   $$q(f) \leq \liminf_{n \to \infty} q(f_n).$$
  \end{enumerate}
 \end{lemma}
Let $q$ be a densely defined lower semi-bounded closed quadratic form on $H$. Then there is a   lower semi-bounded self-adjoint operator $L$   that is associated with $q$. It has the domain
$$D(L) = \{f \in D(q) \mid \text{ex. } g\in H \text{ s.t. } q(f,h) = \as{g,h} \text{ for all } h \in D(q)\}, $$
on which it acts by
$$Lf =g.$$
The value $\lambda_0(q)$ is the infimum of the spectrum of $L$, which is why we also write $\lambda_0(L)$ for it. The domain of $q$ satisfies $D(q) = D((L - \lambda_0(q))^{1/2})$. For $\alpha > -\lambda_0(q)$ we denote by $G_\alpha := (L + \alpha)^{-1}$ the resolvent of $L$. Moreover, for $\alpha > -\lambda_0(q)$ we define the {\em approximating form} $q^{(\alpha)}:H \to \R$ by
 $$q^{(\alpha)}(f) := \alpha \as{f - \alpha G_\alpha f, f}.$$
It follows from the spectral theorem that they are monotone increasing in the parameter $\alpha$. The resolvent and the approximating forms have the following properties. For details we refer to \cite[Section~1.3]{FOT}, where  nonnegative forms are treated. The proofs given there work also for lower semi-bounded forms.
 
\begin{enumerate}[(a)]
 \item  For all $\alpha > -\lambda_0(q)$ the operator $G_\alpha$ is bounded and self-adjoint.
 \item The family $(G_\alpha)_{\alpha > -\lambda_0(q)}$ satisfies the {\em resolvent identity}, i.e., for $\alpha,\beta > -\lambda_0(q)$ we have
 $$G_\alpha - G_\beta = (\beta-\alpha) G_\beta G_\alpha.$$
 \item For all $\alpha > -\lambda_0(q)$ we have $\|G_\alpha\| \leq (\alpha + \lambda_0(q))^{-1}$ and the family $(G_\alpha)_{\alpha > -\lambda_0(q)}$ is {\em strongly continuous}, i.e., for all $f \in H$ we have $\alpha G_\alpha f \to f$, as $\alpha \to \infty$. 
 \item  For all $f \in H$ we have $q(f) = \lim_{\alpha \to \infty} q^{(\alpha)}(f)$, where  it is possible that the limit takes the value $\infty$. In particular,
 $$D(q) = \{f \in H \mid \lim_{\alpha \to \infty} q^{(\alpha)}(f) < \infty\}.$$
\end{enumerate}
Conversely, suppose that  $(\widetilde{G}_\alpha)_{\alpha > C}$ is a family of operators that satisfies properties (a) - (c) as above with $-\lambda_0(q)$ replaced by the constant $C \in \R$. Such a family is called a {\em strongly continuous self-adjoint resolvent family} or simply {\em strongly continuous resolvent}. In this case the functional $\tilde{q}:H \to (-\infty,\infty]$ that is given by
$$\tilde{q}(f) := \lim_{\alpha \to \infty} \alpha \langle f - \alpha \widetilde{G}_\alpha f,f\rangle $$
is a densely defined lower semi-bounded closed quadratic form on $H$ with $C > -\lambda_0(\tilde q)$. For $\alpha  > C$ the operator $\widetilde{G}_\alpha$ coincides with the $\alpha$-resolvent of $\tilde{q}$.

\subsection{Monotone convergence of quadratic forms}

\begin{lemma}[Monotone convergence of quadratic forms]\label{lemma:monotone convergence of forms}
 Let $(q_n)$ be a sequence of densely defined lower semi-bounded closed quadratic forms on $H$ with associated resolvents $(G_\alpha^n)$  and assume that there exists $C \in \R$ such that $C \leq \lambda_0(q_n)$ for all $n \in \N$. If  $(q_n)$ is  monotone decreasing, i.e., $q_n(f) \geq q_{n+1}(f)$ for all $f \in H$ and $n \in \N$, then there exists a  densely defined lower semi-bounded closed quadratic form $q$ with resolvent $(G_\alpha)$ and with the following properties.
 \begin{enumerate}[(a)]
  \item $\lambda_0(q) \geq C.$
  \item For all $\alpha > 0$ we have $G_\alpha^n \to G_\alpha$ strongly, as $n \to \infty$.
  \item The convergence $f_n \to f$ weakly in $H$ implies 
  $$q(f) \leq \liminf_{n\to \infty} q_n(f_n).$$ 
  \item For every $f \in D(q)$ there exists a sequence $(f_n)$ with $f_n \in D(q_n)$, $f_n \to  f$ in $H$  and 
  $$\limsup_{n\to \infty} q_n(f_n) \leq q(f).$$
 \end{enumerate}
\end{lemma}
\begin{proof}
 Properties (c) and (d) say that the sequence $(q_n)$ Mosco converges towards $q$. This is well-known to be equivalent to (b), see e.g. \cite[Theorem~8.3]{CKK}, where a proof for nonnegative forms is given. The modifications for a sequence of forms that are lower semi-bounded with a uniform lower bound is straightforward.  

 (b) is contained in \cite[Theorem~S.16]{RSI} and (a) follows from (d) and  $\lambda_0(q_n) \geq C$.
\end{proof}

\begin{remark}
\begin{enumerate}[$\bullet$]
 \item  Properties (c) and (d) in the previous lemma mean that the sequence $(q_n)$ Mosco converges towards $q$. 
 \item Note that $q$ is not the pointwise limit of $(q_n)$. Even if $f \in D(q_n)$ for all $n \in \N$, the equation $q(f) = \lim_{n \to \infty} q_n(f)$ may fail. It is possible to prove that $q$ is the largest closed densely defined quadratic form that satisfies $q(f) \leq q_n(f)$ for all $f \in H$ and $n \in \N$, see e.g. \cite[Theorem~S.16]{RSI}.
\end{enumerate}
\end{remark}

\subsection{The Beurling-Deny criteria} \label{appendix:Beruling Deny}

In this subsection we consider quadratic forms on the Hilbert space $\ell^2(X,\mu)$, where $X$ is a countable set and $\mu$ is a weight on $X$. We discuss the compatibility of their resolvents with the order structure of this Hilbert space. 

A quadratic form $q$ on $\ell^2(X,\mu)$ is called {\em real} if for any real-valued $f,g \in D(q)$ it satisfies $q(f + ig) = q(f) + q(g)$. For studying a  real form $q$  it suffices to consider its restriction to $D(q)_r: =\{f:X \to \R \mid f \in D(q)\}$ on the real Hilbert space $\ell_r^2(X,\mu):=  \{f:X \to \R \mid f \in \ell^2(X,\mu)\}$.

We say that a lower semi-bounded closed quadratic form on $\ell^2(X,\mu)$ satisfies the {\em first Beurling-Deny criterion} if $f \in D(q)$ implies $|f| \in D(q)$ and $q(|f|) \leq q(f)$.

\begin{proposition}\label{propostion:beurling deny}
 Let $q$ be a  densely defined closed lower semi-bounded  quadratic form on $\ell^2(X,\mu).$ The following assertions are equivalent.
 \begin{enumerate}[(i)]
  \item $q$ satisfies the first Beurling-Deny criterion.
  \item The resolvent of $q$ is positivity preserving, i.e., for $\alpha > - \lambda_0(q)$ and $f \in \ell^2(X,\mu)$ the inequality $f \geq 0$ implies $G_\alpha f \geq 0$.
 \end{enumerate}
In particular, forms satisfying the first Beurling-Deny criterion are real.
\end{proposition}
\begin{proof}
  For nonnegative forms this is contained in \cite[Theorem~XIII.50]{RSIV}.  The proof given there can be carried out verbatim for lower semi-bounded forms. 
  
  Assertion (ii) implies that the resolvent maps $\ell^2_r(X,\mu)$ to $\ell^2_r(X,\mu)$. Using approximating forms this yields that forms satisfying the first Beurling-Deny criterion are real.
\end{proof}

\begin{lemma}\label{lemma:maxima and minima energy inequality}
 Let $q$ be a  densely defined closed lower semi-bounded  quadratic that satisfies the {\em first Beurling-Deny criterion}. For $f,g \in \ell^2_r(X,\mu)$ we have
 $$\|f \wedge g\|_q \leq \|f\|_q + \|g\|_q \text{ and } \|f \vee g\|_q \leq \|f\|_q + \|g\|_q.$$
\end{lemma}
\begin{proof}
 We have $f \wedge g = \frac{1}{2} (f + g - |f-g|)$ and  $f \vee g = \frac{1}{2} (f + g + |f-g|)$. Hence, the inequalities follow from the fact that $\|\cdot\|_q$ is a norm with $\||h|\|_q \leq \|h\|_q$ for all $h \in D(q)$.
\end{proof}

A function $C:\IC \to \IC$ is called a {\em normal contraction} if $C(0) = 0$ and $|C(x) - C(y)| \leq |x-y|$ for each $x,y \in \IC$. A nonnegative quadratic form satisfies the {\em second Beurling-Deny criterion} if for any normal contraction $C$ and $f \in D(q)$ we have $ C \circ f \in D(q)$ and $q(C \circ f) \leq q(f)$. Forms satisfying the second Beurling-Deny criterion are also called {\em Dirichlet forms}. They can be characterized as follows.

\begin{proposition}\label{proposition:secons beurling deny}
 Let $q$ be a densely defined closed nonnegative  quadratic form on $\ell^2(X,\mu).$ The following assertions are equivalent.
 \begin{enumerate}[(i)]
  \item $q$ satisfies the second Beurling-Deny criterion.
  \item $q$ satisfies the first Beurling-Deny criterion and for all  nonnegative $f \in D(q)$ we have $(f \wedge 1) \in D(q)$ and  $q((f \wedge 1)) \leq q(f)$.
  \item The resolvent of $q$ is Markovian, i.e., for $\alpha > 0$ and $f \in \ell^2(X,\mu)$ the inequality $0 \leq f \leq 1$ implies $0 \leq \alpha G_\alpha f \leq 1$.
 \end{enumerate}
\end{proposition}
\begin{proof}
 This is   contained in \cite[Theorem~XIII.51]{RSIV}.
\end{proof}

The following lemma applies to Dirichlet forms but it can also be used for forms that are not closed.
 
\begin{lemma}\label{lemma:postivity of form}
 Let $q$ be a nonnegative quadratic form on $\ell^2(X,\mu)$  such that for each nonnegative $f \in D(q)$ we have $f \wedge 1 \in D(q)$ and $q(f \wedge 1) \leq q(f)$. Let $f,g \in D(q)$ nonnegative with $g \leq 1$ and $g = 1$ on $\{f > 0\}$. Then $q(f,g) \geq 0$.
\end{lemma}
 \begin{proof}
  Let $\varepsilon > 0$. The properties of $f$ and $g$ imply
  $$q(g) = q((g + \varepsilon f) \wedge 1) \leq q(g + \varepsilon f) = q(g) + 2\varepsilon(f,g) + \varepsilon^2(g).$$
  Dividing by $\varepsilon$ and letting $\varepsilon \to 0+$ yields the claim.
 \end{proof}

\bibliographystyle{plain}
 
\bibliography{literatur}

\end{document}